\documentclass[11pt,reqno]{amsart}

\usepackage[utf8]{inputenc}
\usepackage[T1]{fontenc}
\usepackage{hyperref}
\usepackage{url}
\usepackage{booktabs}
\usepackage{amsfonts}
\usepackage{nicefrac}
\usepackage{microtype}
\usepackage{xcolor}
\usepackage{macro_wth}
\usepackage{comment}
\usepackage{autonum}
\usepackage{enumitem}
\setlist[enumerate]{nosep}
\setlist[itemize]{nosep}
\usepackage[margin=1in]{geometry}
\usepackage{parskip}
\usepackage{amsaddr}
\makeatletter
\renewcommand{\paragraph}{\@startsection{paragraph}{4}{\z@}%
  {3.25ex \@plus1ex \@minus.2ex}%
  {-1em}%
  {\normalfont\normalsize\bfseries}}
\makeatother
\usepackage[numbers]{natbib}
\usepackage{multicol}

\usepackage{lipsum}
\RequirePackage{etex}
\usepackage{ocgx2}

\renewcommand{\url}[1]{\href{#1}{#1}}
\newcommand{\doi}[1]{DOI: \url{https://doi.org/#1}}
\definecolor{burntorange}{RGB}{204,85,0}
\newcommand{\imageSize}{0.48}
\newcommand{\imageSizeTwo}{.3}

\hypersetup{
  colorlinks=true,
  linkcolor=burntorange,
  citecolor=burntorange,
  urlcolor=burntorange,
  filecolor=burntorange
}

\makeatletter
\def\@tocline#1#2#3#4#5#6#7{\relax
  \ifnum #1>\c@tocdepth % then omit
  \else
    \par \addpenalty\@secpenalty\addvspace{#2}%
    \begingroup \hyphenpenalty\@M
    \@ifempty{#4}{%
      \@tempdima\csname r@tocindent\number#1\endcsname\relax
    }{%
      \@tempdima#4\relax
    }%
    \parindent\z@ \leftskip#3\relax \advance\leftskip\@tempdima\relax
    \rightskip\@pnumwidth plus4em \parfillskip-\@pnumwidth
    #5\leavevmode\hskip-\@tempdima
      \ifcase #1
       \or\or \hskip 1em \or \hskip 2em \else \hskip 3em \fi%
      #6\nobreak\relax
    \hfill\hbox to\@pnumwidth{\@tocpagenum{#7}}\par% <---- \dotfill -> \hfill
    \nobreak
    \endgroup
  \fi}
\makeatother

\makeatletter
\let\contentsname\@empty
\makeatother

\newcommand{\1}{1 \!\! 1}

\def\vec{\mathrm{vec}}
\def\<{\langle}
\def\>{\rangle}
\newcommand{\Fro}{\mathrm{F}}
\renewcommand{\ss}[1]{_{#1}}
\renewcommand{\ss}[1]{[#1]}

\renewcommand{\[}{\begin{equation}}
\renewcommand{\]}{\end{equation}}
\mathcode`l="8000
\begingroup
\makeatletter
\lccode`\~=`\l
\DeclareMathSymbol{\lsb@l}{\mathalpha}{letters}{`l}
\lowercase{\gdef~{\ifnum\the\mathgroup=\m@ne \ell \else \lsb@l \fi}}%
\endgroup

\newcommand{\iid}{{\rm i.i.d.}\,}

\usepackage{relsize}

\newcommand{\+}{{+}}
\renewcommand{\-}{{-}}

\renewcommand{\=}{\ensuremath{=}}

\renewcommand{\bchi}{{\boldsymbol \chi}}

\newcommand{\model}{\cM_{\theta, a}}

\newcommand{\modelshort}{\cM_{a}}
\newcommand{\modelshortt}{\cM_{a_t}}
\newcommand{\gradient}{\Psi(h, y, a)}
\newcommand{\matGradient}[2]{\Psi\Big(#1, #2, a_t\Big)}

\usepackage{fvextra}
\usepackage{color}
\makeatletter
\def\PY@reset{\let\PY@it=\relax \let\PY@bf=\relax%
    \let\PY@ul=\relax \let\PY@tc=\relax%
    \let\PY@bc=\relax \let\PY@ff=\relax}
\def\PY@tok#1{\csname PY@tok@#1\endcsname}
\def\PY@toks#1+{\ifx\relax#1\empty\else%
    \PY@tok{#1}\expandafter\PY@toks\fi}
\def\PY@do#1{\PY@bc{\PY@tc{\PY@ul{%
    \PY@it{\PY@bf{\PY@ff{#1}}}}}}}
\def\PY#1#2{\PY@reset\PY@toks#1+\relax+\PY@do{#2}}

\@namedef{PY@tok@w}{\def\PY@tc##1{\textcolor[rgb]{0.73,0.73,0.73}{##1}}}
\@namedef{PY@tok@c}{\let\PY@it=\textit\def\PY@tc##1{\textcolor[rgb]{0.24,0.48,0.48}{##1}}}
\@namedef{PY@tok@cp}{\def\PY@tc##1{\textcolor[rgb]{0.61,0.40,0.00}{##1}}}
\@namedef{PY@tok@k}{\let\PY@bf=\textbf\def\PY@tc##1{\textcolor[rgb]{0.00,0.50,0.00}{##1}}}
\@namedef{PY@tok@kp}{\def\PY@tc##1{\textcolor[rgb]{0.00,0.50,0.00}{##1}}}
\@namedef{PY@tok@kt}{\def\PY@tc##1{\textcolor[rgb]{0.69,0.00,0.25}{##1}}}
\@namedef{PY@tok@o}{\def\PY@tc##1{\textcolor[rgb]{0.40,0.40,0.40}{##1}}}
\@namedef{PY@tok@ow}{\let\PY@bf=\textbf\def\PY@tc##1{\textcolor[rgb]{0.67,0.13,1.00}{##1}}}
\@namedef{PY@tok@nb}{\def\PY@tc##1{\textcolor[rgb]{0.00,0.50,0.00}{##1}}}
\@namedef{PY@tok@nf}{\def\PY@tc##1{\textcolor[rgb]{0.00,0.00,1.00}{##1}}}
\@namedef{PY@tok@nc}{\let\PY@bf=\textbf\def\PY@tc##1{\textcolor[rgb]{0.00,0.00,1.00}{##1}}}
\@namedef{PY@tok@nn}{\let\PY@bf=\textbf\def\PY@tc##1{\textcolor[rgb]{0.00,0.00,1.00}{##1}}}
\@namedef{PY@tok@ne}{\let\PY@bf=\textbf\def\PY@tc##1{\textcolor[rgb]{0.80,0.25,0.22}{##1}}}
\@namedef{PY@tok@nv}{\def\PY@tc##1{\textcolor[rgb]{0.10,0.09,0.49}{##1}}}
\@namedef{PY@tok@no}{\def\PY@tc##1{\textcolor[rgb]{0.53,0.00,0.00}{##1}}}
\@namedef{PY@tok@nl}{\def\PY@tc##1{\textcolor[rgb]{0.46,0.46,0.00}{##1}}}
\@namedef{PY@tok@ni}{\let\PY@bf=\textbf\def\PY@tc##1{\textcolor[rgb]{0.44,0.44,0.44}{##1}}}
\@namedef{PY@tok@na}{\def\PY@tc##1{\textcolor[rgb]{0.41,0.47,0.13}{##1}}}
\@namedef{PY@tok@nt}{\let\PY@bf=\textbf\def\PY@tc##1{\textcolor[rgb]{0.00,0.50,0.00}{##1}}}
\@namedef{PY@tok@nd}{\def\PY@tc##1{\textcolor[rgb]{0.67,0.13,1.00}{##1}}}
\@namedef{PY@tok@s}{\def\PY@tc##1{\textcolor[rgb]{0.73,0.13,0.13}{##1}}}
\@namedef{PY@tok@sd}{\let\PY@it=\textit\def\PY@tc##1{\textcolor[rgb]{0.73,0.13,0.13}{##1}}}
\@namedef{PY@tok@si}{\let\PY@bf=\textbf\def\PY@tc##1{\textcolor[rgb]{0.64,0.35,0.47}{##1}}}
\@namedef{PY@tok@se}{\let\PY@bf=\textbf\def\PY@tc##1{\textcolor[rgb]{0.67,0.36,0.12}{##1}}}
\@namedef{PY@tok@sr}{\def\PY@tc##1{\textcolor[rgb]{0.64,0.35,0.47}{##1}}}
\@namedef{PY@tok@ss}{\def\PY@tc##1{\textcolor[rgb]{0.10,0.09,0.49}{##1}}}
\@namedef{PY@tok@sx}{\def\PY@tc##1{\textcolor[rgb]{0.00,0.50,0.00}{##1}}}
\@namedef{PY@tok@m}{\def\PY@tc##1{\textcolor[rgb]{0.40,0.40,0.40}{##1}}}
\@namedef{PY@tok@gh}{\let\PY@bf=\textbf\def\PY@tc##1{\textcolor[rgb]{0.00,0.00,0.50}{##1}}}
\@namedef{PY@tok@gu}{\let\PY@bf=\textbf\def\PY@tc##1{\textcolor[rgb]{0.50,0.00,0.50}{##1}}}
\@namedef{PY@tok@gd}{\def\PY@tc##1{\textcolor[rgb]{0.63,0.00,0.00}{##1}}}
\@namedef{PY@tok@gi}{\def\PY@tc##1{\textcolor[rgb]{0.00,0.52,0.00}{##1}}}
\@namedef{PY@tok@gr}{\def\PY@tc##1{\textcolor[rgb]{0.89,0.00,0.00}{##1}}}
\@namedef{PY@tok@ge}{\let\PY@it=\textit}
\@namedef{PY@tok@gs}{\let\PY@bf=\textbf}
\@namedef{PY@tok@ges}{\let\PY@bf=\textbf\let\PY@it=\textit}
\@namedef{PY@tok@gp}{\let\PY@bf=\textbf\def\PY@tc##1{\textcolor[rgb]{0.00,0.00,0.50}{##1}}}
\@namedef{PY@tok@go}{\def\PY@tc##1{\textcolor[rgb]{0.44,0.44,0.44}{##1}}}
\@namedef{PY@tok@gt}{\def\PY@tc##1{\textcolor[rgb]{0.00,0.27,0.87}{##1}}}
\@namedef{PY@tok@err}{\def\PY@bc##1{{\setlength{\fboxsep}{\string -\fboxrule}\fcolorbox[rgb]{1.00,0.00,0.00}{1,1,1}{\strut ##1}}}}
\@namedef{PY@tok@kc}{\let\PY@bf=\textbf\def\PY@tc##1{\textcolor[rgb]{0.00,0.50,0.00}{##1}}}
\@namedef{PY@tok@kd}{\let\PY@bf=\textbf\def\PY@tc##1{\textcolor[rgb]{0.00,0.50,0.00}{##1}}}
\@namedef{PY@tok@kn}{\let\PY@bf=\textbf\def\PY@tc##1{\textcolor[rgb]{0.00,0.50,0.00}{##1}}}
\@namedef{PY@tok@kr}{\let\PY@bf=\textbf\def\PY@tc##1{\textcolor[rgb]{0.00,0.50,0.00}{##1}}}
\@namedef{PY@tok@bp}{\def\PY@tc##1{\textcolor[rgb]{0.00,0.50,0.00}{##1}}}
\@namedef{PY@tok@fm}{\def\PY@tc##1{\textcolor[rgb]{0.00,0.00,1.00}{##1}}}
\@namedef{PY@tok@vc}{\def\PY@tc##1{\textcolor[rgb]{0.10,0.09,0.49}{##1}}}
\@namedef{PY@tok@vg}{\def\PY@tc##1{\textcolor[rgb]{0.10,0.09,0.49}{##1}}}
\@namedef{PY@tok@vi}{\def\PY@tc##1{\textcolor[rgb]{0.10,0.09,0.49}{##1}}}
\@namedef{PY@tok@vm}{\def\PY@tc##1{\textcolor[rgb]{0.10,0.09,0.49}{##1}}}
\@namedef{PY@tok@sa}{\def\PY@tc##1{\textcolor[rgb]{0.73,0.13,0.13}{##1}}}
\@namedef{PY@tok@sb}{\def\PY@tc##1{\textcolor[rgb]{0.73,0.13,0.13}{##1}}}
\@namedef{PY@tok@sc}{\def\PY@tc##1{\textcolor[rgb]{0.73,0.13,0.13}{##1}}}
\@namedef{PY@tok@dl}{\def\PY@tc##1{\textcolor[rgb]{0.73,0.13,0.13}{##1}}}
\@namedef{PY@tok@s2}{\def\PY@tc##1{\textcolor[rgb]{0.73,0.13,0.13}{##1}}}
\@namedef{PY@tok@sh}{\def\PY@tc##1{\textcolor[rgb]{0.73,0.13,0.13}{##1}}}
\@namedef{PY@tok@s1}{\def\PY@tc##1{\textcolor[rgb]{0.73,0.13,0.13}{##1}}}
\@namedef{PY@tok@mb}{\def\PY@tc##1{\textcolor[rgb]{0.40,0.40,0.40}{##1}}}
\@namedef{PY@tok@mf}{\def\PY@tc##1{\textcolor[rgb]{0.40,0.40,0.40}{##1}}}
\@namedef{PY@tok@mh}{\def\PY@tc##1{\textcolor[rgb]{0.40,0.40,0.40}{##1}}}
\@namedef{PY@tok@mi}{\def\PY@tc##1{\textcolor[rgb]{0.40,0.40,0.40}{##1}}}
\@namedef{PY@tok@il}{\def\PY@tc##1{\textcolor[rgb]{0.40,0.40,0.40}{##1}}}
\@namedef{PY@tok@mo}{\def\PY@tc##1{\textcolor[rgb]{0.40,0.40,0.40}{##1}}}
\@namedef{PY@tok@ch}{\let\PY@it=\textit\def\PY@tc##1{\textcolor[rgb]{0.24,0.48,0.48}{##1}}}
\@namedef{PY@tok@cm}{\let\PY@it=\textit\def\PY@tc##1{\textcolor[rgb]{0.24,0.48,0.48}{##1}}}
\@namedef{PY@tok@cpf}{\let\PY@it=\textit\def\PY@tc##1{\textcolor[rgb]{0.24,0.48,0.48}{##1}}}
\@namedef{PY@tok@c1}{\let\PY@it=\textit\def\PY@tc##1{\textcolor[rgb]{0.24,0.48,0.48}{##1}}}
\@namedef{PY@tok@cs}{\let\PY@it=\textit\def\PY@tc##1{\textcolor[rgb]{0.24,0.48,0.48}{##1}}}

\def\PYZus{\char`\_}

\def\PYZlt{\char`\<}

\def\PYZhy{\char`\-}

\def\PYZdq{\char`\"}

\makeatother

\makeatletter
\@addtoreset{equation}{section}

\makeatother
\usepackage{etoolbox}
\allowdisplaybreaks

\newcommand{\testfunc}{{\bf test}}

\begin{document}

\title{Decoupled Descent: Exact Test Error Tracking Via Approximate Message Passing}

\author{Max Lovig}
\address{Statistics and Data Science, Yale University; \texttt{max.lovig@yale.edu}}

\begin{abstract}
In modern parametric model training, full-batch gradient descent (and its variants) suffers due to progressively stronger biasing towards the exact realization of training data; this drives the systematic ``generalization gap'', where the train error becomes an unreliable proxy for test error. Existing approaches either argue this gap is benign through complex analysis or sacrifice data to a validation set.
In contrast, we introduce decoupled descent (DD), a novel theory-based training algorithm that satisfies a train-test identity---enforcing the train error to asymptotically track the test error for stylized Gaussian mixture models. Within this specific regime, leveraging approximate message passing theory, DD iteratively cancels the biases due to data reuse, rigorously demonstrating the feasibility of zero-cost validation and $100\%$ data utilization. Moreover, DD is governed by a low-dimensional state evolution recursion, rendering the dynamics of the algorithm transparent and tractable. We validate DD on XOR classification, yielding superior performance compared to GD; additionally, we implement noisy MNIST and non-linear probing of CIFAR-10, demonstrating that even when our stylized assumptions are relaxed, DD narrows the generalization gap compared to GD.
\end{abstract}

\maketitle

\newpage
\vspace*{\fill}

{\centering C{\footnotesize ONTENTS}
\begin{multicols}{2}
\tableofcontents
\end{multicols}}

\vspace*{\fill}
\newpage

\section{Introduction}

Consider training data $(X,y) = (x_i, y_i)_{i \in [n]} \in \RR^{n \times d}$ where $x_i \in \RR^d$ and $y_i \in \RR$ are \iid $(x_i, y_i) \sim \PP_{x,y}$. We model this data using a parametric function $\cM_{\beta}: \RR^d
\to \RR$ with parameters $\beta \in \RR^p$ under the loss function $\cL(\cM_\beta(x), y): \RR
\times \RR \to \RR$. We learn $\beta$ by descending the objective
\[L(\beta) = \frac{1}{n} \sum_{i = 1}^n \cL(\cM_\beta(x_i), y_i),\label{eq:trainDesc}\]
using some variant of gradient descent (GD). The desired goal, however, is not to
directly minimize \eqref{eq:trainDesc} but to minimize the
\textit{test error}, i.e. with $(\check x,\check y) \sim \PP_{x,y}$, minimize 
\[\check L(\beta) = \EE_{\check x, \check y}\left[\cL(\cM_{\beta}(\check x), \check y)\right].\label{eq:testDesc}\]
In the classical regime of $n \to \infty$ with $d,p$ fixed, $L(\beta)$ is consistent for $\check L(\beta)$. 
Unfortunately, as seen below, in the modern regime of $n, d \to \infty$ with $n \asymp d \asymp p$, the convergence $L(\beta) \to \check L(\beta)$ fails.

\begin{example}
Consider high-dimensional linear regression, for $n$ \iid samples $x_i \sim \cN(0, \Id_d/d)$ with labels $y_i = 0$. We fit $\cM_\beta(x) = x^\top \beta$ for $\beta \in \RR^d$ under MSE loss $\cL(\hat y, y) = (\hat y - y)^2 = (x^\top \beta)^2$. Initializing $\beta_1$ independent of $(X,y)$ and setting $M = X^\top X$, the GD update $\beta_2 = (\Id - \eta X^\top X)\beta_1$ has train \eqref{eq:trainDesc} and test \eqref{eq:testDesc} errors,
\[L(\beta_2) = \beta_1^\top (\Id - \eta M) M (\Id - \eta M) \beta_1/n,\quad \check L(\beta_2) = \beta_1^\top (\Id - \eta M) (\Id - \eta M) \beta_1/d.\]  If $n \to \infty$ with $d$ bounded then $M/n \to \Id_{d}/d$ almost surely so $L(\beta_2) \to \check L(\beta_2)$; however, when $n,d \to \infty$ with $n \asymp d$, $M$ no longer concentrates to the identity and thus $L(\beta_2) - \check L(\beta_2) \not\to 0$.
\end{example}

Consequently, in modern scaling regimes, a decrease in the training objective for GD may not represent a decrease in test error. Previous work on GD must either justify that this mismatch is benign or sacrifice a portion of the training data to validate test error performance.
In this work, we construct a class of algorithms termed decoupled descent (DD). \textbf{For a set of stylized high-dimensional problems DD satisfies a train-test identity---any mismatch between \eqref{eq:trainDesc} and \eqref{eq:testDesc} is forbidden}. Consequently, this method allows for validation without sacrificing any training data.

\begin{theorem}[Informal version of Theorem \ref{thm:ddLoss}]\label{thm:introThm}
For $n$ \iid draws $(x_i)_{i \in [n]}$ from a $d$-dimensional Gaussian mixture with finitely many modes and responses $(y_i)_{i \in [n]}$, any DD iterate $\beta_t$ satisfies 
\[{\bf Training\;Error}(\beta_t) - {\bf Test\;Error}(\beta_t) \to 0, \quad \text{almost surely as $n,d \to \infty$ with $n/d \to \alpha$.}\]
\end{theorem}

As a benefit of Theorem \ref{thm:introThm}, DD enables \textbf{zero-cost validation}, meaning the training iterate with the lowest train error also has the lowest test error asymptotically. By eliminating the need for the validation set, users achieve 100\% data utilization. Moreover, techniques such as hyperparameter tuning and early stopping on the train error immediately translate performance improvements to the population-level test error.

DD originates from a carefully designed approximate message passing (AMP) algorithm \cite{donoho2009, feng2021,javanmard12}. Thus, DD admits a tractable low-dimensional law for the train and test error. This permits a simple physical interpretation for DD's dynamics and gives rigorous asymptotic guarantees on the behavior of the test error, enabling a principled way to implement descent algorithms. While modern practice uses stochastic gradients, eliminating the full-batch generalization gap is a fundamental first step.

To conclude this section, we outline our contributions below. 
\begin{itemize}
    \item We introduce a family of training algorithms, termed
    \textit{decoupled descent}, for a set of stylized learning problems with Gaussian mixture data, parametric models and general loss functions, covering a wide range of supervised learning tasks.
    \item In contrast to GD, where the
    training error is a biased proxy that worsens over time,
    we show that parameters trained by any DD algorithm enforce the train error  
    to asymptotically equal the test error at each iterate.
    \item We derive a sequence of low-dimensional recursions that track the
    algorithm's dynamical trajectory. This provides insight into how DD algorithms explicitly control the test error.
    \item We illustrate the effectiveness of DD using various applications:
\begin{itemize}
    \item \textbf{XOR Classification:} Validates DD under stylized model assumptions and demonstrates better test error performance compared to GD.
    \item \textbf{Noisy MNIST (0 vs. 8):} Demonstrates that DD is robust to different noise distributions and that the train-test identity is consistent across varying ratios of parameter count to data dimension. 
    \item \textbf{CIFAR-10 Probing:} Applies DD to (possibly whitened) ResNet-18 embeddings. Although the train-test identity is not exact, DD narrows the generalization gap compared to GD despite the lack of Gaussian structure.
\end{itemize}
    %\item We consider some applications: (1) To test DD within our model assumptions, we consider a simple simple XOR classification problem and show improvements over GD; (2) To demonstrate robustness to the underlying noise distribution and the ratio of parameters count to data dimension, we run a classification of zeros versus eights in a noisy MNIST dataset;
    %(3) To demonstrate how DD can be applied to problems without clear noise structure, we train a non-linear probing head on ResNet-18 embeddings for CIFAR-10. In this case the train-test identity does not exactly hold but reduces the generalization gap compared to GD, suitable preprocessing can reduce this gap further.
\end{itemize}

We provide an overview of the notation for this paper in Section \ref{sec:notation}.
\section{The Train-Test Identity And Its Algorithmic Consequences} \label{sec:trainTest}

\subsection{Correcting Full-batch GD}\label{sec:correctingGD}

Let the $i$-th row (or element) of $X \in \RR^{n \times d}$ and $y \in \RR^n$ have $(x_i, y_i) \sim \PP_{x,y}$ where $\PP_{x}$ is a Gaussian mixture with modes $\mu_1, \dots, \mu_J$; conditioned on mode $j$, we have $y_i \sim \PP_j$ independent of the Gaussian realization (see \eqref{eq:dist} for more details).
Consider parameter $\beta = (\theta, a)$ where $\theta \in \RR^{d \times L}$ and $a \in \RR^{L'}$ (with $
L, L'$ bounded) alongside a parametric function $\cM_{\theta, a}(x) = \cM_{a}(x^\top \theta)$ with $\cM_a: \RR^{L} \to \RR$. We use the following shorthand.

\begin{definition}\label{def:lossGrad}
    Given a parametric model $\model(x) = \modelshort(x^\top \theta) = \modelshort(h)$ and loss function
    $\cL$, let $\gradient = \cL(\modelshort(h), y, a)$ and further define its derivatives,
    $\nabla_h \gradient: \RR^{L} \times \RR \times \RR^{L'} \to \RR^L$,
    $\nabla_{a} \gradient: \RR^{L} \times \RR \times \RR^{L'} \to \RR^{L'}$, and Hessian
    $\nabla_h^2 \gradient: \RR^{L} \times \RR \times \RR^{L'} \to \RR^{L \times L}$.
Moreover, when $h \in \RR^{n \times L}$ and $y \in \RR^n$, each of the above functions is its corresponding row-wise application, for example $\Psi(h, y, a)_i = \Psi(h_{i}, y_i, a)$.
\end{definition}
We then run full-batch GD, depending on learning rate parameters $\eta, \gamma > 0$, by the iteration
    \[\label{eq:gd}
    h_t = X \theta_t, \quad \hat h_t = \nabla_h \matGradient{h_t}{y}, \quad \theta_{t+1} = \theta_t - \eta  X^\top \hat h_t, \quad a_{t+1} = a_t - \frac{\gamma}{n} \sum_{i=1}^n \nabla_{a}\Psi(h_t,
    y_i, a_t).\]
As explained in the introduction, this algorithm produces iterates $\theta_t, a_t$ where the training error does not track the test error. We correct this behavior using the following \textit{pure decoupled descent} iteration,
    \[\label{eq:pureDd}\begin{gathered} h_t = X\theta_t\+\eta\!\sum_{s=1}^{t-1}\!\hat h_{s}, \quad \hat h_t = \nabla_h\matGradient{h_t}{y},\\ 
    \hspace{-6pt}\theta_{t+1} = \theta_t\-\eta \Big(\!X^\top \hat h_t\-\alpha  \Big(\frac{1}{n} \!\sum_{i=1}^n
    \!\nabla_h^2\matGradient{h_{t,i}}{y_i} \!\Big)\theta_t \!\Big), \quad a_{t+1} = a_t\-\frac{\gamma}{n}\!\sum_{i=1}^n\!\nabla_{a}\!\matGradient{\!h_{t,i}}{y_i}. \end{gathered}\]

DD adds the correction terms $\eta \sum_{s=1}^{t-1} \hat
h_{s}$, $\eta \alpha \left( \frac{1}{n}\sum_{i=1}^n
\nabla_h^2 \matGradient{h_{t,i}}{y_i}\right)$ applied to the pre-activation $X\theta_t$ and gradient of $\theta_t$, respectively. These
terms cancel correlations due to reuse of data matrix $X$,
avoiding the train-test disconnect that GD suffered from in
the introduction. 

\subsection{The Train-Test Identity}
Next, we codify the notion of train and test error equality from Theorem \ref{thm:introThm}.
Consider the algorithm $\cA(X,y) \mapsto (h_1, \dots, h_T, \theta_1, \dots, \theta_T,
a_1, \dots, a_T)$ where, for each $t \in [T]$, $(h_t, \theta_t, a_t) \in \RR^{n \times L} \times \RR^{d \times L} \times \RR^{L'}$. 
When $\cA$ represents training for a parametric model, $h_t$ represents the pre-activations of our $n$ samples input into $\modelshortt$, depending on the parameters $a_t$. Then, $\theta_t$ represents the parameter we apply to new test examples at time $t$ of training. 
The algorithm $\cA$ satisfies the \textit{train-test identity} if the following holds:

Let $(X,y)$ have \iid rows $(x_i, y_i) \sim \PP_{x,y}$ specified in Section \ref{sec:correctingGD}, $(\check x, \check y) \sim \PP_{x,y}$ and $\cL$ be a loss function. For all $t \in [T]$, the following limit holds almost surely with respect to the data $(X,y)$,
\[\lim_{n,d \to \infty}\frac{1}{n} \sum_{i=1}^n \cL(\modelshortt(h_{t,i}), y_i) -
\EE_{\check x, \check y}[\cL(\modelshortt(\check x^\top \theta_t), \check y)] = 0.\label{eq:trainTestEquiv}\] 
 
Notably, the train-test identity does not require the loss to decrease. Assuming the train-test identity holds, Appendix \ref{sec:reduceTest} provides a method to monotonically decrease the test error.

\subsection{Natural Algorithmic Consequences}\label{sec:algoConq}

Algorithms satisfying the train-test identity (for example pure DD) can implement a ``zero-cost validation'' phenomenon. Standard practice saves an $\epsilon$ proportion of the data to estimate the test error with the average of the trained model's loss on the unseen validation set. This introduces a frustrating tradeoff where 
one wants $\epsilon$ to be small to maximize the data involved in training but not too small where the estimator for the test error becomes unreliable.
Thus, a zero-cost validation
method, i.e. when all data points are used directly in training, is desirable.
This suggests the train-test identity as an algorithmic principle for training parameters; some benefits of this principle are given below. 

\begin{itemize}
    \item \textbf{Early Stopping (Online):} Denoting $E_t \= \frac{1}{n}\sum_{i=1}^n\cL(\cM_{a_{t}}(h_{t,i}), y_i)$, let $D_{t} \= E_{t+1} - E_T$ and $\tilde D_{t} = \log(E_{t+1}) - \log(E_t)$. Stopping when $D_t \geq \epsilon$ or $\tilde D_t \leq \log(1-\epsilon)$ provides a certificate that the subsequent update will not improve the test error by the specified threshold. In particular, $\epsilon = 0$ guarantees the test error is non-increasing during training.
    \item \textbf{Early Stopping (Offline):} Save a subsequence of iterates $\cT \subseteq [T]$ and select the specific parameters $(\theta_{t^*}, a_{t^*})$ that minimize the train loss $E_t$ across the saved time steps.
    \item \textbf{Initializations/Hyperparameter Tuning:} Conduct many parallel runs with varying initializations or hyperparameters (e.g. $\eta, \gamma$) and select the configuration that minimizes the train error, and by the train-test identity, minimizes the test error.
    \item \textbf{Architecture Search:} Optimize model complexity (dimension $L$ or architecture $\mathcal{M}_a$) by selecting the run with the lowest train error; this equivalently identifies the optimal test error architecture. We expect an online method for architecture selection is also possible.
\end{itemize}

\section{Decoupled Descent}

%\subsection{Data Model, Parametric Functions and Training Protocals}

\paragraph{Data Model and Parametric Function} For bounded $J \in \NN$, signal vectors $\mu_1, \dots, \mu_J$, response laws $\PP_1, \dots, \PP_J$, and class probabilities $p_1, \dots, p_J$, consider the Gaussian mixture classification model,
\[(x_i, y_i) \sim \left(\cN\left(\frac{\mu_j}{d}, \frac{\Id_d}{d} \right) \otimes \PP_j\right) \text{ with probability }p_j.\label{eq:dist}\]
We then define the matrix-vector pair $X \in \RR^{n \times d}$, $y \in \RR^n$ where the $i$-th row of $X$ and $y$ are \iid draws $x_i$ and $y_i$. We consider the limit $n,d \to \infty$ with $n/d \to \alpha$. The data matrix
$X$ is commonly seen in mean field analysis of Gaussian mixture models
\cite{mignacco2021}. Specific statistical problems using data \eqref{eq:dist} are provided in Example \ref{ex:data}.
\begin{remark}\label{rm:lowDimData}
    This work requires that $J$ is bounded, we expect this can be relaxed to a $J \to \infty$
    limit after the limit $n,d \to \infty$. This naturally follows by designing finer and finer
    discrete distribution approximations to a low-dimensional data-generating process. Such a low-dimensional process is seen in Example \ref{ex:data} (3) where we could instead consider signal vectors $\mu \sim c v$ with $v \in \RR^d$ and $c$ is endowed some prior distribution supported on a compact interval. 
    Although the finite mixture assumption on the signal is stylized, it establishes a rigorous framework for isolating and correcting data-reuse bias in high-dimensional dynamics.
\end{remark}
For parameters $\theta \in \RR^{d \times L}$ and $a \in \RR^{L'}$, with $L, L'$ bounded, define $\model(x): \RR^d \to \RR$ with $\model(x) = \modelshort(x^\top \theta)$,
where $\modelshort: \RR^L \to \RR$ uses parameter $a$.
To evaluate the performance of this model, we consider a loss
function $\cL(\hat y, y): \RR \times \RR \to \RR$. Specific parametric model and loss pairs following this format are given in Example \ref{ex:mlp}.

\paragraph{Training Protocol} Now, we introduce the general family of DD algorithms. 

\begin{definition}
    Let functions $g: \RR^{L} \times \RR \times \RR^{L'} \to \RR^L$, $f: \RR^{L} \times \RR \times \RR^{L'} \to \RR^{L'}$ and hyperparameters $\eta_0, \eta_1, \gamma_0, \gamma_1$ parameterize \textit{decoupled descent}. When $h \in \RR^{n \times L}$ and $y \in \RR^n$, we denote $g(h,y,a)_i = g(h_i, y_i, a)$ for each $i \in [n]$ as the row wise application of $g$; consider the iteration,
\[\label{eq:dd}\begin{gathered}
    h_t = X \theta_t + \eta_{1} \sum_{s=1}^{t-1} \eta_{0}^{(t-1)-s} \hat h_{s}, \quad \hat h_t = g(h_t, y, a_t)\\
    \tilde \theta_t = X^\top \hat h_t - \alpha  \left(\frac{1}{n} \sum_{i=1}^n
    \nabla_h g(h_{t,i}, y_i, a_t) \right)\theta_t\\ 
    \theta_{t+1} = \eta_{0} \theta_t - \eta_{1} \tilde \theta_t, \quad a_{t+1} = \gamma_{0} a_t - \gamma_{1} \frac{1}{n} \sum_{i=1}^n f(h_{t,i}, y_i, a_t). \end{gathered}\]
\end{definition}

To provide a streamlined analysis of DD algorithms, we make strong 
assumptions on the above setting. We describe them informally below, see Appendix \ref{sec:assumption} for the formal statement.

\begin{assumption}[Informal Version Of Assumption \ref{as:main0}]\label{as:main00}Assume that:

\begin{enumerate}
    \item The rows of $(X,y)$ are $\iid$\! from distribution \eqref{eq:dist} where $n,d \to \infty$ with $n/d \to \alpha \in (0, \infty)$. Moreover, we assume that $d^{-1} \mu_j^\top \mu_k \to \chi_{j,k}$ almost surely as $n,d \to \infty$.
    \item Initializations $\theta_1\!, a_1$ are data-independent with bounded limiting norms and signal alignment.
    \item Uniformly over $y$, the functions $f,g$ (and $\nabla_h f, \nabla_h g$) are Lipschitz and bounded in $h$ and $a$.
    \item $\Psi$ is $(C^4, C^2) \cap {\rm Lip}$ in $(h,a)$ with bounded expected derivatives under distribution \eqref{eq:dist}.
\end{enumerate}
\end{assumption}

\begin{remark}[Universality and more general activation functions]\label{rm:universality}
As DD is a designed AMP algorithm, we conjecture our results are robust to relaxations 
of Assumption \ref{as:main00}. Indeed, AMP is known to exhibit \textit{universality} and we empirically demonstrate this robustness in Section \ref{sec:experiments}, for example:
    \begin{enumerate}
        \item Previous works allow one to replace the Gaussian noise in \eqref{eq:dist} with mean zero, variance $1/d$, independent (but not identical) sub-Gaussian noise \cite{wang2024,chen2020,bayati2015}. We expect our results to hold under this change in noise.
        \item Previous works relax the bounded and Lipschitz requirement for 
        the activation/loss functions $f, g$ and
        $\Psi$ \cite{lovig2025,reeves2025, dandi2025}, although truncating a 
        desired function and considering a bounded Lipschitz extension can give a suitable analysis. 
    \end{enumerate}
\end{remark}

\subsection{Main Result: The State Evolution Of Decoupled Descent}

We begin by providing an asymptotically representation of the test error.
\begin{definition}\label{def:testErr}
    Given $m_{j,\theta} \in \RR^{L}$ for $j \in [J]$, $\Omega_\theta \in \RR^{L \times L}$ and $\bar a \in \RR^{L'}$, define 
    \[\testfunc(m_{1,\theta}, \dots, m_{J, \theta}, \Omega_{\theta}, \bar a) = \sum_{j=1}^J p_j \EE_{\substack{\check Z_\theta \sim \cN\left(0, \Omega_\theta \right)\\\check Y_j \sim \PP_j}}[\cL(\cM_{\bar a}(m_{j,\theta} + \check Z_\theta), \check Y_j)].\]
\end{definition}

It is easy to see that for trained parameters $\theta, a$ with almost sure limits $d^{-1} \mu_j^\top \theta \to m_{j, \theta}$, $d^{-1} \theta^\top \theta \to \Omega_\theta$, $a \to \bar a$, we have that $\lim_{d \to \infty} \EE_{\check x, \check y}[\cL(\cM_{\theta, a}(\check x), \check y)] = \testfunc(m_{1,\theta}, \dots, m_{J, \theta}, \Omega_\theta, \bar a)$ almost surely. This is proven in Appendix \ref{sec:asyTestError}.

We now present the main technical lemma of this work, the state evolution of decoupled descent. For simplicity, we defer the full low-dimensional system of recursive equations describing the state evolution parameters to Appendix \ref{sec:fullSe}.

\begin{lemma}\label{lem:ddState}
If Assumption \ref{as:main00} holds and $\phi: \RR^{L} \times \RR \times \RR^{L'} \to \RR, f, g$ are suitably regular (see Assumption \ref{as:phi}), then almost surely for $(h_t, a_t, \theta_t)_{t \in [T]}$ from~\eqref{eq:dd},
    \[\lim_{n \to \infty} \frac{1}{n} \sum_{i=1}^n \phi(h_{t,i}, y_i, a_t) =
    \sum_{j=1}^J p_{j} \EE[\phi(m_{j,t} + G^t, Y_j, \bar a_t)], \quad
    G^t \sim \cN(0, \Omega_t\ss{t,t}), \quad Y_j \sim \PP_j,\] where $\Omega_t, \Sigma_t, \Xi_t, m_{j,t}, \bar a_t$ are given by the following almost sure limits as $n,d \to \infty$ with $n/d \to \alpha$,
    \[\frac{1}{d}\mu_{j}^\top \theta_t \to m_{j,t},\quad \frac{1}{d} \theta_t^\top \theta_t \to \Omega_t\ss{t,t},\quad \frac{1}{d} \theta_t^\top \tilde \theta_t \to \Xi_t\ss{t,t},\quad \frac{1}{d} \tilde \theta_t^\top \tilde \theta_t \to \alpha  \Sigma_t\ss{t,t},\quad a_t \to \bar a_t.\] 
\end{lemma}

\begin{remark}[Finite Sample Guarantees]\label{rm:finite}
    Again, basing our descent algorithm in approximate message passing has the
    benefits of appealing to prior AMP literature, specifically one can endow the
    above lemma with finite sample guarantees using works \cite{rush2018,li2023,reeves2025,bao2025}.
\end{remark}

The proof of Lemma \ref{lem:ddState} follows by a change of variables to traditional AMP algorithms and is deferred to Appendix \ref{sec:AMP}.
As a consequence, we have
the following theorem that confirms DD satisfies the train-test identity.

\begin{theorem}\label{thm:ddLoss} Consider the state evolution variables in
    Definition~\ref{lem:ddState} and $\testfunc$ from Definition \ref{def:testErr}. If Assumption \ref{as:main00} holds and $g,f$ are suitably regular (Assumption \ref{as:phi}), the following holds almost surely
    for the iterates $(h_t, a_t, \theta_t)_{t \in [T]}$ from \eqref{eq:dd} and $(\check x, \check y)$ are drawn from distribution \eqref{eq:dist},
    \[\lim_{n \to \infty} \frac{1}{n} \sum_{i=1}^n \cL(\cM_{a_t}(h_{t,i}), y_i) = \testfunc(m_{1,t}, \dots, m_{J,t}, \Sigma_t\ss{t,t}, \bar a_t),\] 
    as a consequence, the train-test identity holds for DD, i.e. almost surely,
    \[\lim_{n,d \to \infty} \frac{1}{n} \sum_{i=1}^n \cL(\cM_{a_t}(h_{t,i}), y_i) - \lim_{d
    \to \infty}\EE_{\check x, \check y}[\cL(\cM_{a_t}(\check x^\top \theta), \check y)] = 0.\]
\end{theorem}
The proof of this statement is deferred to Appendix \ref{sec:proofThm}. 
We sketch the proof of Theorem \ref{thm:ddLoss} below. 

\textbf{Proof Sketch}: We rewrite DD, via a change of variables, as a non-separable matrix-valued AMP algorithm \cite{berthier2017, lovig2025}. This establishes that the coordinate-wise average of $\phi$ applied to $h_t$ is governed by a low-dimensional set of recursive equations (Appendix \ref{sec:fullSe}).
Choosing $\phi = \Psi$, this average (with respect to iterate $\theta_t$) only depends on the almost sure limits of $d^{-1} \mu_j^\top \theta_t$ and $d^{-1} \theta_t^\top \theta_t$. These moments evolve identically as if we trained on the data $X^{(t)} = S + Z^{(t)}$ at time $t$ where $Z^{(t)}$ is a fresh Gaussian matrix and $S$ has rows independently assigned $\mu_j/d$ with probability $p_j$.
Due to this behavior, the training error must be the test error since DD essentially evaluates on a fresh independent data set at each update.

\begin{remark}
    A dynamical mean field theory (DMFT) analysis of overfitting in full-batch 
    GD is given in \cite[Equations (C.7)-(C.10)]{montanari2025}. 
    This description contains complex ``response function'' terms represented
    by the integration of a two-time correlation matrix which accounts for the 
    reuse of the data matrix. By implementing DD with
    AMP, however, our algorithm is self correcting in the sense that the
    response function is zero, admitting a simpler analysis than DMFT methods.
    This simplicity also allows for better physical interpretation and insight 
    into algorithmic design.
\end{remark}

\section{Algorithm Design}

The choice of $g,f$ in algorithm \eqref{eq:dd} controls the state evolution variables, and in turn, the asymptotic test error $\testfunc_t = \testfunc(m_{1,t}, \dots, m_{J,t}, \Omega_T\ss{t,t}, \bar a_t)$ by Theorem \ref{thm:ddLoss}. This means that given $g,f$ we can calculate $\testfunc_{t}$ exactly by simulations. Thus, given candidate function classes $\cG$ and $\cF$, with a budgeted run-time $T$, the optimal $g^*$, $f^*$ is found by running $|\cF| |\cG|$ total simulations and selecting the run that minimized $\min_{t \in [T]} \testfunc_t$.
Unfortunately, to run these simulations, we require exact knowledge of the data generating processes \eqref{eq:dist}. Therefore, it is desirable to select $g,f$ which produce low test error independent of the data generating process.

\subsection{Pure Decoupled Descent}

Using \eqref{eq:dd}, we recover pure DD from Section \ref{sec:trainTest} by selecting $g = \nabla_h \Psi$, $f = \nabla_{a}\Psi$, $\eta_{0} = \gamma_{0} = 1$ and $\eta_{1} = \gamma_1 = \eta$. See \eqref{eq:pureDd} for the exact iteration and Appendix \ref{sec:pureDdSe} for the corresponding state evolution.

Traditionally in the analysis of GD, we consider a Taylor expansion on the train error when the learning rate $\eta \to 0$.
Because pure DD satisfies the train-test identity, in contrast to GD, we can directly expand the asymptotic test error.
With $G^t$ defined by the state evolution in Appendix \ref{sec:pureDdSe} and $Y_j \sim \PP_j$ from \eqref{eq:dist}, let
\[
    \bG_t = \Big(\EE[\nabla_{a} \Psi(m_{j,t} + G^t, Y_j, \bar a_t)]\Big)_{j \in [J]}, \qquad \bU_t = \Big(\EE[\nabla_{h}\Psi(m_{j,t} + G^t,Y_j, \bar a_t)]\Big)_{j \in [J]},
\]
and denote $\bchi = (\chi_{j,k})_{j,k \in [J]}$ and $\bp = (p_{j})_{j \in [J]}$ such that $(\diag(\bp) \bchi \diag(\bp))_{j,k} = p_{j} \chi_{j,k} p_{k}$. 

\begin{theorem}\label{thm:pureTaylor}
    Let pure DD satisfy the conditions of Assumption \ref{as:main00} and Assumption \ref{as:phi} for sufficiently small $\eta$ up to bounded time $T \in \NN$. As $\eta \to 0$, $\testfunc_{t+1}$ satisfies the following Taylor expansion with state evolution parameters corresponding to the state evolution recursion \eqref{eq:pureDdState},
    \begin{align}
    \testfunc_{t+1} &= \testfunc_t - \eta \|\bG_t \bp\|_2^2 - \eta \alpha \|\bU_t ({\rm diag}(\bp) \bchi {\rm diag}(\bp))^{1/2}\|_{\Fro}^2 \label{eq:TaylorSignal}\\ 
    &\quad - \frac{\eta}{2} \sum_{j=1}^J p_j
    \Big\langle \EE[\nabla_h^2 \Psi(m_{j,t} + G^t, Y_j, \bar a_t)],
    \Xi_t\ss{t,t} + \Xi_t\ss{t,t}^\top \Big\rangle + \epsilon_{t},\label{eq:TaylorError}
    \end{align}
    where $\sup_{t \in [T]}|\epsilon_t| \leq C\eta^2$ as $\eta \to 0$ (where $C$ depends on constants in Assumption \ref{as:phi}).
\end{theorem}

Notice the terms in \eqref{eq:TaylorSignal} are non-positive, meaning that they directly lead to a decrease in the test error. The remaining term in \eqref{eq:TaylorError}, however, may be positive. We describe methods to bound the impact of this term in Remark \ref{rm:TaylorError}. 
The proof of Theorem \ref{thm:pureTaylor} is deferred to Appendix \ref{sec:design} alongside a general Taylor expansion for generic choices of $g,f$ and learning rate parameters. To give insight on the global convergence properties of DD, we describe fixed points of DD algorithms and compare them to critical points of the test error in Appendix \ref{sec:fixedPoints}.

\section{Applications}\label{sec:experiments}

Below we present three applications of decoupled descent, a fourth application on the signal-less regression problem from the introduction is deferred to Appendix \ref{sec:signalless}. Experiments were conducted on a M4 Mac mini (2024, 10-core CPU, 10-core GPU) with 16GB of unified memory.

\subsection{Improved Training On The XOR Model}\label{sec:XOR}

We consider a high-dimensional variant of the XOR problem ($J = 4$, $p_j = 1/4$) with $n$ \iid data points with the following signal vectors dependent on fixed $v\in \RR^{d/2}$, $\mu_{1} = [v, v], \mu_{2} = [-v, -v], \mu_{3} = [-v, v], \mu_{4} = [v, -v]$, 
where $y_j = 0$ for $j \in \{1,2\}$ and $y_j = 1$ for $j \in \{3,4\}$.

We fit model $\cM_{a, \theta^1, \theta^2}(x) = \sigma(a (x^\top \theta^1) (x^\top \theta^2))$, where $\theta^1 \in
\RR^{d}$, $\theta^2 \in \RR^{d}$, $a \in \RR$ and $\sigma(x)  = \frac{1}{1 +
e^{-x}}$, under cross entropy loss (i.e. $\cL(\hat y,y) = -y \log(\hat y) - (1-y) \log(1-\hat y)$).
These parameters are initialized at $\theta_1^1 \sim \cN(0, \Id_d)$, $\theta_1^2 \sim \cN(0, \Id_d)$ and $a_1 \sim \cN(0,1)$ independently. Both GD and DD iterations for this problem are given in Appendix \ref{sec:detailXor}.

Figure \ref{fig:XorClassification} plots the train and test errors for GD and DD when each $v_i = \lambda$ with $\lambda \in \{1,4,8\}$ as a signal to noise ratio (SNR) parameter. In all cases, DD maintains the train-test identity. In low SNR regimes $(\lambda \in\{1,4\})$, DD outperforms GD and maintains the train-test identity while GD rapidly overfits; for high SNR regimes, DD matches GD's performance because, under the $\lambda \to \infty$ limit, the signal dominates the noise and the effect of memorization is negligible.

\begin{figure}
    \centering
    \includegraphics[width = \imageSize\linewidth]{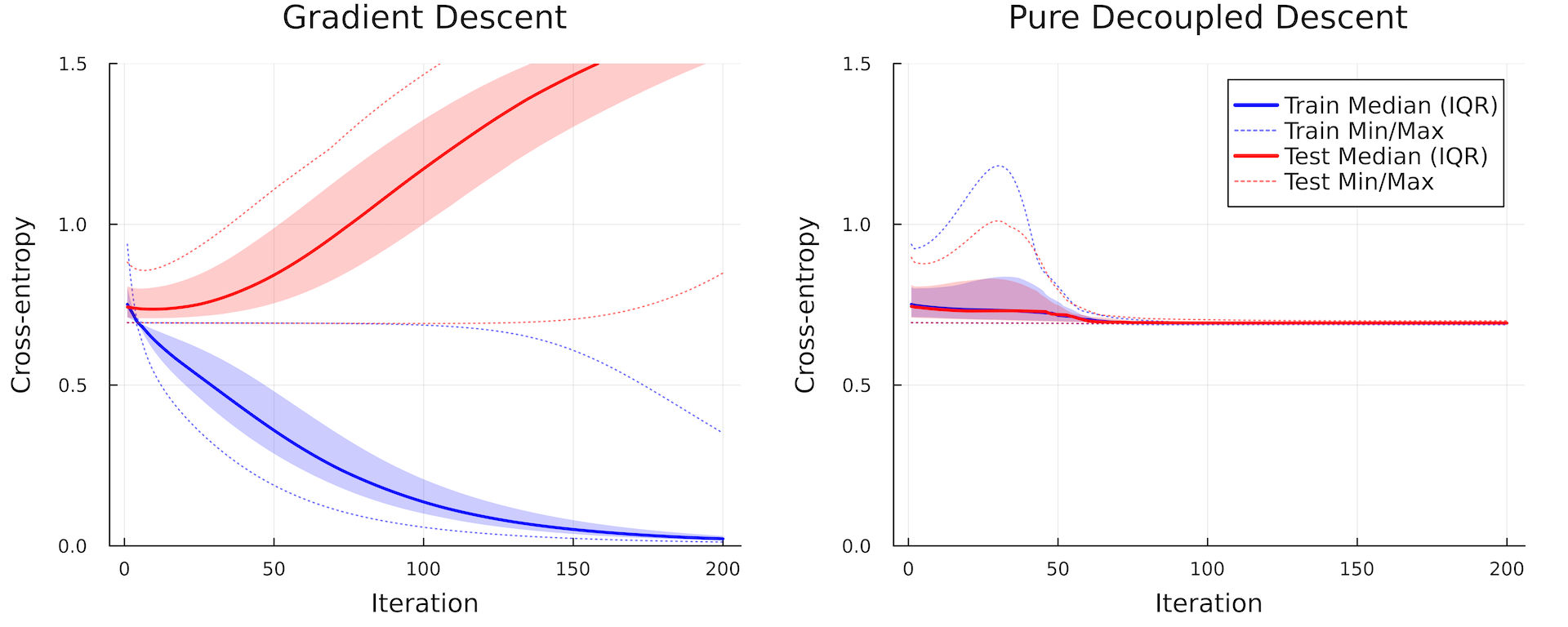}
    \includegraphics[width = \imageSize\linewidth]{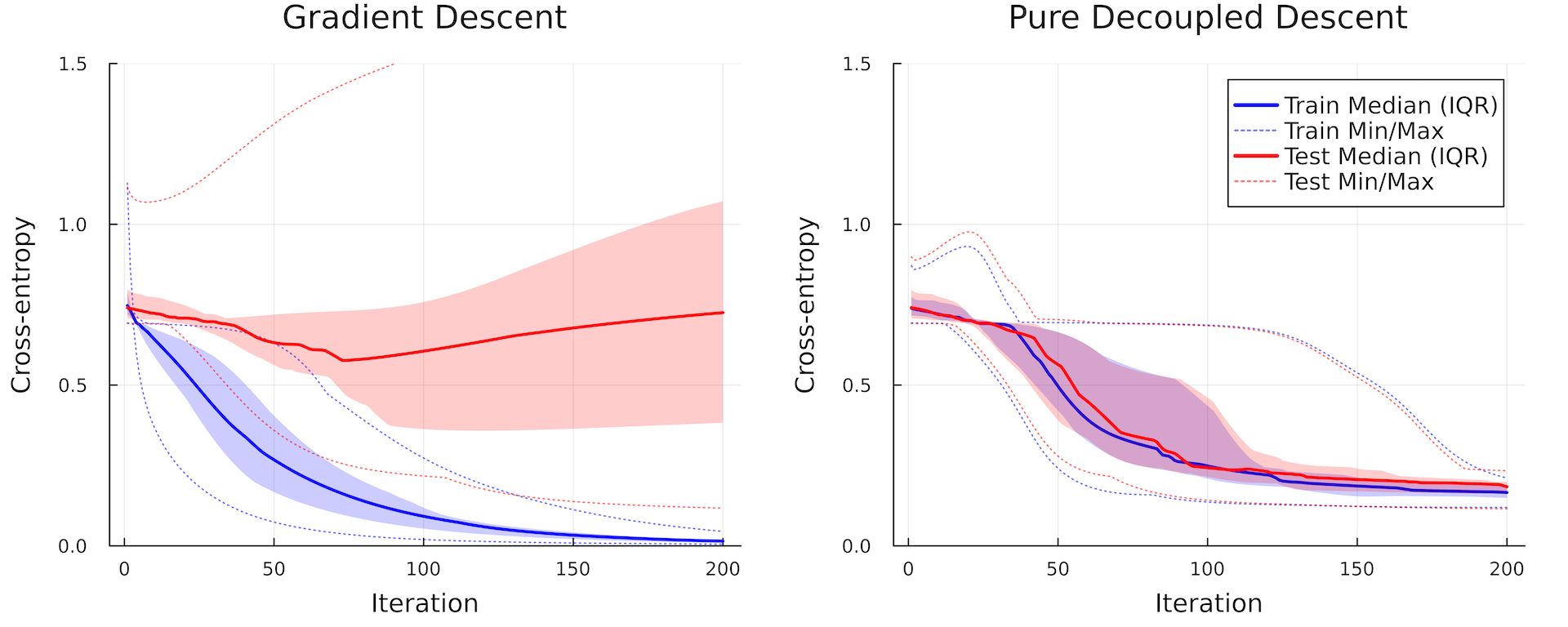}
    
    \includegraphics[width = \imageSize\linewidth]{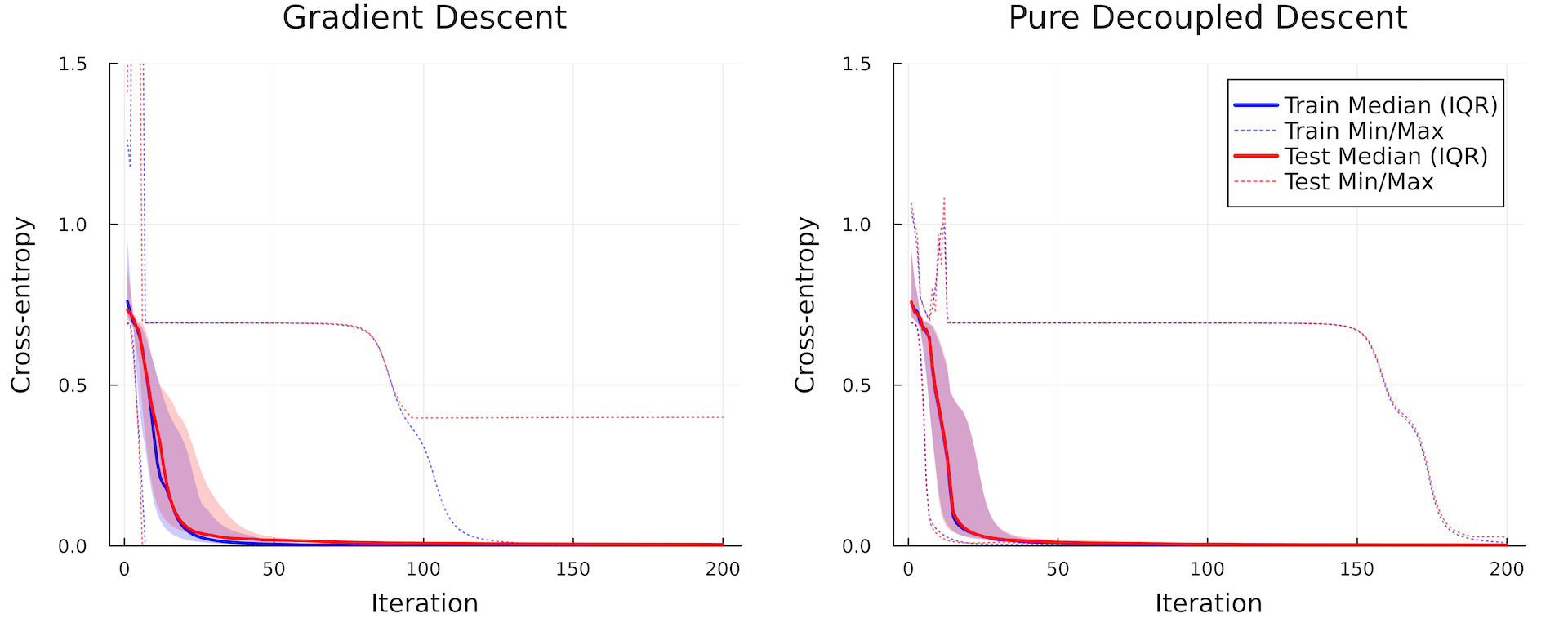}
    \caption{Summary statistics for 20 XOR runs ($n=d=1000$): GD (left) vs. DD (right) with $\eta=0.05$ and SNR $\lambda = 1$ (upper left), $\lambda = 4$ (upper right) and $\lambda = 8$ (bottom). Blue/red denote train/test error; solid lines are medians, shaded areas are interquartile ranges, and dotted lines show min/max. \textbf{Low SNR ($\lambda=1$):} GD overfits (low train/high test error); DD stabilizes both near $\log(2)$, reflecting the non-informative regime. \textbf{Medium SNR ($\lambda=4$):} DD maintains train-test parity, outperforming GD's overfit solution. \textbf{High SNR ($\lambda=8$):} Signal dominates noise; both algorithms achieve similar performance.}
    \label{fig:XorClassification}
\end{figure}

We also provide an example of zero-cost validation hyperparameter tuning from Section \ref{sec:trainTest}. We fix the value of $a_t = 1$ in $\cM_a$ and train $\theta^1$ and $\theta^2$ with a damped variant of pure DD given in Appendix \ref{sec:detailXor}; this training model incorporates varying $\eta_0$ and represents weight regularization for the parameters. In Figure \ref{fig:etaZeroSweep}, we run 20 shared replications of the XOR data and plot the train-test error for each choice of $\eta_0 \in \{1,0.9,0.8\}$ (recalling $\eta_0 = 1$ is pure DD). We can see that our zero-cost validation method is successful, each run demonstrates the train-test identity, and we observe that some weight regularization is beneficial in this case.

\begin{figure}[h]
    \centering
    \begin{minipage}{\imageSizeTwo\linewidth}
        \centering
        \textbf{\boldmath $\eta_0 = 1.0$}\\[1ex]
        \includegraphics[width=\linewidth]{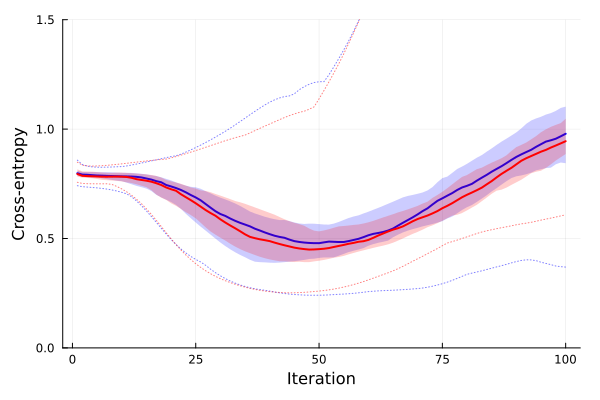}
    \end{minipage}
    \begin{minipage}{\imageSizeTwo\linewidth}
        \centering
        \textbf{\boldmath $\eta_0 = 0.9$}\\[1ex]
        \includegraphics[width=\linewidth]{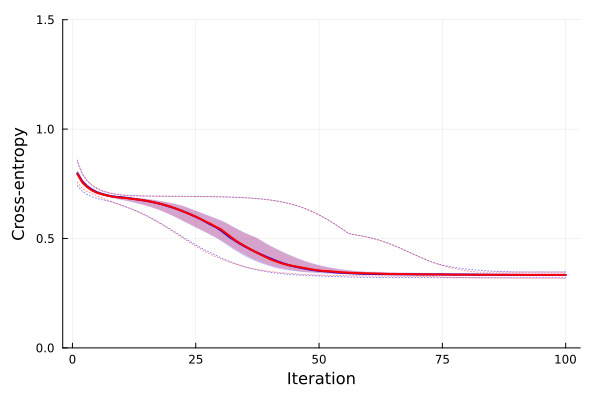}
    \end{minipage}
    \begin{minipage}{\imageSizeTwo\linewidth}
        \centering
        \textbf{\boldmath $\eta_0 = 0.8$}\\[1ex]
        \includegraphics[width=\linewidth]{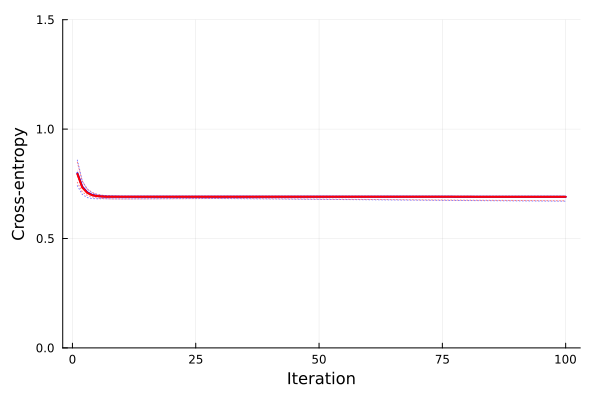}
    \end{minipage}
    \caption{Summary statistics for 50 XOR runs ($n=d=1000$) of damped DD (defined in Appendix \ref{sec:detailXor}): $\eta=0.05$, fixed $a_t = 1$, $\lambda = 4$ and $\eta_0 = 1$ (left), $\eta_0 = 0.9$ (middle) and $\eta_0 = 0.8$ (right). Blue/red colors and line type are equivalent to Figure \ref{fig:XorClassification}. \textbf{No Regularization}: $\eta_0 = 1$ (pure DD) gives an initial decrease in the loss but suffers from inflated values of $\Omega_t\ss{t,t}$ for later iterations (note, the early stopping technique of Section \ref{sec:trainTest} would be helpful in this case). \textbf{Mild Regularization}: $\eta_0 = 0.9$ finds a good balance in weight regularization which controls variance inflation while achieving good test error and training stability. \textbf{High Regularization}: $\eta_0 = 0.8$ regularizes too strongly and leads to an estimator that is too conservative, leading to a larger train-test error.}
    \label{fig:etaZeroSweep}
\end{figure}

\subsection{Varying Layer Widths For Classifying Zeros And Eights For MNIST Data}\label{sec:MNIST}

We apply DD to a simple MNIST binary classification test (i.e. zeros vs. eights). For the train data, we sample $n=800$ digits (grayscale images of dimension $d=784$) from all zeros and eights in the MNIST data set, rescale by SNR $\lambda/d$, and inject mean zero, variance $1/d$, discrete noise,
\[\delta_{-\sqrt{2}/\sqrt{d}} \text{ wp } 1/4, \quad \delta_{\sqrt{2}/\sqrt{d}} \text{ wp } 1/4,\quad \delta_0 \text{ wp }1/2,\label{eq:discNoise}\]
(instead of the standard Gaussian noise from distribution \eqref{eq:dist}). The test set is all zero and eight digits in MNIST with identical processing; this experiment verifies if the train-test identity of decoupled descent demonstrates universality phenomena similar to AMP algorithms and is robust when the data is not immediately represented by $J$ mixture modes.

Additionally, to see if the train-test identity for decoupled descent degrades as $L$ increases in size, we train a two-layer network (hidden widths $L \in \{3, 9, 27\}$, $\tanh$ activation, sigmoid output) using cross entropy loss and learning rate $\eta = 1$. We directly implemented this training procedure in \texttt{Pytorch} using a special MLP class (see Appendix \ref{sec:implementation} for details) and thus utilized the default \texttt{PyTorch} initializations for our parameters. Figure \ref{fig:lossGDvDD} plots the train-test error for these three models, we see the train-test identity still holds regardless of the size of the hidden layer, the discrete (instead of Gaussian) noise and a possibly more complex signal distribution.

\begin{figure}[h]
    \centering
    \includegraphics[width=\imageSize\linewidth]{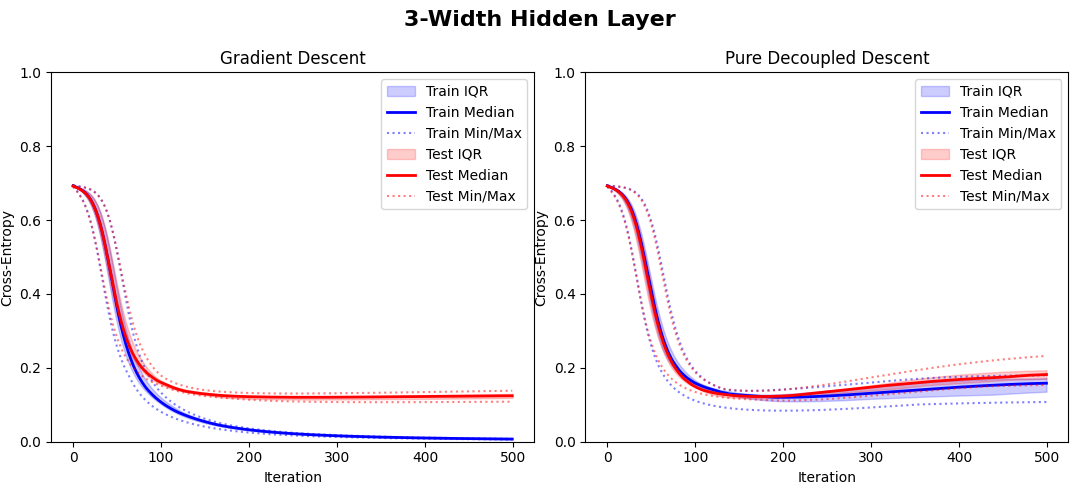}
    \includegraphics[width=\imageSize\linewidth]{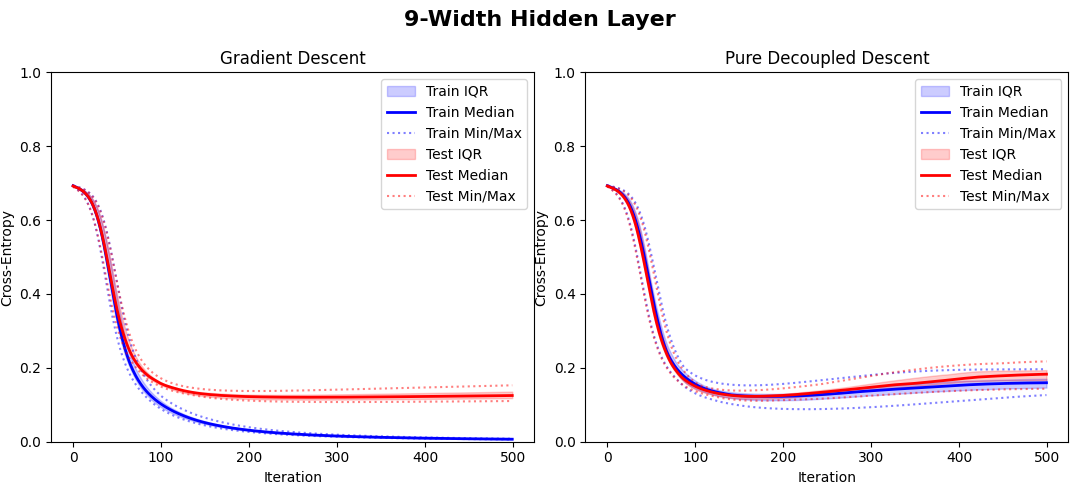}
    
    \includegraphics[width=\imageSize\linewidth]{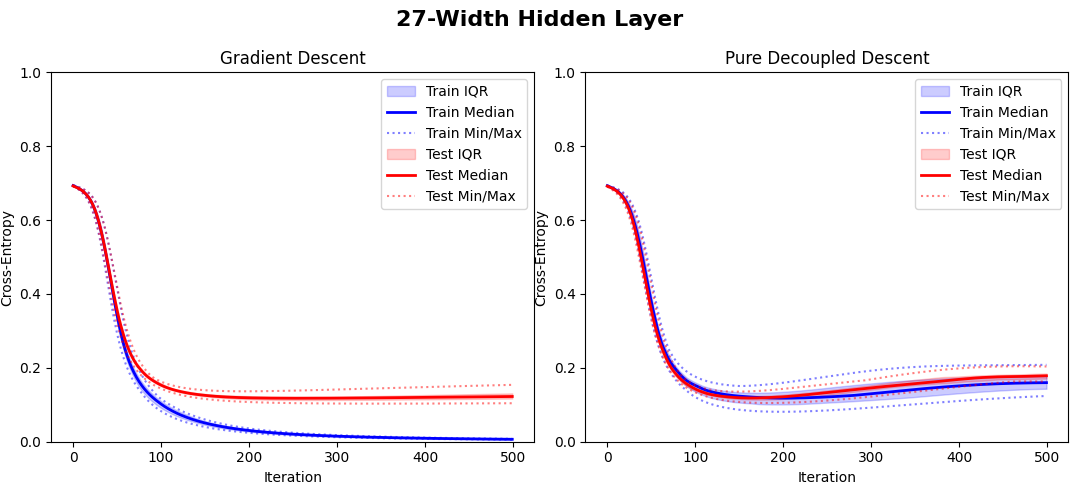}
    \caption{MNIST zeros vs. eights train/test errors ($d=784, n=800, \lambda=30$) over 20 replications: GD (left) vs. DD (right) with discrete noise from distribution \eqref{eq:discNoise}. Blue/red denote train/test error; solid lines are medians, shaded areas are IQRs, dotted lines are min/max. We run a two layer network with hidden layers ($L \in \{3, 9, 27\}$). The train-test error identity continues to hold as $L$ grows and when the Gaussian noise is replaced with this discrete counterpart.}
    \label{fig:lossGDvDD}
\end{figure}

\begin{remark}[On computational complexity]
 Recall implementing pure DD required two correction terms. Term $\eta \sum_{s=1}^{t-1} \hat h_{s}$ adds $O(n)$ storage but no computational overhead, trivial compared to the size of the data. The Hessian term $\eta \alpha \left( \frac{1}{n}\sum_{i=1}^n \nabla_h^2\Psi(h_{t,i}, y_i, a_t)\right)$ requires a per-iteration Jacobian. 
 This Hessian overhead is a negligible, $d$ independent constant---the necessary mechanism for DD to maintain the train-test identity. It is also worth noting that using standard auto-differentiation allows the process of deriving the Hessian automatically, avoiding the process of deriving the correction term on a case by case basis. For completeness, we provide some comparative run-times between GD and pure DD for the $L = 9$ case in Appendix \ref{sec:mnistTime}.
\end{remark}

\subsection{Training A MLP Head On CIFAR-10 ResNet Embeddings For Classifying Cats And Dogs}\label{sec:CIFAR}

We finally consider a modern CIFAR-10
classification task. The goal of this task is to classify images as either
pictures of cats (CIFAR-10 class 3) or dogs (CIFAR-10 class 5). We note that 
the DD algorithms in this section were also natively implemented using \texttt{Pytorch}'s
\texttt{torch.autograd} function (See Appendix \ref{sec:implementation} for the MNIST example, the CIFAR-10 example is similar).

The data matrix for this application is wildly different from Assumption \ref{as:main0}.
Instead of noisy images, we use pre-trained ResNet-18 embeddings of dimension $d = 512$.
We sample $n = 800$ random embedded vectors from the full pool of embedded 
cat and dog vectors from CIFAR-10 for both the training and testing set independently.
Then, we train a two layer MLP classification head with hidden layer width five 
and $\tanh$ activation using both GD and pure DD.

The motivation for using DD on these vectors is from Gaussian
equivalence theory \cite{goldt2021} which postulates one can analyze the
intermediate layers of a trained network with Gaussian surrogates. We consider increasingly stronger ``whitening'' over our training and testing data. In all cases, we standardize the embedding vectors by subtracting their
coordinate-wise mean and dividing by the coordinate-wise standard deviation
(with respect to only the training distribution). A detailed explanation of the whitening procedures are deferred to Appendix \ref{sec:embedding}, informally we have (1) \textit{Vanilla}, rescaling features by $1/\sqrt{d}$; (2) \textit{ZCA (Train)}, applying whitening computed solely on training data \cite{kessy2018}; and (3) \textit{Joint ZCA}, whitening across the combined train-test pool to prevent covariance drift. These procedures should be interpreted as 
increasingly stronger processing steps to make the embedded vectors of the train and test 
set seem more Gaussian.

We can see the training and test curves for both GD and pure
DD in Figure \ref{fig:cifar}. Although DD has reduced the generalization gap, it is larger than previous examples as {\bf we have not injected random noise
in the data set} and therefore are unlikely to satisfy Assumption \ref{as:main00} completely. This suggests a more general method to apply DD to models with non-independent noise structures.

\begin{figure}
    \centering
    \includegraphics[width=\imageSize\linewidth]{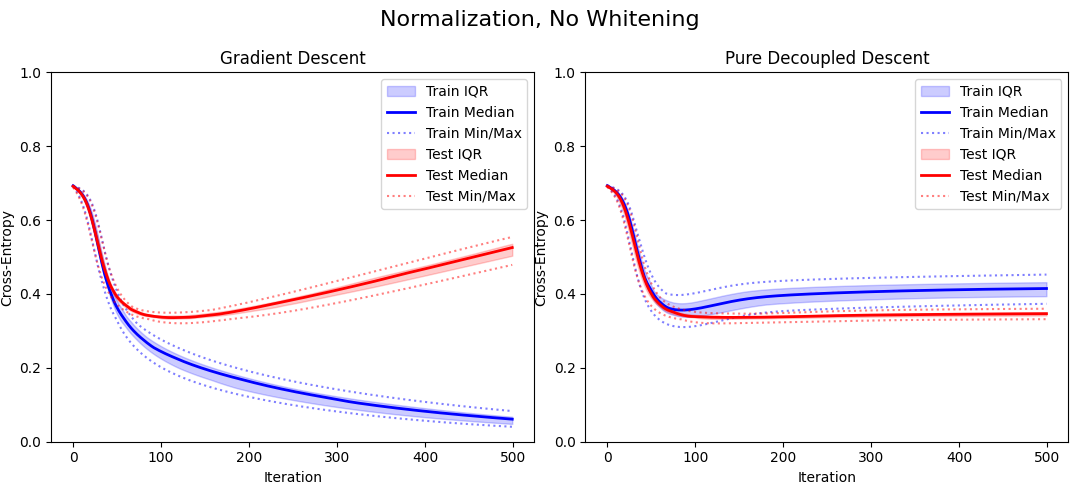}
    \includegraphics[width=\imageSize\linewidth]{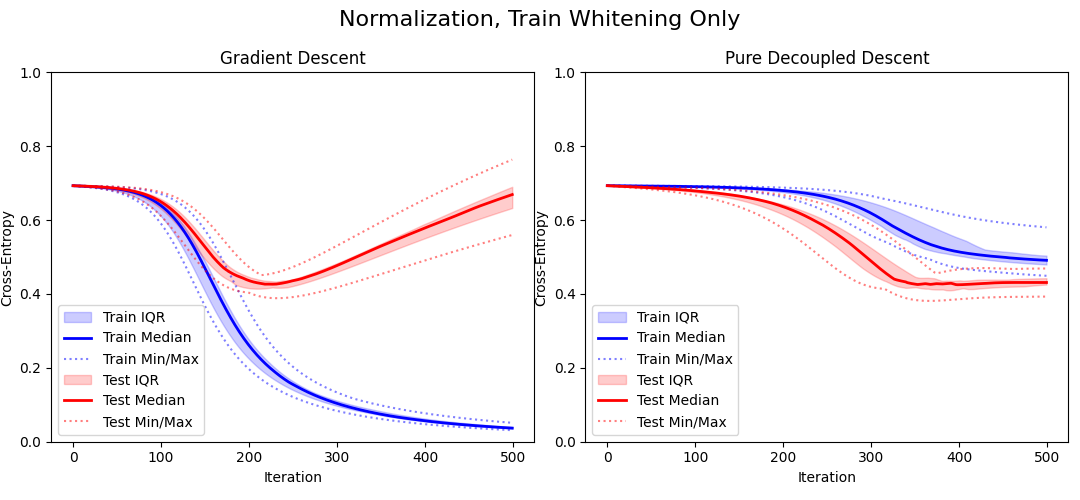}
    
    \includegraphics[width=\imageSize\linewidth]{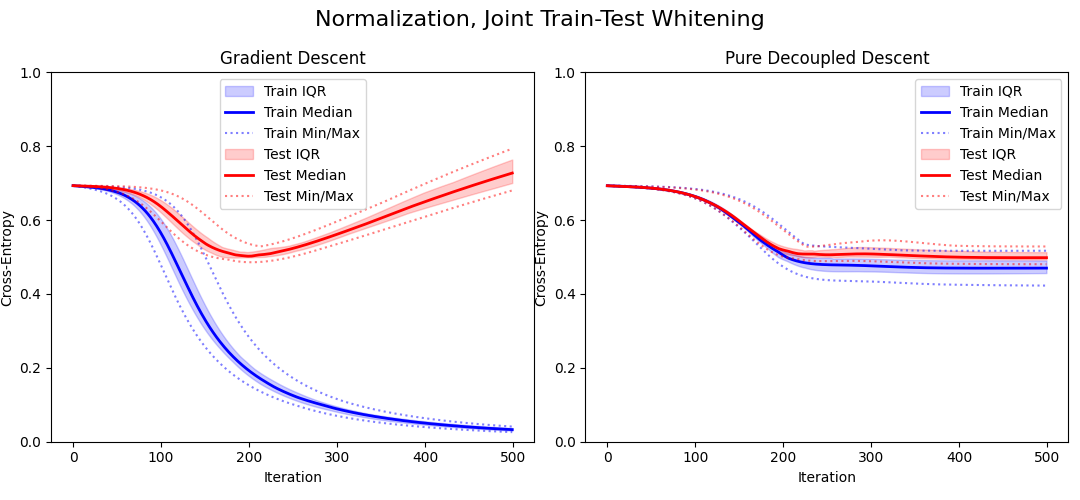}
    \caption{Train/test errors for CIFAR-10 (cats vs. dogs). GD (left) vs. DD (right). Blue/red denote train/test error; solid lines are medians, shaded areas cover the inter-quartile range, dotted lines show min/max, and title is the method of whitening (see Section \ref{sec:CIFAR}). DD reduces overfitting effects compared to GD for training a classification head to ResNet-18 embeddings.} \label{fig:cifar}
\end{figure}

\section{Discussion And Conclusion}\label{sec:conclusion}

To conclude, we presented a novel machine learning training algorithm and demonstrated it solves the train-test error
disconnect present in full-batch GD. We additionally provided some
initial guidance on how to design such algorithms and gave
empirical validation of their success.

\textbf{Limitations} While DD guarantees an exact asymptotic train-test 
identity, the practical deployment faces several constraints. 
The proofs assume Gaussianity on the data $X$, although some relaxation is expected (Remark \ref{rm:universality} (1)). Real-world data, however, contains structured correlations that do not satisfy our assumptions.
Our results are asymptotic and would benefit from specified finite sample rates (Appendix \ref{rm:finite}). We require full-batch training which rules out stochastic methods like SGD; this additionally introduces DD memory and computational overhead due to the correction terms which can result in $\approx L\times$ longer step times (Appendix \ref{sec:mnistTime}). We have only considered models with a wide first layer and finite subsequent layers; over-parameterization of later layers may break the train-test identity, although we do not see this in simulations (Section \ref{sec:MNIST}). Finally, we do not implement a mechanism in DD to represent signal drift, a common second source of overfitting.

\textbf{Broader Impacts} When Assumption \ref{as:main00} approximately holds, DD offers three benefits to practitioners: (1) Our zero-validation principle maximizes data utilization, vital in domains where data is scarce or expensive to collect. (2) The number of training runs can be limited by embedding hyper-parameter tuning into training dynamics. (3) Exact test error tracking gives a honest assessment of generalization, preventing the deployment of overfitted models.

\textbf{Future Directions} We present three future directions. \textbf{Orthogonally Invariant Noise}: Can DD extend to data with dependent elements \cite{zhong2024,fan2021,liu2024} with symmetry properties? It remains open whether the standard Onsager term's partial success on CIFAR-10 is a coincidence or a general property of embedding vectors. \textbf{Stochastic DD}: To reduce the full-batch computational overhead, can recent DMFT derivations for SGD \cite{fan2026,nishiyama2026} inform a mini-batch DD iteration, greatly expanding its practical applicability.
\textbf{Large-Width Networks}: DD currently requires finite-width subsequent layers, adapting deep AMP/DMFT frameworks \cite{yang2023, gerbelot2022,xu2022} to DD with multiple $d$-scaling layers would increase the applicability of these results.

\section*{Acknowledgements}

Thanks to Zhou Fan for the many helpful discussions and support on this project.

\bibliographystyle{plainnat}
\bibliography{newrefs.bib}

\appendix
\section{Deferred Theory}

\subsection{Notation}\label{sec:notation}

Let $[n] = \{1, \dots, n\}$; $\RR$ and $\NN$ denote the real and natural numbers, with $(\cdot)^k$ as the $k$-fold Cartesian product. For $M \in \RR^{n \times L}$, $M_i \in \RR^L$ denotes the $i$-th row. For a block matrix $M$, $M[r,s]$ is its $(r,s)$-th block. $\text{vec}(\cdot)$ and $\text{mat}(\cdot)$ represent standard vectorization and matricization. $\nabla_{(\cdot)}^k$ is the $k$-fold partial derivative. Norms $\|\cdot\|_2, \|\cdot\|_{\infty}, \|\cdot\|_F, \|\cdot\|_{\text{op}}$ are the Euclidean, infinity, Frobenius, and operator norms. $\1\{\cdot\}$ is the indicator function and $\delta_x$ is the Dirac distribution at $x \in \RR$. $\otimes, \odot, \langle \cdot, \cdot \rangle$ denote Kronecker, Hadamard, and inner products (where $\otimes$ additionally means the independent product of two probability measures). $e_B \in \RR^n$ is $(e_B)_i = \1\{i \in B\}$. $O, o, \Theta$ represent standard asymptotic notation with $n$-dependent growth.

\subsection{A Method Of Descent Under The Train-Test Identity}\label{sec:reduceTest}

Suppose an algorithm $\cA$ satisfies the train-test identity and generates the iterate $(\theta_t, h_t, a_t)$. In addition, there exists a number of updates $(\theta^{(k)}_{t+1},
h^{(k)}_{t+1}, a^{(k)}_{t+1})_{k \in [K]}$ generated by $K$ different algorithms $(\cA^{(k)})_{k \in [K]}$, each maintaining the train-test identity.
Each algorithm $\cA^{(k)}$ produces an update with training error
\[\frac{1}{n} \sum_{i=1}^n \cL( \cM_{a^{(k)}_{t+1}}(h^{(k)}_{t+1,i}),
y_i).\label{eq:decoupledProxy}\] And, by the train-test identity, \eqref{eq:decoupledProxy} is an asymptotically consistent estimator for the
test error of update $(\theta^{(k)}_{t+1}, h^{(k)}_{t+1},
a^{(k)}_{t+1})$, i.e. $\EE_{\check x, \check y}[\cL(\cM_{a^{(k)}_{t+1}}(\check
x^\top \theta^{(k)}_{t+1}), \check y)]$. Therefore, reducing the test error at step $t+1$ follows by choosing the update that minimizes \eqref{eq:decoupledProxy}, so long as one exists.
Other algorithms which do not satisfy the train-test identity can fundamentally never offer such a guarantee due to the train-test disconnect from the introduction.

\subsection{Deferred Examples}

\begin{example}\label{ex:data}
    Many common learning problems can be realized by the above data model.
    \begin{enumerate}
        \item The signal-less regression model from the introduction is given with $J = 1$, $\mu_j = 0$ and $\PP_j = \delta_0$ (i.e. a Dirac measure at zero).
        \item A simple classification problem with flipped class noise has $J = 2$, $\mu_1 = v$ and $\mu_2 = -v$, and for some $\epsilon \in (0,1/2)$) we set $\PP_1 = (1-\epsilon) \delta_{-1} + \epsilon \delta_1$ and $\PP_2 = (1-\epsilon) \delta_{1} + \epsilon \delta_{-1}$ with class probabilities $p_1 = p_2 = 1/2$.
        \item A discrete-valued regression problem has $J$ being some large constant, say $100$. We let $\mu_j = j v$ with $v \in \RR^d$. Then let $\PP_j$ be the convolution of $\delta_{c_1 j + c_2}$ and $\PP_{\rm noise}$ where $c_1, c_2 \in \RR$ and $\PP_{\rm noise}$ is the response's noise distribution, finally we can let $p_1, \dots, p_J$ be some arbitrary prior over data examples.
        \item An XOR classification problem has $J = 4$, $\mu_1 = [v,v], \mu_2 = -[v,v], \mu_3 = [-v,v], \mu_4 = [v,-v]$ with $v \in \RR^{d/2}$,  $\PP_j = \delta_{0}$ for $j \leq 2$ and $\PP_j = \delta_{1}$ otherwise. Let each $p_j = 1/4$.
    \end{enumerate}
\end{example}

\begin{example}\label{ex:mlp} Many architectures utilized in practice can be used under our parametric models. Each example below is trained using either MSE loss $\frac{1}{2}(\hat y - y)^2$ or MAD loss $|\hat y - y|$ for simplicity. 
\begin{enumerate}
    \item A simple linear regression model is given by $\model(x) = x^\top \theta$.
    \item A generalized linear model is given by $\model(x) = \sigma(x^\top \theta)$ for a link function $\sigma: \RR \to \RR$.
    \item Let $L = L'$, and let $\sigma, \phi: \RR \to \RR$ be applied element-wise, a two-layer network is given by $\model(x) = \phi(\sigma(x^\top \theta)a)$.
    Deeper networks of finite width are given by setting $a = (\vec(W_1), \dots, \vec(W_{k-1}), w_k)$ where $W_i \in \RR^{L \times L}$ for $i \in [k-1]$, $w_k \in \RR^L$ and selecting $\model(x) = \phi(\sigma(\cdots \sigma(\sigma(x^\top \theta)W_1)\cdots W_{k-1})w_k)$.
    \item Corresponding to Example \ref{ex:data} (4), we could consider a model for the XOR problem with $L = 2, \theta = (\theta_1, \theta_2), L'=1$ and selecting $\model(x) = a(x^\top \theta_1)(x^\top \theta_2)$.
\end{enumerate}
\end{example}

\subsection{Main Assumption}\label{sec:assumption}

\begin{assumption}\label{as:main0}
The following hold when both $n, d \to \infty$, with $n/d \to \alpha \in (0,\infty)$ fixed:
\begin{enumerate}
    \item\label{as:data} [Data Composition] In distribution \eqref{eq:dist}, for each $j,k \in [J]$, $\PP_j$ has bounded moments of all orders, $p_j \in [0,1]$ and $\lim_{d \to \infty} d^{-1}\mu_j^\top \mu_k$ exists. Further, each row of data $(X,y) \in \RR^{n \times d} \times \RR^n$ is drawn \iid from distribution \eqref{eq:dist}.
    \item\label{as:limits} [Initialization Limits] 
    Initialization $\theta_1$ is independent of the data $(X,y)$ and $\lim_{d \to \infty} d^{-1} \theta_1^\top \theta_1 = \bar \theta^2$ almost surely. For each $j \in [J]$, $\lim_{d \to \infty} d^{-1} \theta_1^\top \mu_j = m_{j,1}$ almost surely. The norm of the initialization $\|a_1\|_2$ is uniformly bounded. 
    \item\label{as:algo} [DD Lipschitzness] Uniformly over $y \in \RR$, the functions $f(h, y, a)$, $g(h, y,a)$ and their derivatives $\nabla_h f(h,y,a), \nabla_h g(h,y,a)$ are Lipschitz and bounded in $h$ and $a$.
    \item \label{as:model} [Model And Loss Smoothness] Let $\gradient$ from Definition \ref{def:lossGrad}, uniformly over $y \in \RR$:
    \begin{enumerate}
        \item $\Psi$ is Lipschitz, four times continuously differentiable with respect to
        $h$ and twice continuously differentiable with respect to $a$.
        \item Let $\omega \in \RR^{L \times L}$ be a covariance matrix with bounded operator norm, vectors $m \in \RR^L$ and $a \in \RR^{L'}$ with $\|m\|_2$, $\|a\|_2$ bounded, and let $h = m + \omega^{1/2} G$ where $G \sim \cN(0, \Id_L)$. For each class $j \in [J]$ with $Y_j \sim \PP_j$, there exists a bounded constant $C > 0$ where,
        \begin{align}
    \max\left(\|\EE[\nabla_a^2 \Psi(h, Y_j, a)]\|_{\op}, 
    \|\EE[\nabla_h \nabla_a \Psi(h, Y_j, a)]\|_{\op}, 
    \|\EE[\nabla_h^2 \Psi(h, Y_j, a)]\|_{\op}\right) &\leq C \\
    \max\left(\|\EE[\nabla_h^2 \nabla_a \Psi(h, Y_j, a)]\|_{\Fro}, 
    \|\EE[\nabla_h^3 \Psi(h, Y_j, a)]\|_{\Fro}, 
    \|\EE[\nabla_h^4 \Psi(h, Y_j, a)]\|_{\Fro}\right) &\leq C
\end{align}
    \end{enumerate}
\end{enumerate}

\end{assumption}

\subsection{The Asymptotic Test Error}\label{sec:asyTestError}

We prove the following proposition.

\begin{proposition}\label{prop:testErr}
    If Assumption \ref{as:main0} holds and $\lim_{d \to \infty} \mu_j^\top \theta/d = m_{j,\theta}$, $\lim_{d \to \infty} \theta^\top \theta /d = \Omega_\theta$, $\lim_{n,d \to \infty} a = \bar a$ almost surely for trained parameters $(\theta, a)$, then
    there exists a deterministic function $\testfunc$ depending only on the aforementioned limits such that,
    \[\lim_{d \to \infty} \EE_{\check x, \check y}[\cL(\cM_{\theta, a}(\check x), \check y)] = \testfunc(m_{1,\theta}, \dots, m_{J, \theta}, \Omega_\theta, \bar a).\]
\end{proposition}

\begin{proof}

Consider the test error from \eqref{eq:testDesc} with $\beta = (\theta, a)$ denoting $\check L(\theta, a) = \EE_{\check x, \check y}[\cL(\model(\check x), \check y)]$ where $(\check x, \check y)$ are drawn from distribution \eqref{eq:dist}, we write
\[
\lim_{d \to \infty} \check L(\theta, a) = \lim_{d \to \infty} \sum_{j=1}^J p_j \EE[\cL(\modelshort(\mu_j^\top \theta/d + \check Z_\theta), \check Y_j)],
\]
where $\check Z_\theta \sim \cN(0, \theta^\top \theta/d)$ and $\check Y_j \sim \PP_j$.
Under the assumption that $\mu_j^\top \theta/d$, $\theta^\top \theta / d$ and $a$ have almost sure limits $m_{j,\theta}$ (for each $j \in [J]$), $\Omega_\theta$ and $\bar a$ respectively, we have that
\[\label{eq:asyTest}
        \lim_{d \to \infty} \check L(\theta, a) = \sum_{j=1}^J p_j \EE_{\substack{\check Z_\theta \sim \cN\left(0, \Omega_\theta \right)\\\check Y_j \sim \PP_j}}[\cL(\cM_{\bar a}(m_{j,\theta} + \check Z_\theta), \check Y_j)],
\]
by dominated convergence as Assumption \ref{as:main00} assumes $\Psi(h,y,a) = \cL(\cM_a(h), y_j)$ is Lipschitz. Noting the right hand side of \eqref{eq:asyTest} is from Definition \ref{def:testErr} gives the proof.
\end{proof}

\subsection{Full Definition Of State Evolution}\label{sec:fullSe}

As we see momentarily, DD can be equivalently written as an AMP algorithm, thus off-the-shelf AMP results immediately provide a distributional characterization of the pre-activations $h_1, \dots, h_t \in \RR^{n \times L}$ and parameters $\theta_1, \dots, \theta_t \in \RR^{d \times L}$, $a_1, \dots, a_t
\in \RR^{L'}$. This description is given by a set of low-dimensional recursive equations termed \textit{state evolution}. Informally, we derive a recursion of variables $\Omega_t$, $\Sigma_t$, $\Xi_t$, $m_{1,t}, \dots, m_{J,t}, \bar a_t$ which represent the almost sure limits of (self-)overlaps between $\theta_t, \mu_j$ and $\tilde \theta_t$ alongside the limiting value of $a_t$.

\begin{definition}\label{def:ddState}
Let $(\PP_j)_{j \in [J]}, (\chi_{j,k})_{j,k \in [J]},
(p_j)_{j \in [J]}$ be from Assumption \ref{as:main0} \eqref{as:data},
$(m_{j,1})_{j \in [J]}, \bar \theta^2$ be the limits from Assumption \ref{as:main0}
\eqref{as:limits} and let $g,f$ be the functions from Assumption
\ref{as:main0} \eqref{as:algo}.

Define the following state evolution parameters $m_{j,t} \in \RR^L$ for $j \in [J],
t \in [T]$, block matrices $\Sigma_t, \Omega_t \in (\RR^{L \times L})^{(t+1)
\times (t+1)}, \Xi_t \in (\RR^{L \times L})^{t \times t}$ and $\bar a_t \in
\RR^{L'}$ for $t \in [T]$ recursively as follows. Recall the indexing $\ss{r,s}$
denotes the $(r,s)$-th $L \times L$ block of a given matrix, we have
    \begin{align}
    \Sigma_t\ss{r,s} &= \sum_{j=1}^J p_j \EE[g(G^{r} + m_{j,r}, Y_j, \bar a_r) g(G^{s} + m_{j,s}, Y_j, \bar a_s)^\top ]\\
    l_{j,t} &= p_j \EE[g(G^t + m_{j,t}, Y_j, \bar a_t)]\\
    \Xi_{t}\ss{r+1, s} &= \eta_{0} \Xi_{t-1}\ss{r,s} - \eta_{1} \Big(\alpha \Sigma_{t}\ss{r,s} + \alpha^2 \sum_{j,k=1}^J \chi_{j,k} l_{j,r} l_{k,s}^\top \Big)\\
    \Omega_{t+1}\ss{r+1,s+1} &= \eta_{0}^2 \Omega_{t}\ss{r,s} - \eta_{0} \eta_{1}\left(\Xi_{t}\ss{r,s} + \Xi_{t}\ss{s,r}^\top\right) + \eta_{1}^2\Big(\alpha \Sigma_{t}\ss{r,s} + \alpha^2 \sum_{j,k=1}^J \chi_{j,k} l_{j,r} l_{k,s}^\top \Big)\\ 
    m_{j,t+1} &= \eta_{0} m_{j,t} - \eta_{1} \alpha \sum_{k=1}^J \chi_{j,k} l_{k,t}\\
    \bar a_{t+1} &= \gamma_{0} \bar a_t - \gamma_{1} \sum_{j=1}^J p_{j} \EE[f(G^t + m_{j,t}, Y_j, \bar a_t)]
\end{align}
with indices $r,s \in [t]$ and $j \in [J]$. The expectations are taken over $Y_j
\sim \PP_j$, the random vectors $G^t \sim \cN(0, \Omega_t\ss{t,t})$, and the pairs
$(G^r, G^s) \sim \cN\left(0, \left[\begin{smallmatrix}\Omega_t\ss{r,r} &
\Omega_t\ss{r,s} \\ \Omega_t\ss{s,r} &
\Omega_t\ss{s,s}\end{smallmatrix}\right]\right)$. The system is initialized with $\bar a_1 = a_1$,
$\Omega_t\ss{1,1} = \Omega_1\ss{1,1} = \bar \theta^2$, $\Xi_t\ss{1,s} =
\alpha \sum_{j=1}^J m_{j,1} l_{j,s}^\top$, and $\Omega_{t+1}\ss{1, s+1} =
\eta_{0} \Omega_t\ss{1,s} - \eta_{1} \Xi_t\ss{1,s} =
\Omega_{t+1}\ss{s+1,1}^\top$.
\end{definition}

\begin{remark}\label{rem:ddstate}
    As a special case of Definition \ref{def:ddState}, we
    have the recursion,
    \[\Omega_{t+1}\ss{t+1, t+1} = \eta_{0}^2 \Omega_t\ss{t,t} -
    \eta_{1} \eta_{0} (\Xi_t\ss{t,t} + \Xi_t\ss{t,t}^\top) +
    \eta_{1}^2 \left(\alpha \Sigma_t\ss{t,t} + \alpha^2
    \sum_{j,k=1}^J \chi_{j,k} l_{j,t} l_{k,t}^\top\right).\]
\end{remark}

\subsection{Approximate Message Passing And Proving Lemma \ref{lem:ddState}}\label{sec:AMP}

For Lemma \ref{lem:ddState} and Theorem \ref{thm:ddLoss}, we require the following assumption.
\begin{assumption}\label{as:phi} 
    Uniformly over $y \in \RR$, function $\phi(h, y, a): \RR^{L} \times \RR \times \RR^{L'} \to \RR$ is Lipschitz in $h$ and $a$. Moreover, there exists a bounded constant $C > 0$, independent of $t$, where the matrices from Definition \ref{def:ddState} satisfy $\max_{t \in [T]} \max(\|\Sigma_t^{-1}\|_{\op}, \|\Omega_t^{-1}\|_{\op}, \|\Sigma_T\ss{t,t}\|_{\op}, \|\Omega_T\ss{t,t}\|_{\op}) \le C$.
\end{assumption}

\subsubsection{Relating Back To AMP}

Consider the original DD algorithm dependent on the activation functions $g: \RR^{L} \times \RR \times \RR^{L'} \to \RR^L$ and $f: \RR^{L} \times \RR \times \RR^{L'} \to \RR$ and hyperparameters $\eta_0, \eta_1, \gamma_0, \gamma_1$ from Equation \eqref{eq:dd}, we repeat the algorithm below for convenience below,
\[\label{eq:appdd}\begin{split}
    h_t &= X \theta_t + \eta_{1} \sum_{s=1}^{t-1} \eta_{0}^{(t-1)-s} \hat h_{s}\\
    \hat h_t &= g(h_t, y, a_t)\\
    \tilde \theta_t &= X^\top \hat h_t - \alpha  \left(\frac{1}{n} \sum_{i=1}^n
    \nabla_h g(h_{t,i}, y_i, a_t) \right)\theta_t\\ 
    \theta_{t+1} &= \eta_{0} \theta_t - \eta_{1} \tilde \theta_t\\
    a_{t+1} &= \gamma_{0} a_t - \gamma_{1} \frac{1}{n} \sum_{i=1}^n f(h_{t,i}, y_i, a_t). \end{split}\]

Note, the above algorithm is invariant to a permutation of the rows of $X$, thus we fix a representation with $X = \frac{1}{d}S + Z = \frac{1}{d} \sum_{j=1}^J e_{B_j} \mu_j^\top + Z$ where $B_1, \dots, B_J$ partition $[n]$ and $Z$ is an element-wise independent Gaussian matrix with $Z_{ij} \sim \cN(0,1/d)$.
Moreover, by Assumption \ref{as:main0} \eqref{as:data}, we have the almost sure limits $\lim_{n\to \infty} |B_j|/n = p_j$. 

Therefore, we can equivalently write \eqref{eq:appdd} as,
\[\label{eq:appdd2}\begin{split}
    h_t &= \sum_{j=1}^J e_{B_j} \frac{\mu_j^\top \theta_t}{d} + Z\theta_t + \eta_{1} \sum_{s=1}^{t-1} \eta_{0}^{(t-1)-s} \hat h_{s}\\
    \hat h_t &= g(h_t, y, a_t)\\
    \tilde \theta_t &= \sum_{j=1}^J \mu_j \frac{n}{d} \frac{e_{B_j}^\top \hat h_t}{n} +  Z^\top \hat h_t - \alpha  \left(\frac{1}{n} \sum_{i=1}^n
    \nabla_h g(h_{t,i}, y_i, a_t) \right)\theta_t\\ 
    \theta_{t+1} &= \eta_{0} \theta_t - \eta_{1} \tilde \theta_t\\
    a_{t+1} &= \gamma_{0} a_t - \gamma_{1} \frac{1}{n} \sum_{i=1}^n f(h_{t,i}, y_i, a_t). \end{split}\]

\textbf{We prove Lemma \ref{lem:ddState} in three steps.}

(1) We replace $\frac{\mu_j^\top \theta_t}{d}$, $\frac{e_{B_j}^\top \hat
h_t}{d}$ and $a_t$ by pre-specified vectors $m^\top_{j,t}, l_{j,t}^\top \in \RR^{1 \times L}$ and
$\bar a_t \in \RR^{L'}$ respectively. Moreover, with $\tilde Z \in \RR^{n
\times d}$, we couple $Z = \alpha^{1/2} \tilde Z$ and replace $Z$ by
$\alpha^{1/2} \tilde Z$ noting that $\tilde Z$ has independent $\cN(0,1/n)$
entries.

We then analyze the following {\it frozen} algorithm, recalling that $\alpha = n/d$, 
\[\label{eq:frozenDd}\begin{split}
    h^{\rm frozen}_t &= \sum_{j=1}^J e_{B_j} m^\top_{j,t} + \alpha^{1/2} \tilde Z\theta^{\rm frozen}_t + \eta_{1} \sum_{s=1}^{t-1} \eta_{0}^{(t-1)-s} \hat h^{\rm frozen}_{s}\\
    \hat h^{\rm frozen}_t &= g(h^{\rm frozen}_t, y, \bar a_t)\\
    \tilde \theta^{\rm frozen}_t &= \alpha \sum_{j=1}^J \mu_j l_{j,t}^\top +  \alpha^{1/2} \tilde Z^\top \hat h^{\rm frozen}_t - \alpha  \left(\frac{1}{n} \sum_{i=1}^n
    \nabla_h g(h^{\rm frozen}_{t,i}, y_i, \bar a_t) \right)\theta^{\rm frozen}_t\\ 
    \theta^{\rm frozen}_{t+1} &= \eta_{0} \theta^{\rm frozen}_t - \eta_{1} \tilde \theta^{\rm frozen}_t.
\end{split}\]
The state evolution for this algorithm is given in Appendix \ref{sec:reparm}.

(2) We prove that if $m^\top_{j,t}$, $l^\top_{j,t}$ and $\bar a_t$ are the almost
sure limits of $\frac{\mu_j^\top \theta_t}{d}$, $\frac{e_{B_j}^\top \hat h_t}{d}$
and $a_t$ respectively, then for a suitable class of test functions, say
represented by $\phi: \RR^{n \times (L \cdot T)} \times \RR^n \times \RR^{L'
\cdot T} \to \RR^n$ and $\varphi: \RR^{d \times (L \cdot T)} \times \RR^{L'
\cdot T} \to \RR^n$, we have that both, 
    \[\lim_{n \to \infty} \frac{1}{n} \sum_{i=1}^n \phi(h_1, \dots, h_T, y, a_1, \dots, a_T)_i -
    \frac{1}{n} \sum_{i=1}^n \phi(h^{\rm frozen}_1, \dots, h^{\rm frozen}_T, y, \bar
    a_1, \dots, \bar a_T)_i = 0\]
    and 
    \[\lim_{n \to \infty} \frac{1}{d} \sum_{i=1}^d \varphi(\tilde \theta_1, \dots,
    \tilde \theta_T, a_1, \dots, a_T)_i -
    \frac{1}{d} \sum_{i=1}^d \varphi(\tilde \theta^{\rm frozen}_1, \dots, \tilde \theta^{\rm frozen}_T, \bar
    a_1, \dots, \bar a_T)_i = 0,\]
    almost surely.
This is the content of Appendix \ref{sec:lip}.

(3) We derive the desired almost sure limits $m_{j,t}$, $l_{j,t}$ and $\bar a_t$ in Appendix \ref{sec:almostSureLims}. Combining these three steps proves Lemma \ref{lem:ddState}.

\subsubsection{A Re-parameterized State Evolution}\label{sec:reparm}

For this subsection, we drop the superscript {\it frozen} for notational
convenience. We continue with our analysis of \eqref{eq:frozenDd}, consider the
change of variables
\[\label{eq:reparam}\begin{split} z_t = \alpha^{-1/2} \left(h_t - 
    \sum_{j=1}^J e_{B_j} m_{j,t}^\top\right),& \quad h_t = \alpha^{1/2} z_t + 
    \sum_{j=1}^J e_{B_j} m_{j,t}^\top\\
    w_t = \alpha^{-1/2}\left(\tilde \theta_t -  \alpha \sum_{j=1}^J
    \mu_j l_{j,t}^\top \right),& \quad \tilde \theta_t = \alpha^{1/2} w_t + 
    \alpha \sum_{j=1}^J \mu_j l_{j,t}^\top . \end{split}\]

Then, to analyze algorithm \eqref{eq:frozenDd}, it suffices to analyze the following algorithm and undo the above change of variables,
\[\label{eq:frozenAmp}\begin{split}
    z_t &= \tilde Z\theta_t + \alpha^{-1/2} \eta_{1} \sum_{s=1}^{t-1} \eta_{0}^{(t-1)-s} \hat h_{s}\\
    \hat h_t &= g\left(\alpha^{1/2} z_t +  \sum_{j=1}^J e_{B_j} m^\top_{j,t}, y, \bar a_t\right)\\
    w_t &= \tilde Z^\top \hat h_t - \alpha^{1/2}  \left(\frac{1}{n} \sum_{i=1}^n
    \nabla_h g\left(\left(\alpha^{1/2} w_{t} + \alpha \sum_{j=1}^J \mu_j l_{j,t}^\top\right)_i, y_i, \bar a_t\right) \right)\theta_t\\ 
    \theta_{t+1} &= \eta_{0}^t \theta_1 - \eta_{1} \sum_{s=1}^t \eta_{ 0}^{t-s} \left( \alpha^{1/2} w_s +  \alpha \sum_{j=1}^J \mu_j l_{j,s}^\top \right),\end{split}\]
where the final equality comes from unrolling the recursion of $\theta_{t+1} = \eta_{ 0} \theta_t - \eta_{ 1} \tilde \theta_t$ and changing variables. Then, we denote
\[\label{eq:GF}\begin{split}
    G(z_t) &= g\Bigg(\alpha^{1/2} z_t + \sum_{j=1}^J e_{B_j} m_{j,t}^\top, y, \bar a_t\Bigg)\\
    F_t(w_1, \dots, w_t) &= \eta_{ 0}^t \theta_1 - \eta_{ 1} \sum_{s=1}^t \eta_{ 0}^{t-s} \Bigg(\alpha^{1/2} w_s + \alpha \sum_{j=1}^J \mu_j l_{j,s}^\top\Bigg)
\end{split}\]

Now, treating $y, \mu_1, \dots, \mu_J, e_{B_1}, \dots, e_{B_j}, \theta_1$ as fixed
vectors, we can equivalently write \eqref{eq:frozenAmp} as,
\[\label{eq:frozenAmpSimp}\begin{split}
    z_t &= \tilde Z\theta_t - \sum_{s=1}^t \hat h_s \bB_{t,s}^\top \\
    \hat h_t &= G(z_t)\\
    w_t &= \tilde Z^\top \hat h_t - \theta_t \bC_t^\top\\ 
    \theta_{t+1} &= F_t(w_t),\end{split}\]
where, denoting the derivative with respect to a specified variable $(\cdot)$ as $\partial_{(\cdot)}$,
\begin{align}
\bB_{t,s} &= \frac{1}{n} \sum_{i=1}^d \partial_{w_{s,i}} F_{t-1}(w_{1,i}, \dots, w_{t-1,i}) = \frac{1}{d} \alpha^{-1} \sum_{i=1}^d \partial_{w_{s,i}} F_{t-1}(w_{1,i}, \dots, w_{t-1,i})\\
&= - \eta_{ 1} \eta_{ 0}^{(t-1)-s} \alpha^{-1/2} \Id\\
\bC_t &= \frac{1}{n} \sum_{i=1}^n \partial_{z_{t,i}} G(z_{t,i}) = \alpha^{1/2} n^{-1} \sum_{i=1}^n \nabla_h g(\alpha^{1/2}z_{t,i} + m_{j(i), t}^\top, y_i, \bar a_t)^\top,
\end{align}
with $j(i)$ mapping a coordinate $i \in [n]$ to its corresponding block index of $B_1,
\dots, B_J$. Note, we can immediately check that $G$ is Lipschitz as $g$ is
Lipschitz in $z_t$ (see Assumption
\ref{as:main0} \eqref{as:algo}) where, 
\begin{align}
    \|G(z_t) - G(\tilde z_t)\|_{\Fro} &= \left\|g\left( \alpha^{1/2} z_t + \sum_{j=1}^J e_{B_j} m_{j,t}^\top, y, \bar a_t\right) - g\left( \alpha^{1/2} \tilde z_t + \sum_{j=1}^J e_{B_j} m_{j,t}^\top, y, \bar a_t\right)\right\|_{\Fro}\\
    & \leq L \|\alpha^{1/2}(z_t - \tilde z_t)\|_{\Fro}\\
    &\leq (L \alpha^{1/2})\|z_t - \tilde z_t\|_{\Fro},
\end{align}
and $F_t$ is clearly Lipschitz in $w_1, \dots, w_t$ as it is a linear function of $w_1, \dots, w_t$ with bounded coefficients.

We can then see that algorithm \eqref{eq:frozenAmpSimp} is a full-history
matrix-valued extension of the non-separable AMP theory from \cite{berthier2017}
(or see \cite{lovig2025} for a full-history vector-valued AMP algorithm). Using
the standard construction of reformatting the vector amp iterates  (temporarily overloading the definition of the iterates $z_t, w_t$) $(z_1, w_1), \dots,
(z_{t\cdot L}, w_{t \cdot L}) \in \RR^{2 \times n}$ into matrix amp iterates $z^{\rm mat}_t$ and $w^{\rm mat}_t$ with
\[z^{\rm mat}_t = [z_{(t-1)+1}, \dots, z_{t \cdot  L - 1}], \quad w^{\rm mat}_t
= [w_{(t-1)+1}, \dots, w_{t \cdot  L - 1}],\] we can conclude that
\eqref{eq:frozenAmpSimp} uses the correct Onsager correction terms and is a valid
rectangular AMP algorithm. We leave further details of matrix valued AMP
iterates to the survey \cite[Section 6.7]{feng2021} and \cite{javanmard12}. We
then have the following low dimensional prescription of the iterates $(z_t,
w_t)_{t \in [T]}$ from \eqref{eq:frozenAmpSimp}.

\begin{definition}\label{def:frozenAmpState}
    Given the functions $G,F_t$ from \eqref{eq:GF} and deterministic vector $\theta_1$, with initialization $\frac{\alpha^{-1}}{d}\theta_1^\top \theta_1 = \Omega^z_1$, we recursively define the sequence of matrices $\Sigma^w_1, \dots, \Sigma^w_T, \Omega^z_2, \dots, \Omega^z_{T+1}$ with
    \begin{align}
        \Sigma^w_t\ss{r,s} &= \frac{1}{n} \EE[G(Z_r)^\top G(Z_s)], \quad r,s \in [t]\\
        &(Z_r, Z_s) \sim \cN\left(0, \begin{bmatrix}\Omega^z_t\ss{r,r} & \Omega^z_t\ss{r,s} \\ \Omega^z_t\ss{s,r} & \Omega^z_t\ss{s,s}\end{bmatrix} \otimes \Id_n \right)\\
        \Omega^z_{t+1}\ss{r+1,s+1} &= \frac{1}{n} \EE[F_r(W_1, \dots, W_r)^\top F_s(W_1, \dots, W_s)], \quad r,s \in [t]\\ 
        &= \frac{\alpha^{-1}}{d} \EE[F_r(W_1, \dots, W_r)^\top F_s(W_1, \dots, W_s)]\\
        \Omega^z_{t+1}\ss{1,s+1} &= \frac{\alpha^{-1}}{d} \theta_1^\top \EE[F_s(W_1, \dots, W_s)] = \Omega^z_{t+1}\ss{s+1,1}^\top, \quad s \in [t]\\
        \Omega^z_{t+1}\ss{1,1} &= \Omega^z_1\\
        &(W_1, \dots, W_t) \sim \cN\left(0, \Sigma^w_t \otimes \Id_d \right).
    \end{align}
\end{definition}

\begin{lemma}[\cite{lovig2025}, Theorem 3.3  (extended to matrix-valued AMP
    algorithms)]\label{lem:appSE} Let $\phi_1,\phi_2: \RR^{n
    \times (L \cdot T)} \to \RR^n, \varphi_1, \varphi_2: \RR^{d \times (L \cdot
    T)} \to \RR^d$ each be uniformly Lipschitz function in $n,d$. Let $(z_t,
    w_t)$ be the iteration \eqref{eq:frozenAmpSimp}, and denote $(Z_1, \dots, Z_T) \sim
    \cN(0,\Omega^z_T \otimes \Id_n)$, $(W_1, \dots, W_T) \sim \cN(0, \Sigma^w_T
    \otimes \Id_d)$ from Definition \ref{def:frozenAmpState}. If $\max_{t \in [T]}\max(\|\Sigma^w_t\|_{\op}, \|(\Sigma^w_t)^{-1}\|_{\op}, \|\Omega^z_t\|_{\op}, \|(\Omega^z_t)^{-1}\|_{\op})$ is bounded almost surely, then the following
    limits hold almost surely, 
    \begin{align}
        \lim_{n \to \infty} \frac{1}{n} \sum_{i=1}^n \phi(z_1, \dots, z_T)_i - \lim_{n \to \infty} \frac{1}{n} \sum_{i=1}^n \EE[\phi(Z_1, \dots, Z_T)_i] &= 0\\
        \lim_{d \to \infty} \frac{1}{d} \sum_{i=1}^d \varphi(w_1, \dots, w_T)_i - \lim_{d \to \infty} \frac{1}{d} \sum_{i=1}^d \EE[\varphi(W_1, \dots, W_T)_i] &= 0,
    \end{align}
    where $\phi(\cdot)_i = \phi_1(\cdot)_i\phi_2(\cdot)_i$ and
    $\varphi(\cdot)_i = \varphi_1(\cdot)_i\varphi_2(\cdot)_i$.
\end{lemma}

\begin{definition}\label{def:frozenDdState} 
Let Assumption \ref{as:main0} hold with $(\PP_j)_{j \in [J]}$ from Assumption \ref{as:main0} \eqref{as:data}, $\bar \theta^2$ from Assumption \ref{as:main0} \eqref{as:limits}, and $g$ be from Assumption
\ref{as:main0} \eqref{as:algo}.
    
Given initialization $\check \Omega_1 = \bar \theta^2$, sequences $(m_{j,t})_{j \in [J], t \in [T]}$, $(l_{j,t})_{j \in [J], t \in [T]}$ and $(\bar a_t)_{t \in [T]}$, we recursively define
the sequence of matrices $\check \Sigma_1, \dots, \check \Sigma_T, \check
\Omega_2, \dots, \check \Omega_{T+1}$ with
        \begin{align}
        \check \Sigma_t\ss{r,s} &= \sum_{j=1}^J p_j \EE[g(\check Z_{r} + m_{j,r}, Y_j, \bar a_r) g(\check Z_{s} + m_{j,s}, Y_j, \bar a_s)^\top ], \quad r,s \in [t]\\
        &(\check Z_r, \check Z_s) \sim \cN\left(0, \begin{bmatrix}\check \Omega_t\ss{r,r} & \check \Omega_t\ss{r,s} \\ \check \Omega_t\ss{s,r} & \check \Omega_t\ss{s,s}\end{bmatrix}\right)\\
        \check \Xi_{t}\ss{r+1, s} &= \eta_{ 0} \check \Xi_{t-1}\ss{r,s} -
    \eta_{ 1} \Big(\alpha \check \Sigma_{t}\ss{r,s} +
    \alpha^2 \sum_{j,k=1}^J \chi_{j,k} l_{j,r} l_{k,s}^\top \Big), \quad r \in [t-1], s \in [t]\\
        \check \Xi_t\ss{1,s} &=  \alpha \sum_{j=1}^J
    m_{j,1} l_{j,s}^\top, \quad s \in [t]\\
        \check \Omega_{t+1}\ss{r+1,s+1} &= \eta_{ 0}^2 \check \Omega_{t}\ss{r,s} \-
    \eta_{ 0} \eta_{1}\left(\check \Xi_{t}\ss{r,s} \+ \check \Xi_{t}\ss{s,r}^\top\right) \+
    \eta_{ 1}^2\Big(\alpha \check \Sigma_{t}\ss{r,s} \+ \alpha^2\!\!\!
    \sum_{j,k=1}^J\!\!\chi_{j,k} l_{j,r} l_{k,s}^\top \Big)\\ 
    &\quad r,s \in [t]\\
    \check \Omega_{t+1}\ss{1, s+1} &= \eta_{ 0}
    \check \Omega_t\ss{1,s} - \eta_{1} \check \Xi_t\ss{1,s} = \check \Omega_{t+1}\ss{s+1,1}^\top, \quad s \in [t]\\ 
    \check \Omega_{t+1}\ss{1,1} &= \bar \theta^2,
    \end{align}
\end{definition}

\begin{corollary}\label{cor:appSe} Let $(h_t, \hat h_t, \tilde \theta_t, \theta_t)$
    be the iteration from \eqref{eq:frozenDd} and let Assumption \ref{as:main0}
    hold. The following limits hold almost surely, for $j \in [J], r,s,t \in [T]$,
        \begin{align}
        \lim_{d \to \infty} \bigg|\frac{1}{d} \mu_j^\top \theta_t &- \frac{1}{d} \mu_j^\top \EE[(\alpha^{1/2} \check W_t + \alpha \sum_{j=1}^J \mu_j l^\top_{j,t})]\bigg| = 0\\
        \lim_{d \to \infty} \frac{1}{d} \theta_r^\top \theta_s &= \check \Omega_T\ss{r,s}\\
        \lim_{d \to \infty} \frac{1}{d} \theta_r^\top \tilde \theta_s &= \check \Xi\ss{r,s}\\
        \lim_{n \to \infty} \frac{1}{n} \hat h_r^\top \hat h_s &= \check \Sigma_T\ss{r,s},
    \end{align}
    where $(\check W_1, \dots, \check
    W_T) \sim \cN(0, \check \Sigma_T \otimes \Id_d)$ from Definition
    \ref{def:frozenDdState} and $Y_j \sim \PP_j$.

    Moreover, if $\bar \phi(h,y,a,j): \RR^L \times \RR \times \RR^{L'} \times [J] \to
    \RR$ is a Lipschitz function in $h$ uniformly over $y,a,j$, then
    the following limit also holds almost surely, 
    \[\lim_{n,d \to \infty}\frac{1}{n} \sum_{i=1}^n \bar \phi(h_{t,i}, y_i, \bar a_t, j(i)) =
    \sum_{j \in [J]} p_j \EE[\bar \phi(m_{j,t}
    + \check Z_t, Y_j, \bar a_t, j)],\]
    where $j(i)$ maps coordinate $i$ to its block $B_j$ with $i \in B_j$ and $(\check Z_1, \dots, \check Z_T) \sim \cN(0,\check \Omega_T)$ from Definition \ref{def:frozenDdState}.
\end{corollary}

\begin{proof}
    First, we prove the following inductive claim, for all $t \in [T]$, there
    exists almost sure limits $\tilde \Omega^z_t, \tilde \Sigma^w_t$ of
    $\Omega^z_t, \Sigma^w_t$ as $d \to \infty$ which are recursively defined by
    the following system initialized at $\tilde \Omega_1^z =
    \alpha^{-1}\bar \theta^2$ (with $\bar \theta^2$ from Assumption \ref{as:main0})
    where $Y_j \sim \PP_j$,
    \[\label{eq:tildeState}\begin{split}
        \tilde \Sigma^w_t\ss{r,s} &= \sum_{j=1}^J p_j \EE[g(\alpha^{1/2} \tilde Z_{r} + m_{j,r}, Y_j, \bar a_r) g(\alpha^{1/2} \tilde Z_{s} + m_{j,s}, Y_j, \bar a_s)^\top ], \quad r,s \in [t]\\
        &(\tilde Z_r, \tilde Z_s) \sim \cN\left(0, \begin{bmatrix}\tilde
        \Omega^z_t\ss{r,r} & \tilde \Omega^z_t\ss{r,s} \\ \tilde
        \Omega^z_t\ss{s,r} & \tilde \Omega^z_t\ss{s,s}\end{bmatrix}\right)\\
        \tilde \Xi_{t}\ss{r+1, s} &= \eta_{ 0} \tilde \Xi_{t-1}\ss{r,s} -
    \eta_{ 1} \Big(\alpha \tilde \Sigma^w_{t}\ss{r,s} +
    \alpha^2 \sum_{j,k=1}^J \chi_{j,k} l_{j,r} l_{k,s}^\top \Big), \quad r, s \in [t]\\
        \tilde \Xi_t\ss{1,s} &= \alpha \sum_{j=1}^J
    m_{j,1} l_{j,s}^\top, \quad s \in [t]\\
        \tilde \Omega^z_{t+1}\ss{r+1,s+1} &= \tilde \Omega^z_{t}\ss{r,s} \-
    \alpha^{-1} \tilde \Xi_{t}\ss{r,s} \- \alpha^{-1} \tilde \Xi_{t}\ss{s,r}^\top \+
    \alpha^{-1}\!\Big(\!\alpha \tilde \Sigma^w_{t}\ss{r,s} \+ \alpha^2\!\!
    \sum_{j,k=1}^J\!\chi_{j,k} l_{j,r} l_{k,s}^\top\!\Big),\\ 
    &\quad r,s \in [t]\\ 
    \tilde \Omega_{t+1}^z\ss{1, s+1} &= \eta_{ 0}
    \tilde \Omega_t^z\ss{1,s} - \alpha^{-1}\eta_{1} \tilde \Xi_t\ss{1,s} = \tilde
    \Omega_{t+1}\ss{s+1,1}^\top, \quad s \in [t]\\ 
    \tilde \Omega^z_{t+1}\ss{1,1} &= \alpha^{-1} \bar \theta^2.
    \end{split}\]

    We proceed by induction. For each $t \in [T]$, we first prove the almost sure limit of $\Omega^z_t$ and then the almost sure limit of $\Sigma^w_t$, deriving the limits of intermediate variables in Definition \ref{def:frozenAmpState} between these steps. The base case follows immediately as $\lim_{d \to
    \infty} \Omega^z_1 = \lim_{d \to \infty} \alpha^{-1} \frac{\theta_1^\top \theta_1}{d} = \alpha^{-1}
    \bar \theta^2$ by Assumption \ref{as:main0}. Then, with $Z_1 \sim
    \cN(0, \Omega^z_1 \otimes \Id_n)$, we write that 
    \begin{align}
    \Sigma^w_1 &= \frac{1}{n} \EE\left[g\left(\alpha^{1/2} Z_1 + \sum_{j=1}^J e_{B_j} m_{j,1}^\top, y, \bar a_1\right)^\top g\left(\alpha^{1/2} Z_1 + \sum_{j=1}^J e_{B_j} m_{j,1}^\top, y, \bar a_1\right)\right]\\
    &= \sum_{j=1}^J \frac{|B_j|}{n} \frac{1}{|B_j|} \sum_{i \in B_j} \EE[g(\alpha^{1/2} Z_{1,i}^\top +  m_{j(i),1}, y_i, \bar a_1) g(\alpha^{1/2} Z_{1,i}^\top + m_{j(i),1}, y, \bar a_1)^\top].
    \end{align}
    Then, applying the strong law of large numbers on the set of coordinates $i \in
    B_j$ for the random variables $y_i \sim \PP_j$, the continuous mapping
    theorem for the almost sure limit $\lim_{n \to \infty} |B_j|/n = p_j$, and the continuous mapping theorem applied to
    the variance $Z_{1,i}$ (i.e. $\Omega_1^z$ with almost sure limit $\tilde
    \Omega^z_1$) using that $g$ is a continuous function, gives that the almost sure limit of $\Sigma_1^w$ is
    \[\tilde \Sigma^w_1 = \sum_{j=1}^J p_j \EE[g(\alpha^{1/2} \tilde Z_{1} +
    m_{j,1}, Y_j, \bar a_1) g(\alpha^{1/2} \tilde Z_{1} + 
    m_{j,1}, Y_j, \bar a_1)^\top],\] where $\tilde Z_1 \sim \cN(0, \tilde
    \Omega^z_1)$.

    Now, we assume the inductive claim holds up to time $t-1$, i.e. the
    matrices $\tilde \Omega^z_{t-1}$ and $\tilde \Sigma^w_{t-1}$ are the almost
    sure limits of the matrices $\Omega^z_{t-1}$ and $\Sigma^w_{t-1}$ respectively.

    Recall the choice of functions $G,F_t$ from \eqref{eq:GF}, we can then
    rewrite $\Omega^z_{t}$, with $r,s \in [t-1]$, as
    {\footnotesize 
    \begin{align}
        \Omega^z_{t}\ss{r+1, s+1} &= \frac{\alpha^{-1}}{d} \EE\Bigg[\!\!\left(\!\eta_{0}^r \theta_1 \- \eta_{ 1}\! \sum_{p=1}^r \!\eta_{ 0}^{r-p}\! \left(\!\alpha^{1/2} W_p \+ \alpha\!\sum_{j=1}^J\!\mu_j l_{j,p}^\top\!\right)\!\!\right)^{\!\!\!\top}\!\!\!\left(\!\eta_{ 0}^s \theta_1 \- \eta_{ 1}\!\sum_{q=1}^s\!\eta_{ 0}^{s-q} \!\left(\!\alpha^{1/2} W_q \+ \alpha\!\sum_{k=1}^J\!\mu_k l_{k,q}^\top\!\right)\!\!\right)\!\!\Bigg]\\
        &= \frac{\alpha^{-1}}{d} \EE\Bigg[\!\Bigg\{\!\eta_{ 0}\underbrace{\left(\!\eta_{ 0}^{r-1} \theta_1 \- \eta_{ 1}\! \sum_{p=1}^{r-1}\! \eta_{ 0}^{(r - 1)-p}\!\left(\!\alpha^{1/2} W_p \+ \alpha\!\sum_{j=1}^J\!\mu_j l_{j,p}^\top\! \right)\!\right)}_{\cA_{r-1}} \- \eta_{ 1}\!  \underbrace{\left(\!\alpha^{1/2} W_r \+  \alpha\!\sum_{j=1}^J\! \mu_j l_{j,r}^\top \right)}_{\cB_r}\!\! \Bigg\}^\top\\ 
        &\times \Bigg\{\!\eta_{0}\!\underbrace{\left(\!\eta_{ 0}^{s-1} \theta_1 \- \eta_{ 1}\! \sum_{q=1}^{s-1}\! \eta_{ 0}^{(s-1)-q}\! \left(\alpha^{1/2} W_q \+  \alpha\!\sum_{k=1}^J\! \mu_k l_{k,q}^\top\!\right)\!\right)}_{\cA_{s-1}} \- \eta_{ 1}\! \underbrace{\left(\!\alpha^{1/2} W_s \+  \alpha\! \sum_{k=1}^J\! \mu_k l_{k,s}^\top \!\right)}_{\cB_s}\!\Bigg\} \!\Bigg]\\
        &= \alpha^{-1}\frac{1}{d} \left(\EE[\eta_{ 0}^2 \cA_{r-1}^\top \cA_{s-1} - \eta_{ 0}\eta_{ 1}(\cA_{r-1}^\top \cB_s + \cB_r^\top \cA_{s-1}) + \eta_{ 1}^2 \cB_r^\top \cB_s] \right)
    \end{align}
    }
    where $(W_1, \dots, W_{t-1}) \sim \cN(0, \Sigma^w_{t-1} \otimes \Id_d)$.
    Moreover, we have the following simplifications, 
    \begin{align}
        \frac{\alpha^{-1}}{d}\EE[\cA_{r-1}^\top \cA_{s-1}] &= \Omega^z_{t-1}\ss{r,s},\tag{As $\cA_{r-1}$ is $F_{r-1}(W_1, \dots, W_r)$ and similarly for $\cA_{s-1}$}\\
        \frac{1}{d} \EE[\cA_{r-1}^\top \cB_s] &= \frac{1}{d} \EE\left[\left(\eta^{r-1}_{0} \theta_1 - \eta_{1} \sum_{p=1}^{r-1} \eta^{(r-1)-p}_{0} \cB_p \right)^\top \cB_s\right],\\
        \frac{1}{d} \EE[\cB_r^\top \cA_{s-1}] &= \frac{1}{d} \EE\left[\cB_r^\top \left(\eta^{s-1}_{0} \theta_1 - \eta_{1} \sum_{q=1}^{s-1} \eta^{(s-1)-q}_{0} \cB_q \right)\right],\\
        \frac{1}{d} \EE[\cB_r^\top \cB_s] &= \frac{1}{d}\!\!\left(\!\!\alpha \EE[W_r^\top W_s] \+ \alpha^{3/2}\! \sum_{k=1}^J \EE[W_r]^\top \mu_k l_{k,r}^\top \+ \alpha^{3/2}\! \sum_{j=1}^J l_{j,r} \mu_j^\top \EE[W_s] \+ \alpha^2\!\! \sum_{j,k=1}^J\!\! l_{j,r} \mu_j^\top \mu_k l_{k,s}^\top \!\right)\\
        &= \alpha \frac{\EE[W_r^\top W_s]}{d} + \alpha^2 \sum_{j,k=1}^J l_{j,r} \frac{\mu_j^\top \mu_k}{d} l_{k,s}^\top\\
        &= \alpha \Sigma^w_{t-1}\ss{r,s} + \alpha^2 \sum_{j,k=1}^J l_{j,r} \frac{\mu_j^\top \mu_k}{d} l_{k,s}^\top\\
    \end{align}
    as $\EE[W_r] = 0$. This in turn implies that,
    \[\label{eq:shortTildeOmega}\begin{split}
        \Omega^z_{t}\ss{r+1, s+1} &= \eta_{ 0}^2\Omega^z_{t-1}\ss{r,s} - \alpha^{-1}\eta_{ 0}\eta_{ 1}(\Xi^z_{t-1}\ss{r,s} + \Xi^z_{t-1}\ss{s,r}^\top)\\
        &\qquad + \alpha^{-1} \eta^2_{ 1}\left(\alpha \Sigma^w_{t-1}\ss{r,s} + \alpha^2 \sum_{j,k=1}^J l_{j,r} \frac{\mu_j^\top \mu_k}{d} l_{k,s}^\top \right),
    \end{split}\]
    where, with initialization $\Xi^z_{t} \ss{1,s} = \frac{1}{d} \theta_1^\top
    \EE[\cB_s]$ and $r \in [t-2], s \in [t-1]$, we recursively define,
    \begin{align}
        \Xi^z_{t-1} \ss{r+1,s} &= \frac{1}{d} \EE\left[\left( \eta_{0} \theta_1 - \eta_{1} \sum_{p=1}^{r} \eta^{r-p}_{\theta,0} \cB_p \right)^\top \cB_s\right]\\
        &= \frac{1}{d} \EE\left[\left(\eta_{ 0}\left(\eta_{0} \theta_1 - \eta_{1} \sum_{p=1}^{r-1} \eta^{(r-1)-p}_{\theta,0} \cB_p \right)- \eta_{1} \cB_r\right)^\top \cB_s\right]\\
        &= \eta_{0}\Xi^z_{t-2}\ss{r,s} - \eta_{ 1} \frac{1}{d} \EE\left[\cB_r^\top \cB_s\right]\\
        &= \eta_{0}\Xi^z_{t-2}\ss{r,s} - \eta_{1}\left(\alpha \Sigma_{t-1}^w\ss{r,s} +  \alpha^2 \sum_{j,k=1}^J l_{j,r} \frac{\mu_j^\top \mu_k}{d} l_{k,s}^\top \right).
    \end{align}

    By the inductive hypothesis, we have the covariance matrix $\Sigma^w_{t-1}$
    has the almost sure limit $\tilde \Sigma^w_{t-1}$ and the covariance matrix
    $\Omega^z_{t-1}$ has the almost sure limit $\tilde \Omega^z_{t-1}$. Thus, under equation \eqref{eq:shortTildeOmega}, we
    apply the strong law of large numbers on the coordinates $i \in B_j$ for the
    independent random variables $y_i \sim \PP_j$, the continuous mapping
    theorem for the almost sure limit $\lim_{n \to \infty} |B_j|/n = p_j$, limit $\lim_{d \to
    \infty} \frac{\mu_j^\top \mu_k}{d} = \chi_{j,k}$ from Assumption
    \ref{as:main0} \eqref{as:data} and limit $\lim_{d \to \infty} \frac{\theta_{1}^\top
    \mu_j}{d} = m_{j,1}$ from Assumption \ref{as:main0} \eqref{as:limits}
    to give that $\Omega^z_{t}$ has the following almost sure limit, with $r,s \in [t-1]$,
    \begin{align}
        \tilde \Omega^z_{t}\ss{r+1, s+1} &= \eta_{ 0}^2 \tilde \Omega^z_{t-1}\ss{r,s} - \alpha^{-1}\eta_{ 0}\eta_{ 1}(\tilde \Xi_{t-1}\ss{r,s} + \tilde \Xi_{t-1}\ss{s,r}^\top)\\
        &\qquad + \alpha^{-1} \eta^2_{ 1}\left(\alpha \tilde \Sigma^w_{t-1}\ss{r,s} + \alpha^2 \sum_{j,k=1}^J  \chi_{j,k} l_{j,r} l_{k,s}^\top \right),
    \end{align}
    where we recursively define, with $r \in [t-2]$ and $s \in [t-1]$,
    \[\tilde \Xi_{t-1}\ss{r+1, s} = \eta_{ 0} \tilde \Xi_{t-2}\ss{r,s} -
    \eta_{ 1} \left(\alpha \tilde \Sigma^w_{t-1}\ss{r,s} + 
    \alpha^2 \sum_{j,k=1}^J \chi_{j,k}  l_{j,r} l_{k,s}^\top \right),\]
    with $\tilde \Xi_{t-1}\ss{1,s} = \alpha \sum_{j=1}^J
    m_{j,1} l_{j,s}^\top$. 

    Moreover, we can write
    \begin{align}\Omega_t^z\ss{1,s+1} &= \frac{\alpha^{-1}}{d} \theta_1^\top \EE\Bigg[\eta_{0} \left(\eta_{ 0}^{s-1} \theta_1 - \eta_{ 1} \sum_{q=1}^{s-1} \eta_{ 0}^{(s-1)-q} \left(\alpha^{1/2} W_q + \alpha \sum_{k=1}^J \mu_k l_{k,q}^\top \right) \right)\\ 
        &\qquad- \eta_{ 1} \left(\alpha^{1/2} W_s + \alpha \sum_{k=1}^J \mu_k l_{k,s}^\top \right)\Bigg]\\
        &= \eta_{ 0} \Omega^z_{t-1}\ss{1,s} - \alpha^{-1} \eta_{ 1} \Xi^z _{t-1}\ss{1,s}
    \end{align} as $\EE[W_q] = 0$, which can be argued identically as above to have
    the almost sure limit, 
    \[\tilde \Omega_t^z\ss{1, s+1} = \eta_{ 0} \tilde \Omega_{t-1}^z\ss{1,s} -
    \alpha^{-1} \eta_{ 1} \tilde \Xi_{t-1}\ss{1,s}.\]
    As a final step, we can use the aforementioned argument on the
    convergence of $\Omega^z_1$ to give that $\Omega^z_t\ss{1,1}$ has the almost
    sure limit $\tilde \Omega^z_t\ss{1,1}= \tilde \Omega^z_1\ss{1,1}$.

    Then, we can write $\Sigma^w_t$ as,
    \begin{align}
        \Sigma^w_t\ss{r,s} &= \frac{1}{n} \EE\left[g\left(\alpha^{1/2} Z_r + \sum_{j=1}^J e_{B_j} m_{j,r}^\top, y, \bar a_r\right) g\left(\alpha^{1/2} Z_s + \sum_{j=1}^J e_{B_j} m_{j,s}^\top, y, \bar a_s\right)\right]\\
        &= \sum_{j=1}^J \frac{|B_j|}{n} \frac{1}{|B_j|} \sum_{i \in B_j} \EE[g(\alpha^{1/2} Z_{r,i} + m_{j(i),r}, y_i, \bar a_r)^\top g(\alpha^{1/2} Z_{s,i} + m_{j(i),s}, y_i, \bar a_s)^\top]
    \end{align}
    where $(Z_1, \dots, Z_t) \sim \cN(0, \Omega^z_t \otimes \Id_n)$.

    An application of the strong law of large numbers over the block $B_j$ where each
    $y_i \sim \PP_j$ independently and the continuous mapping theorem on the
    covariance of $Z_{r,i}$ and $Z_{s,i}$, using that $g$ is a continuous
    function, then gives that $\Sigma^w_t$ has the almost sure limit,
    \[\tilde \Sigma^w_t \ss{r,s} = \sum_{j=1}^J p_j \EE[g(\alpha^{1/2} \tilde Z_{r} + m_{j,r}, Y_j, \bar a_r) g(\alpha^{1/2} \tilde Z_{s} + m_{j,s}, Y_j, \bar a_s)^\top ],\]
    where $(\tilde Z_r, \tilde Z_s) \sim \cN\left(0, \begin{bmatrix}\tilde
    \Omega^z_t\ss{r,r} & \tilde \Omega^z_t\ss{r,s} \\ \tilde \Omega^z_t\ss{s,r}
    & \tilde \Omega^z_t\ss{s,s}\end{bmatrix}\right)$.
    This proves the inductive claim.

    Recalling system \eqref{eq:tildeState}, we can relate this system back to
    Definition \ref{def:frozenDdState} by considering the change of variables
    $\check Z_t = \alpha^{1/2} \tilde Z_t$, $\check \Omega_{t+1} = \alpha \tilde
    \Omega^z_{t+1}, \check \Xi_t = \tilde \Xi_{t}$ and $\check \Sigma_t = \tilde
    \Sigma^w_t$. Recall the reparameterization \eqref{eq:reparam}, where, for
    any two functions $\phi, \varphi$ given in Lemma \ref{lem:appSE}, we can
    write
    \[\frac{1}{n}\sum_{i=1}^n \phi(h_1, \dots, h_T)_i =  \frac{1}{n}\sum_{i=1}^n \phi\left(\alpha^{1/2} z_1 + \sum_{j=1}^J e_{B_j} m_{j,1}^\top, \dots, \alpha^{1/2} z_T + \sum_{j=1}^J e_{B_j} m_{j,T}^\top\right)_i,\]
    or 
    \[\frac{1}{d} \sum_{i=1}^d \varphi(\tilde \theta_1, \dots, \tilde \theta_T) =
    \frac{1}{d} \sum_{i=1}^d \varphi\left(\alpha^{1/2} w_1 +  \alpha
    \sum_{j=1}^J \mu_j l_{j,1}^\top, \dots, \alpha^{1/2} w_T +  \alpha
    \sum_{j=1}^J \mu_j l_{j,T}^\top\right)_i.\]

    Notice, if both $\phi_1, \phi_2$ and $\varphi_1, \varphi_2$ are $L$-Lipschitz in their first argument
    then the following functions are $L\alpha^{1/2}$ Lipschitz with
    respect to either $z_1, \dots, z_T$ or $w_1, \dots, w_T$, for both $i \in \{1,2\}$,
    \begin{align}
        (z_1, \dots, z_T) \mapsto \phi_i\left(\alpha^{1/2}z_1 + \sum_{j=1}^J e_{B_j} m_{j,1}^\top, \dots, \alpha^{1/2} z_T + \sum_{j=1}^J e_{B_j} m_{j,T}^\top\right)\\
        (w_1, \dots, w_T) \mapsto \varphi_i\left(\alpha^{1/2} w_1 + \alpha \sum_{j=1}^J \mu_j l_{j,1}^\top, \dots, \alpha^{1/2} w_T + \alpha \sum_{j=1}^J \mu_j l_{j,T}^\top\right).
    \end{align}

    Therefore, by Lemma \ref{lem:appSE} and the aforementioned change of variables between system \eqref{eq:tildeState} and Definition \ref{def:frozenDdState}, we have almost surely that 
    \[\lim_{n \to \infty} \frac{1}{n}\sum_{i=1}^n \phi(h_1, \dots, h_T)_i -
    \lim_{n \to \infty} \frac{1}{n}\sum_{i=1}^n
    \phi\left(\check Z_1 +  \sum_{j=1}^J e_{B_j} m_{j,1}^\top, \dots,
    \check Z_T +  \sum_{j=1}^J e_{B_j} m_{j,T}^\top\right)_i = 0,\label{eq:phiLim}\]
    or 
    \[\lim_{d \to \infty} \frac{1}{d} \sum_{i=1}^d \varphi(\tilde \theta_1, \dots,
    \tilde \theta_T)_i \- \lim_{d \to \infty} \frac{1}{d} \sum_{i=1}^d
    \varphi\!\left(\alpha^{1/2} \check W_1 \+  \alpha\! \sum_{j=1}^J\! \mu_j
    l_{j,1}^\top, \dots, \alpha^{1/2} W_T \+  \alpha \sum_{j=1}^J \mu_j
    l_{j,T}^\top\!\right)_{\!\!i} \= 0,\label{eq:varPhiLim}\] where $(\check Z_1, \dots, \check Z_T) \sim
    \cN(0, \check \Omega_T \otimes \Id_n)$ and $(\check W_1, \dots, \check W_T)
    \sim \cN(0, \check \Sigma_T \otimes \Id_d)$.

    Each of the limits in the statement of the lemma follow immediately from the
    following choice of $\phi$ and $\varphi$, using equations \eqref{eq:phiLim} or \eqref{eq:varPhiLim} and
    using the identicality of the row distribution of $\check Z_1, \dots, \check
    Z_T$ and $\check W_1, \dots, \check W_T$ (recalling that $\varphi(\cdot)_i =
    \varphi_{1}(\cdot)_i\varphi_{2}(\cdot)_i$ and similarly for $\phi$):
    \begin{align}
        \frac{1}{d} \mu_j^\top \tilde \theta_t,& \quad \varphi(\tilde \theta_1, \dots, \tilde \theta_T)_i = \mu_{j,i} \tilde \theta_{t,i}\\
        \frac{1}{d} \theta_r^\top \theta_s,& \quad \varphi(\tilde \theta_1, \dots, \tilde \theta_T)_i = \left(\eta_{ 0}^r \theta_1 - \eta_{1} \sum_{p=1}^r \eta_{ 0}^{r-p} \tilde \theta_{p,i}\right)\left(\eta_{ 0}^s \theta_1 - \eta_{1} \sum_{q=1}^s \eta_{ 0}^{s-q} \tilde \theta_{q,i}\right)\\
        \frac{1}{d} \theta_r^\top \tilde \theta_s,& \quad \varphi(\tilde \theta_1, \dots, \tilde \theta_T)_i = \left(\eta_{ 0}^r \theta_1 - \eta_{1} \sum_{p=1}^r \eta_{ 0}^{r-p} \tilde \theta_{p,i}\right)\tilde \theta_{s}\\
        \frac{1}{n}\hat h_r^\top \hat h_s,& \quad \phi(h_1, \dots, h_T)_i = g(h_{r,i}, y_i, \bar a_r)g(h_{s,i}, y_i, \bar a_s)\label{eq:needYConc}
    \end{align}
    and, for the first three limits, invoking the law of large numbers on row-wise independent laws of $\check W_1, \dots, \check W_t, \check Z_1, \dots, \check Z_t$ to relate back to Definition
    \ref{def:frozenDdState}. Note, \eqref{eq:needYConc}, and more generally the
    statement on $\frac{1}{n} \sum_{i=1}^n \bar\phi(h_{t,i}, y_i, \bar a_t,j) =
    \sum_{j = 1}^J \frac{|B_j|}{n} \sum_{i \in B_j} \bar \phi(h_{t},
    y, \bar a_t,j)_i$ (where $\bar \phi$ is applied row-wise) in the statement of this corollary follows by selecting $(\phi_1)_i = \bar
    \phi(\cdot, y_i, \bar a_t, j)$ (or $g(\cdot, y_i, \bar a_t)$)
    for each $i \in [n]$ and $\phi_2$ being the all ones vector (or $g(\cdot,
    y_i, \bar a_t)$) and using equation \eqref{eq:phiLim}. Then an application of the continuous mapping theorem for the almost sure limit $\lim_{n \to \infty} |B_j| / n$, the identicality of the row distribution of $\check Z_1, \dots, \check Z_T$, alongside the strong law of large numbers applied
    to the block $B_j$ where each $y_i \sim \PP_j$ independently proves the
    claim.
\end{proof}

We conclude this section by observing that Definition \ref{def:frozenDdState}, with the exception of the updates on the frozen state evolution parameters, is equivalent Definition \ref{def:ddState}.

\subsubsection{Convergence of Lower-order Recursions Of $m_{j,t}, l_{j,t}$ And $\bar a_t$}\label{sec:lip}

Let $\tilde m_{j,t}^\top = \frac{\mu_j^\top \theta_t}{d}$ and $\tilde l_{j,t}^\top = \frac{e_{B_j}^\top \hat h_t}{d}$, consider the following algorithm equivalent to \eqref{eq:appdd},
\[\begin{split}
    h_t &= \sum_{j=1}^J e_{B_j} \tilde m_{j,t}^\top + Z\theta_t + \eta_{ 1} \sum_{s=1}^{t-1} \eta_{ 0}^{(t-1)-s} \hat h_{s}\\
    \hat h_t &= g(h_t, y, a_t)\\
    \tilde \theta_t &= \sum_{j=1}^J \alpha \mu_j \tilde l_{j,t}^\top +  Z^\top \hat h_t - \alpha  \left(\frac{1}{n} \sum_{i=1}^n \nabla_h g(h_{t,i}, y_i, a_t) \right)\theta_t\\ 
    \theta_{t+1} &= \eta_{ 0} \theta_t - \eta_{1} \tilde \theta_t\\
    a_{t+1} &= \gamma_{0} a_t - \gamma_{1} \frac{1}{n} \sum_{i=1}^n f(h_{t,i}, y_i, a_t). 
\end{split}\]

We prove the following. 

\begin{lemma}\label{lem:lipAs} Assume that almost surely the following limits exist, $\lim_{d \to \infty} \tilde m_{j,t} = m_{j,t} \in \RR^L$, $\lim_{d \to \infty} \tilde l_{j,t} = l_{j,t} \in \RR^L$ and $\lim_{d \to \infty} a_{t} = \bar a_t$. If $\phi_1, \phi_2: \RR^{n \times (L \cdot T)} \times \RR^n \times \RR^{L' \cdot T} \to \RR^n$, $\varphi_1, \varphi_2: \RR^{d \times (L \cdot T)} \times \RR^{L' \cdot T} \to \RR^n$ are Lipschitz in their first argument (and third argument for $\phi_1, \phi_2$, both uniformly over the second argument $y \in \RR^n$), then we have that almost surely,
    \[\lim_{n \to \infty} \frac{1}{n} \sum_{i=1}^n \phi(h_1, \dots, h_T, y, a_1, \dots, a_T)_i - \frac{1}{n} \sum_{i=1}^n \phi(h^{\rm frozen}_1, \dots, h^{\rm frozen}_T, y, \bar a_1, \dots, \bar a_T)_i = 0\]
    and 
    \[\lim_{n \to \infty} \frac{1}{d} \sum_{i=1}^d \varphi(\tilde \theta_1, \dots, \tilde \theta_T, a_1, \dots, a_T)_i - \frac{1}{n} \sum_{i=1}^n \phi(\tilde \theta^{\rm frozen}_1, \dots, \tilde \theta^{\rm frozen}_T, \bar a_1, \dots, \bar a_T)_i = 0,\]
    where $\phi(\cdot)_i = \phi_1(\cdot)_i\phi_2(\cdot)_i$ and $\varphi(\cdot)_i = \varphi_1(\cdot)_i\varphi_2(\cdot)_i$.
\end{lemma}

\begin{proof}
    We proceed by induction on $t \in [T]$ to prove the iterates of the original and frozen algorithm have a vanishing difference in two-norm almost surely as $n,d \to \infty$ with $n/d \to \alpha$. Specifically, we prove the following limits hold almost surely,
    \begin{align}
        \lim_{n \to \infty} \frac{1}{n} \|h_t - h_t^{\rm frozen}\|_2^2 = 0, \quad &\lim_{n \to \infty} \frac{1}{n} \|\hat h_t - \hat h_t^{\rm frozen}\|_2^2 = 0, \label{eq:ind_h}\\
        \lim_{d \to \infty} \frac{1}{d} \|\tilde \theta_t - \tilde \theta_t^{\rm frozen}\|_2^2 = 0, \quad &\lim_{d \to \infty} \frac{1}{d} \|\theta_{t+1} - \theta_{t+1}^{\rm frozen}\|_2^2 = 0 \label{eq:ind_th}.
    \end{align}

    \textbf{Base Case ($t=1$):} 
    By initialization, $\theta_1 = \theta_1^{\rm frozen}$. Evaluating the pre-activations gives,
    \begin{align}
        h_1 - h_1^{\rm frozen} &= \left(\sum_{j=1}^J e_{B_j} \tilde m_{j,1}^\top + Z\theta_1 \right) - \left(\sum_{j=1}^J e_{B_j} m_{j,1}^\top + Z\theta_1^{\rm frozen} \right) \\
        &= \sum_{j=1}^J e_{B_j} (\tilde m_{j,1} - m_{j,1})^\top.
    \end{align}
    Since $\lim_{d \to \infty} \tilde m_{j,1} = m_{j,1}$ almost surely, and that $\lim_{n \to \infty}\|e_{B_j}\|_2^2 / n = p_j \in [0,1]$ from Assumption \ref{as:main0} \eqref{as:data} (with $J$ being $n$-independent), we have $\lim_{n \to \infty} \frac{1}{n}\|h_1 - h_1^{\rm frozen}\|_2^2 = 0$ almost surely. 
    Thus, as the function $g$ is Lipschitz with respect to input $h_1, a_1$ (for some bounded constant $C>0$), we have that
    \begin{align}
        \frac{1}{n}\|\hat h_1 - \hat h_1^{\rm frozen}\|^2_2 &= \frac{1}{n}\|g(h_1, y, a_1) - g(h_1^{\rm frozen}, y, \bar a_1)\|^2_2\\
        &\leq C\left(\frac{1}{n}\|h_1 - h_1^{\rm frozen}\|^2_2 + \|a_1 - \bar a_1\|_2^2\right).
    \end{align}
    Therefore, as an immediate consequence of $\lim_{n \to \infty} \frac{1}{n}\|h_1 - h_1^{\rm frozen}\|_2^2 = 0$ and $\lim_{d \to \infty}a_1 = \bar a_1$, we have $\lim_{n \to \infty} \frac{1}{n}\|\hat h_1 - \hat h_1^{\rm frozen}\|_2^2 = 0$ almost surely.

    \textbf{Inductive Step:} 
    Assume equations \eqref{eq:ind_h} and \eqref{eq:ind_th} hold for all $s \le t$. We first consider the update for $\tilde \theta_t$:
    \begin{align}
        \tilde \theta_t - \tilde \theta_t^{\rm frozen} &= \alpha \sum_{j=1}^J \mu_j (\tilde l_{j,t} - l_{j,t})^\top + Z^\top (\hat h_t - \hat h_t^{\rm frozen}) \\
        &\quad - \alpha \left( \left\langle \nabla_h g(h_t, y, a_t) \right\rangle \theta_t - \left\langle \nabla_h g(h_t^{\rm frozen}, y, \bar a_t) \right\rangle \theta_t^{\rm frozen} \right),
    \end{align}
    where $\langle \cdot \rangle$ is a shorthand for the empirical average over $n$ coordinates. We bound the two-norm of each of these three terms separately:
    \begin{enumerate}
        \item \textbf{Signal Term:} Since $\lim_{d \to \infty} \tilde l_{j,t} = l_{j,t}$ and $\frac{1}{d}\|\mu_j\|_2^2 = O(1)$ by Assumption \ref{as:main0} \eqref{as:data}, we immediately have $ \lim_{d \to \infty} \frac{1}{d} \| \alpha \sum_{j=1}^J \mu_j (\tilde l_{j,t} - l_{j,t})^\top\|_2^2 = 0$ almost surely.
        \item \textbf{Noise Term:} Since $Z \in \RR^{n \times d}$ is a matrix with independent $\cN(0, 1/d)$ entries, the operator norm $\|Z\|_{\rm op}$ is almost surely bounded \cite{bai1993} as $n, d \to \infty$. Thus, almost surely,
        \[\lim_{d \to \infty} \frac{1}{d}\|Z^\top (\hat h_t - \hat h_t^{\rm frozen})\|_2^2 \leq \lim_{d \to \infty} \frac{\|Z\|_{\rm op}^2}{d} \|\hat h_t - \hat h_t^{\rm frozen}\|_2^2 = \lim_{d \to \infty} \alpha \|Z\|_{\rm op}^2 \frac{1}{n} \|\hat h_t - \hat h_t^{\rm frozen}\|_2^2 = 0, \] 
        where the final equality applied the inductive hypothesis on $\hat h_t$.
        \item \textbf{Onsager Term:} Adding and subtracting $\langle \nabla_h g(h_t^{\rm frozen}, y, \bar a_t) \rangle \theta_t$, we have by the triangle inequality and Cauchy-Schwarz,
\begin{align}
    &\frac{\alpha^2}{d} \left\| \langle \nabla_h g(h_t, y, a_t) \rangle \theta_t - \langle \nabla_h g(h_t^{\rm frozen}, y, \bar a_t) \rangle \theta_t^{\rm frozen} \right\|_2^2 \\
    &\le 2\alpha^2 \left\| \langle \nabla_h g(h_t, y, a_t) \rangle - \langle \nabla_h g(h_t^{\rm frozen}, y, \bar a_t) \rangle \right\|_{\rm op}^2 \frac{\|\theta_t\|_2^2}{d} \\
    &\quad + 2\alpha^2 \left\| \langle \nabla_h g(h_t^{\rm frozen}, y, \bar a_t) \rangle \right\|_{\rm op}^2 \frac{\|\theta_t - \theta_t^{\rm frozen}\|_2^2}{d}.
\end{align}
By Assumption \ref{as:main0} \eqref{as:algo}, $\nabla_h g$ is Lipschitz in both $h$ and $a$. Thus, the operator norm difference $\left\| \langle \nabla_h g(h_t, y, a_t) \rangle - \langle \nabla_h g(h_t^{\rm frozen}, y, \bar a_t) \rangle \right\|_{\rm op}^2$ vanishes almost surely by the inductive hypothesis on $h_t$ and the almost sure convergence of $\lim_{n \to \infty} a_t = \bar a_t$. Since $\frac{1}{d}\|\theta_t\|_2^2 = O(1)$ almost surely (by the inductive hypothesis and Corollary \ref{cor:appSe}), the first term vanishes. The second term vanishes because $\frac{1}{d}\|\theta_t - \theta_t^{\rm frozen}\|_2^2 \to 0$ by the inductive hypothesis, and the empirical average of $\nabla_h g$ is bounded.
    \end{enumerate}
     Therefore, $\frac{1}{d}\|\tilde \theta_t - \tilde \theta_t^{\rm frozen}\|_2^2 \to 0$.

    By the linear update $\theta_{t+1} = \eta_{0}\theta_t - \eta_{1}\tilde \theta_t$, we trivially obtain $\frac{1}{d}\|\theta_{t+1} - \theta_{t+1}^{\rm frozen}\|_2^2 \to 0$. 

    Similarly, evaluating the update for $h_{t+1}$,
    \begin{align}
        h_{t+1} - h_{t+1}^{\rm frozen} &= \sum_{j=1}^J e_{B_j}(\tilde m_{j,t+1} - m_{j,t+1})^\top + Z(\theta_{t+1} - \theta_{t+1}^{\rm frozen}) \\
        &\quad + \eta_{ 1} \sum_{s=1}^{t} \eta_{ 0}^{t-s} (\hat h_{s} - \hat h_{s}^{\rm frozen}).
    \end{align}
    Applying the assumed limits for $\tilde m_{j, t+1}$, the bounded operator norm of $Z$, and the inductive hypothesis on the past $\hat h_s$ iterates, we find $\lim_{n \to \infty}\frac{1}{n}\|h_{t+1} - h_{t+1}^{\rm frozen}\|_2^2 = 0$ almost surely. 

    \textbf{Convergence of Test Functions:}
    We now prove that almost sure convergence of the iterates implies the almost sure convergence of the test functions. Consider $\phi(\cdot)_i = \phi_1(\cdot)_i \phi_2(\cdot)_i$ where $\phi_1, \phi_2: \RR^{n \times (L \cdot T)} \times \RR^n \times \RR^{L' \cdot T} \to \RR^n$ are Lipschitz functions satisfying the statement of the Lemma. For notational simplicity, we abbreviate $\phi(h_1, \dots, h_T, y, a_1, \dots, a_T)$ as $\phi(h)$ and .
    $\phi(h^{\rm frozen}_1, \dots, h^{\rm frozen}_T, y,\bar a_1, \dots, \bar a_T)$ as $\phi(h^{\rm frozen})$.
    
    We then have that,
    \begin{align}
        &\frac{1}{n} \left|\sum_{i=1}^n (\phi(h)_i - \phi(h^{\rm frozen})_i)\right| = \frac{1}{n}\left|\phi_1(h)^\top \phi_2(h) - \phi_1(h^{\rm frozen})^\top\phi_2(h^{\rm frozen})\right| \\
        &\quad\le  \frac{1}{n}\left(\left|\phi_1(h)^\top \phi_2(h) - \phi_1(h)^\top\phi_2(h^{\rm frozen})\right| + \left|\phi_1(h)^\top \phi_2(h^{\rm frozen}) - \phi_1(h^{\rm frozen})^\top\phi_2(h^{\rm frozen})\right| \right)\label{eq:testAsZero}
    \end{align}
    By the Cauchy-Schwarz inequality, the first term is bounded by,
    \[
    \left( \frac{1}{n}\|\phi_1(h)\|^2_2\right)^{1/2} \left( \frac{1}{n} \|\phi_2(h) - \phi_2(h^{\rm frozen})\|_2^2 \right)^{1/2}.\]
    Because $\phi_1$ is Lipschitz (for some constant $C > 0$) in $h,a$ we have that
    \begin{align}
    \frac{1}{n} \|\phi_1(h)\|_2^2 &= \frac{1}{n} \|\phi_1(h) - \phi_1(h^{\rm frozen}) + \phi_1(h^{\rm frozen})\|_2^2\\ 
    &\leq \frac{2}{n} \|\phi_1(h) - \phi_1(h^{\rm frozen})\|_2^2 + \frac{2}{n} \|\phi_1(h^{\rm frozen})\|_2^2\\
    &\leq 2 C\left( \frac{1}{n} \|h - h^{\rm frozen}\|^2_F + \sum_{t=1}^T \|a_t - \bar a_t\|_2^2 \right) + \frac{2}{n} \|\phi_1(h^{\rm frozen})\|_2^2
    \end{align}
    and using the inductive hypothesis, the assumed almost sure convergence of $a_t$ and Corollary \ref{cor:appSe} on $\frac{2}{n} \|\phi_1(h^{\rm frozen})\|_2^2$ implies that 
    it is almost surely bounded by a constant. Moreover, using that $\phi_2$ is Lipschitz (for some bounded Lipschitz constant $C > 0$) in both argument $h_1, \dots, h_T$ and $a_1, \dots, a_T$, we have that
    \begin{align}
        \frac{1}{n} \|\phi_2(h) - \phi_2(h^{\rm frozen})\|_2^2 &= \frac{1}{n} \|\phi_2(h,y,a_1, \dots, a_T) - \phi_2(h^{\rm frozen},y, \bar a_1, \dots, \bar a_T)\|_2^2\\
        &\leq C\left( \frac{1}{n} \|h - h^{\rm frozen}\|^2_F + \sum_{t=1}^T \|a_t - \bar a_t\|_2^2 \right).
    \end{align}
    As we have previously shown that $(\frac{1}{n}\|h - h^{\rm frozen}\|_F^2)$ converges to $0$ almost surely, then the first term of \eqref{eq:testAsZero} converges to zero almost surely. Applying the symmetric argument to the second term in \eqref{eq:testAsZero} gives the desired almost sure limit for $\phi$.

    An identical argument gives the exact same result for $\varphi(\tilde \theta_1, \dots, \tilde \theta_T, a_1, \dots, a_T)$. This completes the proof.
\end{proof}

\subsubsection{The Almost Sure Limits $m_{j,t}$, $l_{j,t}$ And $\bar
a_t$}\label{sec:almostSureLims}

To avoid confusion with the function $\phi$ used in the appendix, we denote the test function in Lemma~\ref{lem:ddState} as $\bar \phi$.
    
Combining Corollary \ref{cor:appSe} and Lemma \ref{lem:lipAs}, the proof of
Lemma \ref{lem:ddState} follows after observing that Assumption \ref{as:phi} can
allow $\bar \phi$ to be utilized in both Corollary \ref{cor:appSe} and Lemma \ref{lem:lipAs} with
the identification of $\phi_1 = \bar \phi$ (applied row-wise) and $\phi_2$ being the all ones
vector. As Definition \ref{def:frozenAmpState} and Definition \ref{def:ddState} are identical up to the frozen state evolution parameters, and have proven we can replace $\tilde m_{j,t}, \tilde l_{j,t}$ and $a_t$ by the frozen sequence of their respective almost sure limits, we are done once we identify the recursions for the almost sure
limits of $\tilde m_{j,t}, \tilde l_{j,t}$ and $a_t$. 

For the case of $\tilde m_{j,t}$, we have that, 
\[\tilde m_{j,t+1} = \frac{\theta_{t+1}^\top \mu_j}{d} = \eta_{ 0} \tilde m_{j,t} -
\eta_{1} \frac{\tilde \theta_t^\top \mu_j}{d},\]
recursively plugging in the almost sure limit of $\frac{1}{d} \mu_j^\top \tilde
\theta_t$ from Corollary \ref{cor:appSe} and that $\EE[\check W_t] = 0$, we have that for each $s \in [T]$, 
\[\lim_{d \to \infty} \frac{\tilde \theta_s^\top \mu_j}{d} = \lim_{d \to
\infty} \alpha \sum_{k=1}^J \frac{\mu_k^\top \mu_j}{d} l_{k,s} = 
\alpha \sum_{k=1}^J \chi_{j,k} l_{k,s},\] by the assumed limit of
$\frac{\mu_j^\top \mu_k}{d} = \frac{\mu_k^\top \mu_j}{d}$ in Assumption
\ref{as:main0} \eqref{as:data}. Thus, we have that almost surely that $\lim_{d
\to \infty} \tilde m_{j,t} = m_{j,t}$ defined by the recursion 
\[m_{j,t+1} = \eta_{0} m_{j,t} - \eta_{1} \sum_{k=1}^J  \alpha \chi_{j,k} l_{k,t}.\]

Similarly, using that $\tilde l_{j,t} = \frac{e_{B_j}^\top \hat
h_t}{d} = \frac{|B_j|}{n} \frac{1}{|B_j|} \sum_{i \in B_j} \hat
h_{t,i} = \frac{|B_j|}{n} \frac{1}{|B_j|} \sum_{i \in B_j}
g(h_{t,i}, y_i, a_t),$ we have the almost sure limit,
\[l_{j,t} = p_j\EE[g(Z_t + m_{j,t}, Y_j, \bar a_t)],\]
where $Z_t \sim \cN(0,\check \Omega_t\ss{t,t})$.

As a final note, using that $a_{t+1} = \gamma_{0} a_t - \gamma_{1}
\sum_{j=1}^J \frac{|B_j|}{n} \frac{1}{|B_j|} \sum_{i\in B_j} f(h_{t,i}, y_i,
a_t)$, and thus by Corollary \ref{cor:appSe} with $\phi_1 = f$ and
$\phi_2$ being the all ones vector (recalling from Assumption \ref{as:main0}
\eqref{as:algo} that $f$ is Lipschitz in its first argument), we have the
following almost sure limit,
\[\bar a_{t+1} = \gamma_{0} \bar a_t - \gamma_{1} \sum_{j=1}^J p_j
\EE[f(G^t+ m_{j,t}, Y_j, \bar a_t)],\]
noting that $\bar a_{1} = a_1$ and we assumed in Assumption \ref{as:main0} \eqref{as:limits} that $\|a_1\|_2$ is bounded and that $f$ is bounded from Assumption \ref{as:main0} \eqref{as:algo}, meaning that each limit has $\|\bar a_t\|_2$ bounded for $t \in [T]$.

Combining these almost sure limits with the state evolution recursion from
Definition \ref{def:frozenDdState}, we obtain the state evolution in Definition
\ref{def:ddState} and thus conclude Lemma \ref{lem:ddState}.

\subsection{The Proof Of Theorem \ref{thm:ddLoss}}\label{sec:proofThm}
\begin{proof}[Proof Of Theorem \ref{thm:ddLoss}]
Under Assumption \ref{as:main0}, we have that Assumption \ref{as:phi} holds for
the choice of $\phi(h, y, a) = \cL(\cM_{a}(h), y)$. Invoking Lemma
\ref{lem:ddState} gives the following almost surely,
\[\lim_{n \to \infty}\frac{1}{n} \sum_{i=1}^n \cL(\cM_{a_t}(h_{t,i}), y_i) =
\sum_{j=1}^J p_{j} \EE[\cL(\cM_{\bar a_t}(m_{j,t} + G^t), Y_j)], \quad G^t \sim \cN(0, \Omega_t\ss{t,t}), \quad Y_j \sim \PP_j,\]
it is then immediate that $\lim_{n \to \infty}\frac{1}{n} \sum_{i=1}^n \cL(\cM_{a_t}(h_{t,i}), y_i) = \testfunc(m_{1,t}, \dots, m_{J,t}, \Omega_t\ss{t,t}, \bar a_t)$ from Definition \ref{def:testErr}, invoking Proposition \ref{prop:testErr} concludes the proof.
\end{proof}

\subsection{The State Evolution Of Pure DD}\label{sec:pureDdSe}

Using Definition \ref{def:ddState}, we can specify the state evolution for this algorithm. We abbreviate $\nabla_h \Psi_{r,j} = \nabla_h \Psi(m_{j,r} + G^r, Y_j, \bar a_r)$ where $G^1, \dots, G^t \sim \cN(0, \Omega_t)$, $Y_j \sim \PP_j$ and define $V_{r,s} = \alpha \Sigma_{t}\ss{r,s} + \alpha^2 \sum_{j,k=1}^J \chi_{j,k} p_j p_k \EE[\nabla_h \Psi_{r,j}]\EE[\nabla_h \Psi_{s,k}]^\top$. With identical initializations from Definition \ref{def:ddState} and indices $r,s \in [t]$, we recursively define,
\[\label{eq:pureDdState}\begin{split}
\Sigma_t\ss{r,s} &= \sum_{j=1}^J p_{j} \EE[\nabla_h \Psi_{r,j} \nabla_h \Psi_{s,j}^\top]\\
 \Xi_{t}\ss{r+1, s} &= \Xi_{t-1}\ss{r,s} - \eta V_{r,s}\\
 \Omega_{t+1}\ss{r+1,s+1} &= \Omega_t\ss{r,s} - \eta (\Xi_t\ss{r,s} + \Xi_t\ss{s,r}^\top) + \eta^2 V_{r,s} \\
m_{j,t+1} &= m_{j,t} - \eta \alpha \sum_{k=1}^J \chi_{j,k} p_{k} \EE[\nabla_{h}\Psi_{t,k}]\\
 \bar a_{t+1} &= \bar a_t - \eta \sum_{j=1}^J p_{j} \EE[\nabla_{a}\Psi(m_{j,t} + G^t, Y_j, \bar a_t)].
\end{split}\]

\subsection{Taylor Expansion Of (Pure) DD}\label{sec:design}

\begin{lemma}\label{lem:orderStein} Let $\mu \in \RR^L$, $\Sigma \in \RR^{L
    \times L}$ and $c \in \RR$. Then for $X \sim \cN(c\mu, \Sigma)$, the
    following identities hold,
    \begin{align}
        \nabla_{\mu} \EE[f(X)] &= c\EE[\nabla_x f(x)|_{x = X}]\\
        \nabla_{\Sigma} \EE[f(X)] &= \frac{1}{2}\EE[\nabla^2_{x} f(x)|_{x = X}]\\
        \nabla^2_{\mu} \EE[f(X)] &= c^2\EE[\nabla^2_{x} f(x)|_{x = X}]\\
        \nabla^2_{\Sigma} \EE[f(X)] &= \frac{1}{4}\EE[\nabla^4_{x} f(x)|_{x = X}]\\
        \nabla_\mu \nabla_{\Sigma} \EE[f(X)] &= \frac{c}{2}\EE[\nabla^3_{x} f(x)|_{x = X}].
    \end{align}
\end{lemma}

\begin{proof}
    See \cite[Equation (2)]{sklaviadis2026}.
\end{proof}

\begin{lemma}\label{lem:dominated}
    Suppose Assumption \ref{as:main0} holds. Let $\cT$ be the space of all
    $\theta = (\bar{a}, m_1, \dots, m_J, \omega)$ where $\|\bar{a}\|_2 \leq
    C$, $\|m_j\|_2 \leq C$, and $\|\omega^{-1}\|_{\op} \leq C$, $\|\omega\|_{\op} \leq C$ for some constant $C >
    0$. If $G \sim \cN(0, \Id_L)$ and $Y_j \sim \PP_j$, then the expectations of
    the derivatives of $\Psi(m_j + \omega^{1/2}G, Y_j, \bar{a})$ with
    respect to $\bar{a}$, $m_j$, and $\omega$ exist, and the expectation
    operator can be interchanged with the derivative operator.
\end{lemma}

\begin{proof}
    To interchange the expectation and derivative, it suffices to show that the partial derivatives of $\Psi(m_j + \omega^{1/2}G, Y_j, \bar{a})$ with respect to $\bar{a}$, $m_j$, and $\omega$ are uniformly bounded by an integrable function for all $\theta \in \cT$.
    
    Let $x = m_j + \omega^{1/2}G$. By Assumption \ref{as:main0} \eqref{as:model} (a), $\Psi$ is Lipschitz with respect to $h$ and $a$. Therefore, the gradients $\nabla_h \Psi$ and $\nabla_a \Psi$ exist almost everywhere and are uniformly bounded by constant $L > 0$.
    Thus, each of the following derivatives satisfy
    \begin{align}
        \|\nabla_{\bar{a}} \Psi\|_2 &= \|\nabla_a \Psi\|_2 \le L \\
        \|\nabla_{m_j} \Psi\|_2 &= \|\nabla_h \Psi\|_2 \le L \\
        \|\nabla_{\omega} \Psi\|_{\text{F}} &\le \|\nabla_h \Psi\|_2 \|\nabla_\omega (\omega^{1/2}G)\|_{\text{op}} \le L \cdot C' \|G\|_2,
    \end{align}
    where $C' > 0$ is a constant bounding the derivative of the matrix square root uniformly over the domain $\cT$; such a constant exists by the assumption that $\max(\|\omega^{-1}\|_{\op}, \|\omega\|_{\op})\leq C$.
    
    We construct the dominating function $H(G) = \max(L, L \cdot C' \|G\|_2)$. Because $G \sim \cN(0, \Id_L)$, the second moment $\|G\|_2$ is bounded. Thus, $H(G)$ is integrable and bounds the parameter derivatives independently of $\theta \in \cT$, allowing the interchange of expectation and differentiation.
\end{proof}

\begin{lemma}\label{lem:bigHessBound}
    Let Assumption \ref{as:main0} \eqref{as:model} hold for a bounded constant $C > 0$, 
    the operator norm of the block Hessian matrix $\nabla^2_{a, m, {\rm vec}(\omega)}\EE[\Psi(m + \omega^{1/2} G, Y_j, a)]$ is bounded by a constant $C' > 0$ depending only on $C$.
\end{lemma}

\begin{proof}
    We can expand the matrix $\nabla^2_{a, m, {\rm vec}(\omega)}\EE[\Psi]$ into the following $3 \times 3$ block matrix,{\small
       \[\hspace{-1cm}\begin{bmatrix}
    \nabla_{a}^2 \EE[\Psi(m \+ \omega^{1/2} G, y, a)] 
    & \nabla_{a} \nabla_m \EE[\Psi (m \+ \omega^{1/2} G, y, a)] 
    & \nabla_{a} \nabla_{{\rm vec}(\omega)} \EE[\Psi (m \+ \omega^{1/2} G, y,
    a)]\\ 
    \nabla_{m} \nabla_{a} \EE[\Psi (m \+ \omega^{1/2} G, y, a)] 
    &\nabla_{m}^2 \EE[\Psi (m \+ \omega^{1/2} G, y, a)] 
    & \nabla_m \nabla_{{\rm vec}(\omega)} \EE[\Psi (m \+ \omega^{1/2} G, y, a)]\\
    \nabla_{{\rm vec}(\omega)} \nabla_{a}\EE[\Psi (m \+
    \omega^{1/2} G, y, a)] & \nabla_{{\rm vec}(\omega)} \nabla_{m} \EE[\Psi (m
    \+ \omega^{1/2} G, y, a)] & \nabla_{{\rm vec}(\omega)}^2 \EE[\Psi (m \+
    \omega^{1/2} G, y, a)] \end{bmatrix}.\]
    }
    Applying the derivative identities from Lemma \ref{lem:orderStein}, each block in the above matrix becomes one of
    $\EE[\nabla_a^2 \Psi]$, 
    $\EE[\nabla_h \nabla_a \Psi]$, 
    $\EE[\nabla_h^2 \Psi]$, 
    $\frac{1}{2} \EE[\nabla_h^2 \nabla_a \Psi]$, 
    $\frac{1}{2} \EE[\nabla_h^3 \Psi]$, and 
    $\frac{1}{4} \EE[\nabla_h^4 \Psi]$, where we have suppressed the input for notational simplicity.
    
    Assumption \ref{as:main0} \eqref{as:model} guarantees the operator or Frobenius norm of each of these derivatives is bounded by $C$. Because the operator norm of any block matrix is bounded by the summed operator norms of its block matrices and the Frobenius norm strictly upper-bounds the operator norm, the operator norm of the entire block matrix is bounded by a finite linear combination of constant $C$, concluding the proof.
\end{proof}

\begin{theorem}\label{thm:ddLearn} 
Let $T\in \NN$ finite. Consider algorithm \eqref{eq:dd}; recall the state
    evolution variables (specifically $\Omega_t$, $m_{j,t}$ and $\bar a_t$) from
    Definition \ref{def:ddState} and $\testfunc_t = \testfunc(m_{1,t}, \dots, m_{J,t}, \Omega_t\ss{t,t}, \bar a_t)$ from Definition \ref{def:testErr}.

    Let $\epsilon = \max(|1 - \eta_{0}|, |1 - \gamma_{0}|, \eta_{1}, \gamma_{1})$. If Assumption \ref{as:main0} and Assumption \ref{as:phi} (with respect to $g,f$) hold for sufficiently small $\epsilon > 0$, then in the limit $\epsilon \to 0$,
    $\testfunc_t$ satisfies the following expansion with $G^t \sim \cN(0,
    \Omega_t\ss{t,t})$,
    \begin{align}
        \testfunc_{t+1} &= \testfunc_t\\
        %% a-damping
        &\quad - (1 - \gamma_{0}) \sum_{j=1}^J p_{j} \Big \langle
        \EE[\nabla_{a} \Psi(m_{j,t} + G^t, Y_j, \bar a_t)], \bar
        a_t \Big \rangle \label{eq:aDamp}\\
        %% a-descent
        &\quad - \gamma_{1} \sum_{j,j'=1}^J p_{j}p_{j'} \Big \langle
        \EE[\nabla_{a} \Psi(m_{j,t} + G^t, Y_j, \bar a_t)],
        \EE[f(m_{j',t} + G^t, Y_{j'}, \bar a_t)] \Big \rangle
        \label{eq:aDescent}\\
        %% Theta-damping
        &\quad - (1 - \eta_{0}) \sum_{j=1}^J p_{j} \Big
        \langle \EE[\nabla_{h}\Psi(m_{j,t} + G^t, Y_j, \bar a_t)],
        m_{j,t} \Big \rangle \label{eq:ThDamp}\\
        %% Theta-signal
        &\quad -\!\eta_{1} \alpha\!\!\!\!\sum_{j,j' = 1}^J\!\!\!
        p_{j}\chi_{j,j'} p_{j'} \Big \langle\!\EE[\nabla_{h}\!\Psi(
        m_{j,t} \+ G^t,Y_j, \bar a_t)], \EE[g(m_{j',t} \+ G^t, Y_{j'},
        \bar a_t)]\! \Big\rangle \label{eq:ThSignal}\\
        &\quad -\eta_0(1-\eta_0)\sum_{j=1}^J p_j \Big\langle \EE[\nabla_h^2
        \Psi(m_{j,t} + G^t, Y_j, \bar a_t)], \Omega_t\ss{t,t} \Big\rangle \label{eq:omDamp}\\
        &\quad - \frac{1}{2}\eta_{0}\eta_{1} \sum_{j=1}^J p_j
        \Big\langle \EE[\nabla_h^2 \Psi(m_{j,t} + G^t, Y_j, \bar a_t)],
        \Xi_t\ss{t,t} + \Xi_t\ss{t,t}^\top \Big\rangle \label{eq:omCross}\\
        &\quad + \delta_t, \label{eq:tayRemainder}
    \end{align}
    where $\sup_{t \in [T]} |\delta_t| \leq C \epsilon^2$ as $\epsilon \to 0$.
\end{theorem}

\begin{remark}[Physical Interpretation Of The First Order Expansion]
    The generalization error expansion from Theorem \ref{thm:ddLearn} decomposes
    the local behavior of DD optimization into distinct physical
    terms. Connecting to works in spin glasses, one can interpret the
    order parameters $m_{j,t}$ and $\Omega_t\ss{t,t}$ as the magnetization (with
    respect to external field vectors $\mu_1, \dots, \mu_J$) and self-overlap
    of the parameter $\theta_t$ respectively. In addition, each of terms
    \eqref{eq:aDamp}-\eqref{eq:omCross} has the following physical
    interpretation:
    \begin{itemize}
        \item Equations \eqref{eq:aDamp} and \eqref{eq:ThDamp} correspond to
        weight decay due to the introduction of a non-zero damping factor of
        $1-\gamma_{0}$ and $1-\eta_{0}$ respectively.
        \item Equation \eqref{eq:aDescent} enforces the representation of the
        model $\cM_{a}$ to align the low-dimensional signal $a$ with the loss
        landscape through the use of function $f$.
        \item Equation \eqref{eq:ThSignal} drives the signal acquisition of each
        of the vectors $\mu_1, \dots, \mu_J$ by $\theta$ using the function $g$.
        Notice, the presence of $\chi_{j,j'}$ and $\alpha$ incorporates the
        overlap structure and data aspect ratio natively into the path of the
        optimization trajectory.
        \item Equations \eqref{eq:omDamp} and \eqref{eq:omCross} govern the variance inflation of the relative noise
        in the directions perpendicular to $\mu_1, \dots, \mu_J$. The first term
        represents the effect of damping. The second term, through unrolling the
        definition of $\Xi_t\ss{t,t}$, can be seen as cumulative correlation of
        the current gradient direction with the direction of past gradient
        evaluations. Each of these effects is
        compared to the relative scale of the expected Hessian of the
        post-activations corresponding to the iterate $\theta_t$, eliciting the intuition
        that flat regions of the loss landscape suffer from little inflation
        while steep regions are subject to strong inflation. 
    \end{itemize}
\end{remark}

\begin{proof}[Proof Of Theorem \ref{thm:ddLearn}]
    By Definition \ref{def:testErr}, the value of $\testfunc_{t+1}$ can be written as,
    \[\testfunc_{t+1} = \sum_{j=1}^J p_{j} \EE[\cL(\cM_{\bar
    a_{t+1}}(m_{j,t+1} + (\Omega_{t+1}\ss{t\+1, t\+1})^{1/2}G), Y_j)],\]
    where $G \sim \cN(0, \Id_L)$ and $Y_j \sim \PP_j$. For convenience, we abbreviate
    $\Omega_{s}\ss{s,s} = \omega_s$ for all $s \in [t+1]$. Using the recursive
    definitions of $m_{j,t+1}$, $\bar a_{t+1}$ from Definition
    \ref{def:ddState} and $\omega_{t+1}$ from Remark \ref{rem:ddstate}, we have that
    \begin{align}\testfunc_{t+1} &= \sum_{j=1}^J p_{j} \EE[\cL(\cM_{\bar a_t + \Delta
    \bar a}((m_{j,t} + \Delta m_j) + (\omega_t + \Delta
    \omega)^{1/2}G), Y_j)]\\
    &= \sum_{j=1}^J p_j \EE[\Psi((m_{j,t} + \Delta m_j) + (\omega_t + \Delta
    \omega)^{1/2}G, Y_j, \bar a_t + \Delta
    \bar a)],\end{align}
    where we define (recalling $l_{k,t}$ from Definition \ref{def:ddState})
    \begin{align}
    \Delta \bar a &= -(1 - \gamma_{0}) \bar a_t - \gamma_{1} A,\\
    \Delta m_j &= -(1 - \eta_{0}) m_{j,t} - \eta_{1} B_j,\\
    \Delta \omega &= -(1 - \eta_{0}^2) \omega_t - \eta_{0}
    \eta_{1}(\Xi_t\ss{t,t} + \Xi_t\ss{t,t}^\top) + \eta_{1}^2 \left(\alpha \Sigma_t\ss{t,t} + \alpha^2 \sum_{k,k'=1}^J \chi_{k,k'} l_{k,t} l_{k',t}^\top \right)\\
    A &= \sum_{j = 1}^J p_{j} \EE[f(m_{j,t} + \omega_t^{1/2} G, Y_j, \bar a_t)],\\
    B_{j} &= \alpha \sum_{k = 1}^J \chi_{j,k} p_{k} \EE[g(m_{k,t} + \omega_t^{1/2} G, Y_{k}, \bar a_t)]\\
    l_{k,t} &= p_j \EE[g(m_{k,t} + m_{k,t}, Y_k, \bar a_t)],
    \end{align}
    with $(G^1, \dots, G^t) \sim \cN(0, \Omega_t)$ implicit in the definition
    of $\Sigma_t$ and $\Xi_t$.

    By Assumption \ref{as:main0} \eqref{as:algo}, the functions $f,g$ are bounded, therefore each of $A, (B_j)_{j \in [J]}, (l_{j,t})_{j \in [J]}$ are element-wise bounded which further implies that $\|A\|_2, \|B_j\|_2, \|l_{j,t}\|_2$ are bounded (dependent on dimension $L$) for all $j \in [J]$. 
    
    Moreover, by Assumption \ref{as:phi} the operator norm of the matrices $\omega_t = \Omega_t\ss{t,t} = \Omega_T\ss{t,t}$ and $ \Sigma_t\ss{t,t} = \Sigma_T\ss{t,t}$ have bounded operator norm independent of $t$, and therefore the values of $\|\omega_t\|_F$, $\|\Sigma_t\ss{t,t}\|_F$ are bounded independent of $t \in [T]$. Invoking Lemma \ref{lem:ddState}, we have that $\lim_{d \to \infty} \frac{\theta_t^\top \tilde \theta_t}{d} = \Xi\ss{t,t}$, then using that 
    \[\left\|\frac{\theta_t^\top \tilde \theta_t}{d}\right\|_F \leq \left\|\frac{\theta_t}{\sqrt{d}}\right\|_F \left\|\frac{\tilde \theta_t}{\sqrt{d}}\right\|_F = \sqrt{\text{Tr}\left(\frac{\theta_t^\top \theta_t}{d}\right)} \sqrt{\text{Tr}\left(\frac{\tilde \theta_t^\top \tilde \theta_t}{d}\right)}\]
    and the limiting values of $\lim_{d \to \infty} \frac{\theta_t^\top \theta_t}{d} = \omega_t$ and $\lim_{d \to \infty} \frac{\tilde \theta_t^\top \tilde \theta_t}{d} = \alpha \Sigma_t\ss{t,t}$, we have that $\|\Xi_t\ss{t,t}\|_F$ is also bounded independent of $t$. Therefore, each term of $\Delta \bar a$, $\Delta m_j$ (for each $j \in [J]$) and $\Delta \omega$ is bounded by $C \epsilon$ (for some constant $C > 0$ when $\epsilon$ is sufficiently small) uniformly over $t \in [T]$.

    Now, define
    \[\cG(\bar a, m_1, \dots, m_J, \omega) = \sum_{j=1}^J p_{j}
    \EE[\Psi(m_{j} + \omega^{1/2} G, Y_j, \bar a)].\] By Taylor's
    theorem, there exists a point $\xi = (\xi_{\bar a}, \xi_{(m_{j})_{j \in
    [J]}}, \xi_{\omega})$ on the line connecting $(\bar a_t, (m_{j,t})_{j \in
    [J]}, \omega_t)$ to $(\bar a_{t+1}, (m_{j, t+1})_{j \in [J]},
    \omega_{t+1})$ such that,
    \[\label{eq:taylorG}\begin{split}
    \cG(\bar a_{t+1}, m_{1,t+1}, \dots, m_{J,t+1}, \omega_{t+1})
    &= \cG(\bar a_t, m_{1,t}, \dots, m_{J,t}, \omega_t)\\
    &\quad + \nabla \cG(\bar a_t, m_{1,t}, \dots, m_{J,t}, \vec(\omega_t))^\top
    \Big(\Delta \bar a, \Delta m_1, \dots, \Delta m_J, \vec(\Delta \omega )\Big) \\
    &\quad+ \frac{1}{2} \Big(\Delta \bar a, \Delta m_1, \dots, \Delta m_J,
    \vec(\Delta \omega) \Big)^\top \nabla^2 \cG(\xi)\\
    &\qquad \times \Big(\Delta \bar a, \Delta m_1, \dots, \Delta m_J, \vec(\Delta
    \omega) \Big),\end{split}\]
    where both the gradient and Hessian are with respect to $(\bar a, m_{1}, \dots, m_{J}, \vec(\omega))$.

    Notice the matrix $\nabla^2 \cG(\xi)$ is a finite linear combination of the Hessian matrices from Lemma \ref{lem:bigHessBound}. Recall from Definition \ref{def:ddState} that each state evolution parameter is a derived by a recursion of expectations over the functions $g,f$. Using the boundedness of these functions 
    and the boundedness of the initialization $\|a_1\|_2 = \|\bar a_1\|_2$, we can therefore conclude that 
    each of $\|\bar a_t\|_2$, $\|\bar a_{t+1}\|_2$, $\max_{j \in [J]}\|m_{j,t}\|_2$, $\max_{j \in [J]}\|m_{j,t+1}\|_2$, $\|\omega_t\|_{\op}$, $\|\omega_{t+1}\|_{\op}$ are bounded. Thus, the corresponding values for $\xi$ are also bounded due to being a linear interpolation of the above iterates. 
    
    Then, applying the triangle equality on the operator norm of $\nabla^2 \cG(\xi)$ to each $j$-term and using Assumption \ref{as:main0} \eqref{as:model} (b) to invoke Lemma \ref{lem:bigHessBound}, we have that $\|\nabla^2 \cG(\xi)\|_{\op} \leq C$ for some constant $C > 0$. Moreover, as we previously concluded that $(\Delta \bar a, \Delta m_1, \dots, \Delta m_J,
    \Delta \omega)$ is bounded by $C \epsilon$ for sufficiently small $\epsilon$ uniformly over $t$, each second order term in
    \eqref{eq:taylorG} is of order $\epsilon^2$ uniformly over $t \in [T]$.

    It remains to compute the first order Taylor expansion. Using Lemma \ref{lem:dominated} to interchange
    expectation and derivative, we can calculate,
    \begin{align}
        \nabla_{\bar a} \cG(\bar a, m_1, \dots, m_J, \omega) &= \sum_{j=1}^J
        p_{j} \EE\Big[\nabla_{\bar a} \Psi(m_{j} + \omega^{1/2}G, Y_j,
        \bar a)\Big],\label{eq:gradA}\\
        \nabla_{m_{j}} \cG(\bar a, m_1, \dots, m_J, \omega) &= p_{j}
        \EE\Big[\nabla_{h} \Psi(m_{j} + \omega^{1/2} G, Y_j, \bar
        a)\Big],\label{eq:gradM}\\
        \nabla_{\omega} \cG(\bar a, m_1, \dots, m_J, \omega)  &= \frac{1}{2}
        \sum_{j = 1}^J p_{j} \EE\Big[\nabla_h^2 \Psi(m_{j} +
        \omega^{1/2} G, Y_j, \bar a)\Big],\label{eq:gradOm}
    \end{align}
    where the latter two derivatives follow from Lemma \ref{lem:orderStein}.
    Using \eqref{eq:gradA}, the first order contribution from $\Delta \bar a$ gives,
    \begin{align}
    \langle \nabla_{\bar a} \cG, \Delta \bar a \rangle &= -(1 - \gamma_{0}) \sum_{j=1}^J p_j \Big\langle \EE[\nabla_{a}\Psi(m_{j,t} +
    G^t, Y_j, \bar a_t)], \bar a_t \Big\rangle\\
    &\quad - \gamma_{1} \sum_{j, k = 1}^J p_j p_{k} \Big\langle
    \EE[\nabla_{a}\Psi( m_{j,t} + G^t, Y_j, \bar a_t)], \EE[f(
    m_{k,t} + G^t, Y_{k}, \bar a_t)] \Big\rangle.
    \end{align}
    This produces terms \eqref{eq:aDamp} and \eqref{eq:aDescent}.
    Using \eqref{eq:gradM}, the first order contribution from $\Delta m_j$ gives,
    \begin{align}
    \sum_{j=1}^J \langle \nabla_{m_j}\cG, \Delta m_j \rangle &= -(1 -
    \eta_{0}) \sum_{j=1}^J p_j \Big\langle
    \EE[\nabla_h\Psi(m_{j,t} + G^t, Y_j, \bar a_t)], m_{j,t}
    \Big\rangle\\
    &\quad - \eta_{1} \alpha \sum_{j, k=1}^J p_j \chi_{j,k} p_{k}
    \Big\langle \EE[\nabla_h \Psi( m_{j,t} + G^t, Y_j, \bar a_t)],
    \EE[g(m_{k,t} + G^t, Y_{k}, \bar a_t)] \Big\rangle.
    \end{align}
    This produces terms \eqref{eq:ThDamp} and \eqref{eq:ThSignal}.
    Using \eqref{eq:gradOm}, The first order contribution from $\Delta \omega$ gives,
    \begin{align}
    \langle \nabla_\omega \cG, \Delta \omega \rangle &= -\frac{1}{2}(1 -
    \eta_{0}^2) \sum_{j=1}^J p_j \Big\langle \EE[\nabla_h^2
    \Psi(m_{j,t} + G^t, Y_j, \bar a_t)], \omega_t \Big\rangle\\
    &\quad - \frac{1}{2}\eta_{0}\eta_{1} \sum_{j=1}^J p_j
    \Big\langle \EE[\nabla_h^2 \Psi(m_{j,t} + G^t, Y_j, \bar a_t)],
    \Xi_t\ss{t,t} + \Xi_t\ss{t,t}^\top \Big\rangle\\
    &\quad + \frac{1}{2}\eta_{1}^2 \sum_{j=1}^J p_j
    \Big\langle \EE[\nabla_h^2 \Psi(m_{j,t} + G^t, Y_j, \bar a_t)],
    \alpha \Sigma_t\ss{t,t} + \alpha^2 \sum_{k,k'=1}^J \chi_{k,k'} l_{k,t} l_{k',t}^\top \Big\rangle\\
    &=-\eta_0(1-\eta_0)\sum_{j=1}^J p_j \Big\langle \EE[\nabla_h^2
    \Psi( m_{j,t} + G^t, Y_j, \bar a_t)], \omega_t \Big\rangle\\
    &\quad - \frac{1}{2}\eta_{0}\eta_{1} \sum_{j=1}^J p_j
    \Big\langle \EE[\nabla_h^2 \Psi(m_{j,t} + G^t, Y_j, \bar a_t)],
    \Xi_t\ss{t,t} + \Xi_t\ss{t,t}^\top \Big\rangle\\
    &\quad \+ \frac{1}{2}\eta_{1}^2 \sum_{j=1}^J p_j
    \Big\langle \EE[\nabla_h^2 \Psi(m_{j,t} \+ G^t, Y_j, \bar a_t)],
    \alpha \Sigma_t\ss{t,t} \+  \alpha^2\!\!\!\sum_{k,k'=1}^J\!\!\!\chi_{k,k'} l_{k,t} l_{k',t}^\top \Big\rangle\label{eq:eta1Sq}\\
    &\quad-\frac{1}{2}(1-\eta_0)^2 \sum_{j=1}^J p_j \Big\langle \EE[\nabla_h^2
    \Psi(m_{j,t} + G^t, Y_j, \bar a_t)], \omega_t \Big\rangle\label{eq:eta0Rewrite}\\
    &=-\eta_0(1-\eta_0)\sum_{j=1}^J p_j \Big\langle \EE[\nabla_h^2
    \Psi( m_{j,t} + G^t, Y_j, \bar a_t)], \omega_t \Big\rangle\\
    &\quad - \frac{1}{2}\eta_{0}\eta_{1} \sum_{j=1}^J p_j
    \Big\langle \EE[\nabla_h^2 \Psi( m_{j,t} + G^t, Y_j, \bar a_t)],
    \Xi_t\ss{t,t} + \Xi_t\ss{t,t}^\top \Big\rangle + \delta_t,
    \end{align}
    where $\sup_{t \in [T]} |\delta_t| \leq C \epsilon^2$ for some constant $C>0$ for sufficiently small $\epsilon$.
    The final inequality in the above display follows by rewriting \eqref{eq:eta0Rewrite} using $(1-\eta_0^2) = (1-\eta_0)^2 + 2\eta_0(1-\eta_0)$, and using the inequality $\langle M_1, M_2 \rangle \leq \|M_1\|_F \|M_2\|_F$ to absorb term \eqref{eq:eta1Sq} using the $\eta_1^2$ prefactor. This second argument is allowed by the assumed boundedness of $\|\EE[\nabla^2_h \Psi(m_{j,t} + G^t, Y_j, \bar a_t)]\|_{F}$ (Assumption \ref{as:main0} \eqref{as:model} (b)), the boundedness of $\|\omega_t\|_F$, $\|\Sigma_t\ss{t,t}\|_F$ independent of time $t\in [T]$ and recognizing that $\|l_{k,t} l_{k',t}^\top\|_F \leq \|l_{k,t}\|_F\|l_{k',t}\|_F$ is bounded as $l_{k,t}$ is an expectation over $g$, a bounded function.
    This produces terms \eqref{eq:omDamp} and \eqref{eq:omCross}.

    Combining all first order contributions and bounding the second order
    remainder by the $\epsilon^2$ order remainder $\delta_t$ (say by enlarging its bound by $2C \epsilon^2$) completes the proof.
\end{proof}

\begin{proof}[Proof of Theorem \ref{thm:pureTaylor}]
    First, we invoke Theorem \ref{thm:ddLearn} with the choice of $\eta_0, \gamma_0 = 1$, $\eta_1 = \gamma_1 = \eta$, $g = \nabla_h \Psi$, $f = \nabla_a \Psi$. This gives the Taylor expansion
    \begin{align}
    \testfunc_{t+1} &= \testfunc_t\\
    %% a-descent
    &\quad - \eta \sum_{j,j'=1}^J p_{j}p_{j'} \Big \langle
    \EE[\nabla_{a} \Psi(m_{j,t} + G^t, Y_j, \bar a_t)],
    \EE[\nabla_a \Psi(m_{j',t} + G^t, Y_{j'}, \bar a_t)] \Big \rangle \\
    %% Theta-signal
    &\quad - \eta \alpha \sum_{j,j' = 1}^J
    p_{j}\chi_{j,j'} p_{j'} \Big \langle \EE[\nabla_{h}\Psi(
    m_{j,t} + G^t,Y_j, \bar a_t)], \EE[\nabla_h \Psi( m_{j',t} + G^t, Y_{j'},
    \bar a_t)] \Big \rangle \\
    &\quad - \frac{\eta}{2} \sum_{j=1}^J p_j
    \Big\langle \EE[\nabla_h^2 \Psi(m_{j,t} + G^t, Y_j, \bar a_t)],
    \Xi_t\ss{t,t} + \Xi_t\ss{t,t}^\top \Big\rangle + \epsilon_t,
\end{align}
where $\sup_{t \in [T]} |\epsilon_t| \leq C \eta^2$ for constant $C > 0$ for sufficiently small $\eta > 0$.
Using the notation given in the statement of Theorem \ref{thm:pureTaylor} and simplifying the above equation gives the proof.
\end{proof}

Used as an important step in the proof of Theorem \ref{thm:ddLearn},  we present a high level observation that allows one control the norm of the matrix $\Xi\ss{t,t}$ in the Taylor expansion of Theorem \ref{thm:pureTaylor}.

\begin{remark}\label{rm:TaylorError}Observe, the first term on line \eqref{eq:TaylorError} need not be negative and may increase the test error. Moreover, unrolling $\Xi_t\ss{t,t}$ from Definition \ref{def:ddState}, this term is equal to,
    \[\eta \sum_{j=1}^J p_j \Big\langle \mathbb{E}[\nabla_h^2 \Psi(m_{j,t} + G^t, Y_j, \bar a_t)], \alpha \sum_{k=1}^J m_{k,1} l_{k,t}^\top - \eta \sum_{s=1}^{t-1} \Big( \alpha \Sigma_t\ss{s,t} + \alpha^2 \sum_{a,b=1}^J \chi_{a,b} l_{a,s} l_{b,t}^\top \Big) \Big\rangle.\]
    It is not clear the above summation is bounded independent of time $t$. Thankfully, the assumption that $\Sigma_t\ss{t,t}$ and $\Omega_t\ss{t,t}$ have bounded operator norms (and thus bounded Frobenius norms) from Assumption \ref{as:phi} allows for such control. Indeed, using the limits from Lemma \ref{lem:ddState} and 
    \[\left\|\frac{\theta_t^\top \tilde \theta_t}{d}\right\|_F \leq \left\|\frac{\theta_t}{\sqrt{d}}\right\|_F \left\|\frac{\tilde \theta_t}{\sqrt{d}}\right\|_F,\]
    we have that $\|\Xi_t\ss{t,t}\|_F$, and thus the term in \eqref{eq:TaylorError} by Assumption \ref{as:main0} \eqref{as:model}, is bounded in the $n,d$ limit.
\end{remark}

\subsection{The Benefits of Damping For Test Error Critical Points}\label{sec:fixedPoints}

Further tuning hyperparameters in \eqref{eq:dd} can give favorable global behavior as well.

\begin{theorem}\label{thm:global}
Let $g = \nabla_h \Psi$ and $f = \nabla_a \Psi$ in Algorithm \eqref{eq:dd}, with $\eta, \gamma \in (0,1)$, consider the one dimensional hyperparameter subspace over $c \in \RR_+$,
\[\eta_{0} = 1-\eta,\quad \eta_1 = c\eta,\quad \gamma_0 = 1,\quad \gamma_1 = \gamma,\label{eq:hpSubspace}.\]
Any fixed point $(m_1^*, \dots, m_J^*, \Omega^*, \bar a^*)$ of Definition \ref{def:ddState}, must satisfy the equations
\begin{align}
m_{j}^* &= - c \alpha \left(\sum_{k=1}^J \chi_{j,k} p_k \EE[g(m_{k}^* + G^*, Y_k, \bar a^*)]\right)\label{eq:fixM}\\
0 &= - \sum_{j=1}^J p_j \EE[f( m_{j}^* + G^*, Y_j, \bar a^*)]\label{eq:fixA},
\end{align}
where $Y_j \sim \PP_j$ and $G^* \sim \cN(0, \Omega^*)$. Moreover, for any deterministic vector $v \in \RR^d$ perpendicular to $(\mu_j)_{j \in [J]}$, any iterate $\theta$ produced by the above fixed point must satisfy $\lim_{d \to \infty} d^{-1}v^\top \theta = 0$.
\end{theorem}
\begin{proof}
Assume the fixed point given in the statement of the theorem.
Using the equations of $m_{j,t+1}$ and $\bar a_{t+1}$ for the given fixed point, we have that $(m_j^*)_{j \in [J]}$, $\Omega^*$, $\bar a^*$ satisfy the following system,
\begin{align}
m_{j}^* &= \eta_0 m_j^* - \eta_1 \alpha \sum_{k=1}^J \chi_{j,k} p_k \EE[g(m_{k}^* + G^*, Y_k, \bar a^*)]\\
\bar a^* &= \gamma_0\bar a^* - \gamma_1 \sum_{j=1}^J p_j \EE[f(m_{j}^* + G^*, Y_j, \bar a^*)],
\end{align}
where $Y_j \sim \PP_j$ and $G^* \sim \cN(0, \Omega^*)$. Rearranging these equations and simplifying gives the fixed point condition in the theorem. Note, by adding a dummy signal vector $\mu_{J+1} = v, p_{J+1} = 0$ where $v \perp \mu_j$ for $j \in [J]$ into distribution \eqref{eq:dist}, we can also conclude the second statement in the theorem, i.e. that $m_{J+1}^* = 0$, as $\chi_{J+1, k} = 0$ for all $k \in [J]$ and $\eta_0 \in (0,1)$. 
\end{proof}

Let $(\theta^*, a^*)$ be a critical point for the test error, assume that $H^* = \sum_{j=1}^J p_j \EE[\nabla^2_h
\Psi(\mu_j^\top \theta^* / d + Z^\top \theta^*, Y_j, a)]$ is invertible and isotropic when $Z \sim \cN(0, 1/d)$. Then, there exists a $c \in \RR_+$ where \eqref{eq:fixM} and \eqref{eq:fixA} are equivalent to projected critical point conditions on $(\theta^*, a^*)$.

A test error critical point satisfies $\nabla_{a} \EE_{\check x, \check y}\left[\cL(
\cM_{a^*}(\check x^\top \theta^*), \check y) \right] = 0$ and $\nabla_{\theta} \EE_{\check x, \check y}\left[\cL(
\cM_{a^*}(\check x^\top \theta), \check y)\right] = 0$ with $(\check x, \check y)$ drawn from distribution \eqref{eq:dist}. 
Expanding \eqref{eq:dist} in terms of the mixture on $\mu_1, \dots, \mu_J$, interchanging expectation and derivative by Lemma \ref{lem:dominated}, we equivalently write this system as,
\begin{align}
    \label{eq:aCritical}\sum_{j=1}^J p_j \EE[\nabla_{a} \Psi(\mu_j^\top \theta^* / d + Z^\top
    \theta^*, Y_j, a^*)] &= 0,\\
    \label{eq:thetaCritical} -\sum_{j=1}^J p_j \mu_j
    \EE[\nabla_h\Psi(\mu_j^\top \theta^* / d + Z^\top \theta^*, Y_j, a^*)] &= \theta H^*,
\end{align}
where $Z \sim \cN(0, \Id_d/d)$. As $H^*$ is invertible, for any deterministic vector $v\in \RR^d$, \eqref{eq:thetaCritical} implies that,
\[\frac{v^\top\theta^*}{d} = -\Big(\sum_{j=1}^J p_j \frac{v^\top \mu_j}{d}
\EE[\nabla_h \Psi(\mu_j^\top \theta^* / d + Z^\top \theta^*, Y_j, a))]
\Big)(H^*)^{-1}.\label{eq:criticalM}\]
There are two cases of interest: (1) If $v = \mu_j$ for some $j \in [J]$, assuming that $\lim_{d \to \infty} \mu_j^\top \theta^* / d = m_{j,\theta^*}$, $\lim_{d \to \infty} (\theta^*)^\top \theta^* /d = \Omega_{\theta^*}$, $\lim_{n,d \to \infty} a^* = \bar a_{\theta^*}$, and replacing $(H^*)^{-1} = (c^*)^{-1} \Id$ for some constant $c^* > 0$ by the isotropic assumption, asymptotically solving \eqref{eq:criticalM} is equivalent to,
\[m_{j,\theta} = - c^* \sum_{k=1}^J \chi_{j,k} p_k \EE[\nabla_h \Psi(
m_{k,\theta^*} + G, Y_j, \bar a_{\theta^*}))],\quad G \sim \cN(0,\Omega_{\theta^*})\label{eq:popProj};\]
(2) If $v$ is perpendicular to the span of $(\mu_j)_{j \in [J]}$, then assuming that $
\lim_{d \to \infty}\frac{v^\top \theta}{d} = m^{\perp}$ almost surely, we have the asymptotic equation
\[m^{\perp} = 0\label{eq:perpProj}.\]
Next, under the same limits $a, \mu_j^\top \theta / d$ and $\theta^\top \theta / d$, \eqref{eq:aCritical} is asymptotically equivalent to,
\[\sum_{j=1}^J p_j \EE[\nabla_{\bar a} \Psi(m_{j, \theta^*} + G, Y_j, \bar a_{\theta^*})] = 0,
\quad G \sim \cN\left(0, \Omega_{\theta^*} \right).\label{eq:aCriticalTwo}\]

Comparing \eqref{eq:popProj}, \eqref{eq:perpProj} and \eqref{eq:aCriticalTwo} with $m_{j, \theta^*} = m_j^*$ for each $j \in [J]$, $\Omega_{\theta^*} = \Omega^*$ and $\bar a_{\theta^*} = \bar a^*$, and selecting $c = (c^* \alpha)^{-1}$, the statement of Theorem \ref{thm:global} is identical to the set of equations \eqref{eq:aCritical} and \eqref{eq:popProj}. 

\begin{remark}[Adaptive Hessian estimation]
    Note, in the case where $H^*$ is not isotropic, one naturally would incorporate the Hessian $H_t$ into a DD algorithm, the data-based estimate of 
    \[\hat H_t = \frac{1}{n} \sum_{i=1}^n \nabla^2_h \Psi(h_{t,i}, y_i, a_t),\] is
    already calculated for the correction term in both pure DD and is expected to be an asymptotically exact estimate of $H_t$. But, this requires a time-inhomogeneous choice of $g = g_t$; such an extension to our results is possible but is left to future work.
\end{remark}

\section{Deferred Applications}

\subsection{Faithful Test Error Tracking In Signal-less Learning}\label{sec:signalless}

Recall this problem from the introduction, consider
$n$ \iid data points $x_i \sim \cN(0,\Id_d/d)$ and $y_i = 0$, generating the data $(X,y)$. Let  $\cM_\theta(x) = x^\top \theta$ with $\theta \in \RR^d$ initialized at $\theta_1 \sim \cN(0, \Id_d)$ and trained under loss $\cL(\hat y,y) = \frac{1}{2}(\hat y - y)^2$ (rescaled for convenience), realized by distribution \eqref{eq:dist} with $J = 1$ and $\mu_1$ being the zero vector.

We run the following algorithms:

\begin{center}
\begin{tabular}{c c}
    \textbf{GD} & \textbf{(Damped) Pure DD} \\[1ex]
    $\begin{aligned}
        h_t &= X\theta_t \\
        \tilde{\theta}_t &= X^\top h_t \\
        \theta_{t+1} &= \theta_t - \eta \tilde{\theta}_t
    \end{aligned}$
    &
    $\begin{aligned}
        h_t &= X \theta_t + \eta_1 \sum_{s=1}^{t-1} \eta_0^{(t-1)-s} \hat{h}_s \\
        \tilde{\theta}_t &= X^\top h_t - \alpha \theta_t \\
        \theta_{t+1} &= \eta_0\theta_t - \eta_1 \tilde{\theta}_t
    \end{aligned}$
\end{tabular}
\end{center}

Figure \ref{fig:noSignalReg} plots the train and test errors for GD and pure DD ($\eta_0 = 1, \eta_1 = \eta$). We immediately see the train-test disconnect in GD is corrected in both DD algorithms. Note, for pure DD, the train-test error explodes.
This is predicted from the variance inflation terms in $\Omega_t$ from Definition \ref{def:ddState} and occurs because this problem has no signal. Specifically, we can use Remark \ref{rem:ddstate} to conclude the
variance recursion (and thus the test error recursion) takes form,
\[\Omega_{t+1}\ss{t+1, t+1} = \eta_{0}^2 \Omega_t\ss{t,t} - \eta_{1}
\eta_{0} (\Xi_t\ss{t,t} + \Xi_t\ss{t,t}^\top) + \eta_{1}^2\alpha \Omega_t\ss{t,t},\]
which may or may not be a contraction depending on the choices of $\eta_0, \eta_1$. 

\begin{figure}
    \centering
    \includegraphics[width=.6\linewidth]{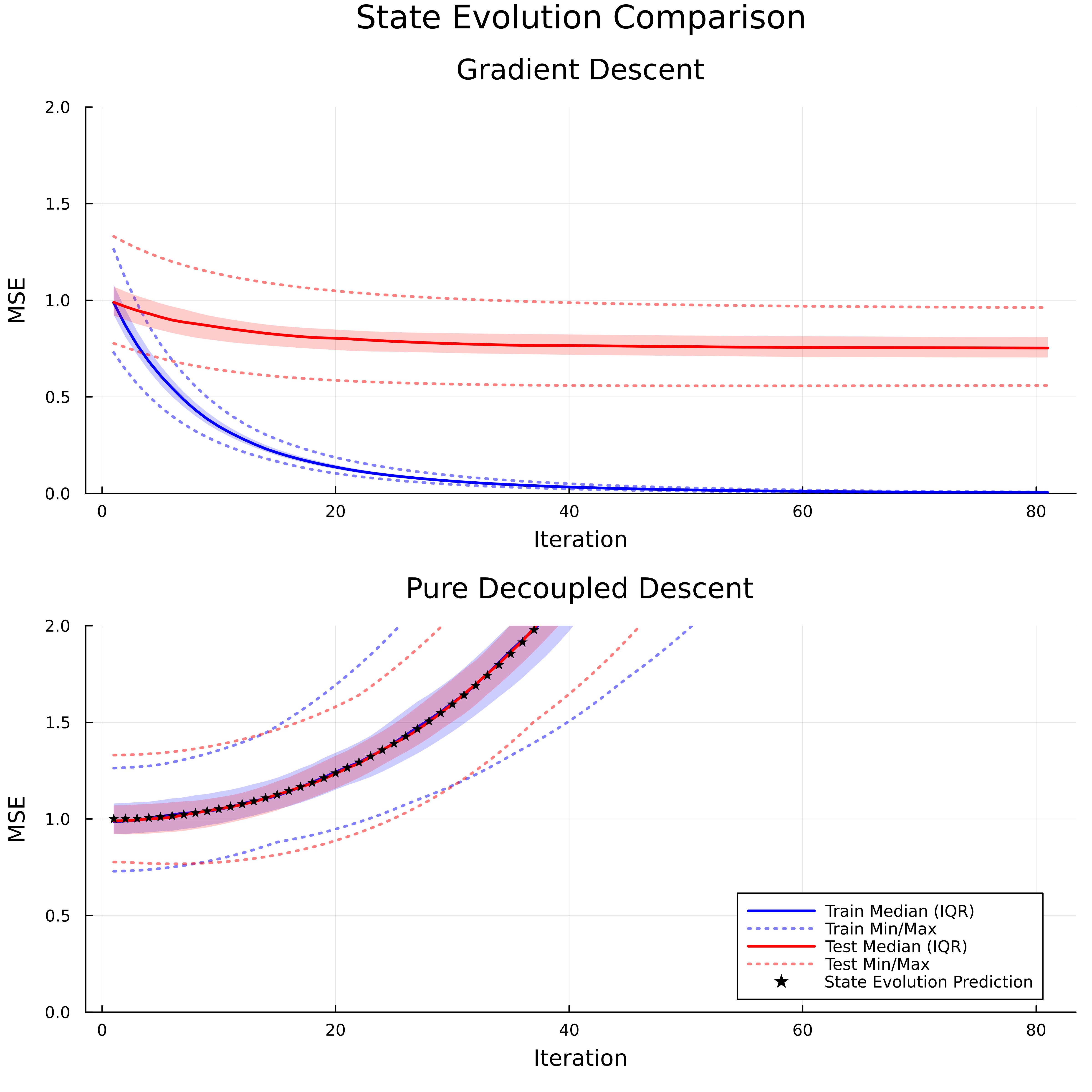}
    \caption{Summary stats for 100 signal-less regression runs ($n=200, d=800$): GD (top) vs. DD (bottom) with $\eta=0.05$. Blue/red denote train/test error; solid lines are medians, shaded areas are interquartile ranges, and dotted lines show min/max. Notice, by design, the trajectories of the train and test error are identical for pure DD while they significantly diverge for GD.}
    \label{fig:noSignalReg}
\end{figure}

This naturally leads into relying on the zero-cost validation feature of DD. We consider the following one dimensional subspace dependent on $\eta$, 
\[\eta_0 = 1-\eta,\qquad  \eta_1 = c \eta\label{eq:sweepC}.\]
We then sweep the value of $c$ over some specified range for a fixed data realization. The results of many individual replications of this experiment are given in Figure \ref{fig:noSignalSweep}. Again, for each simulation the train-test identity holds perfectly as expected and the trajectory of the test error is tracked by the state evolution prediction. Moreover, since
each of these replications used the same dataset, we can optimize over the hyperparameter $c$ by selecting whichever algorithm gets the optimal training (and thus test) error.

\begin{figure}
    \centering
    \includegraphics[width = .7\linewidth]{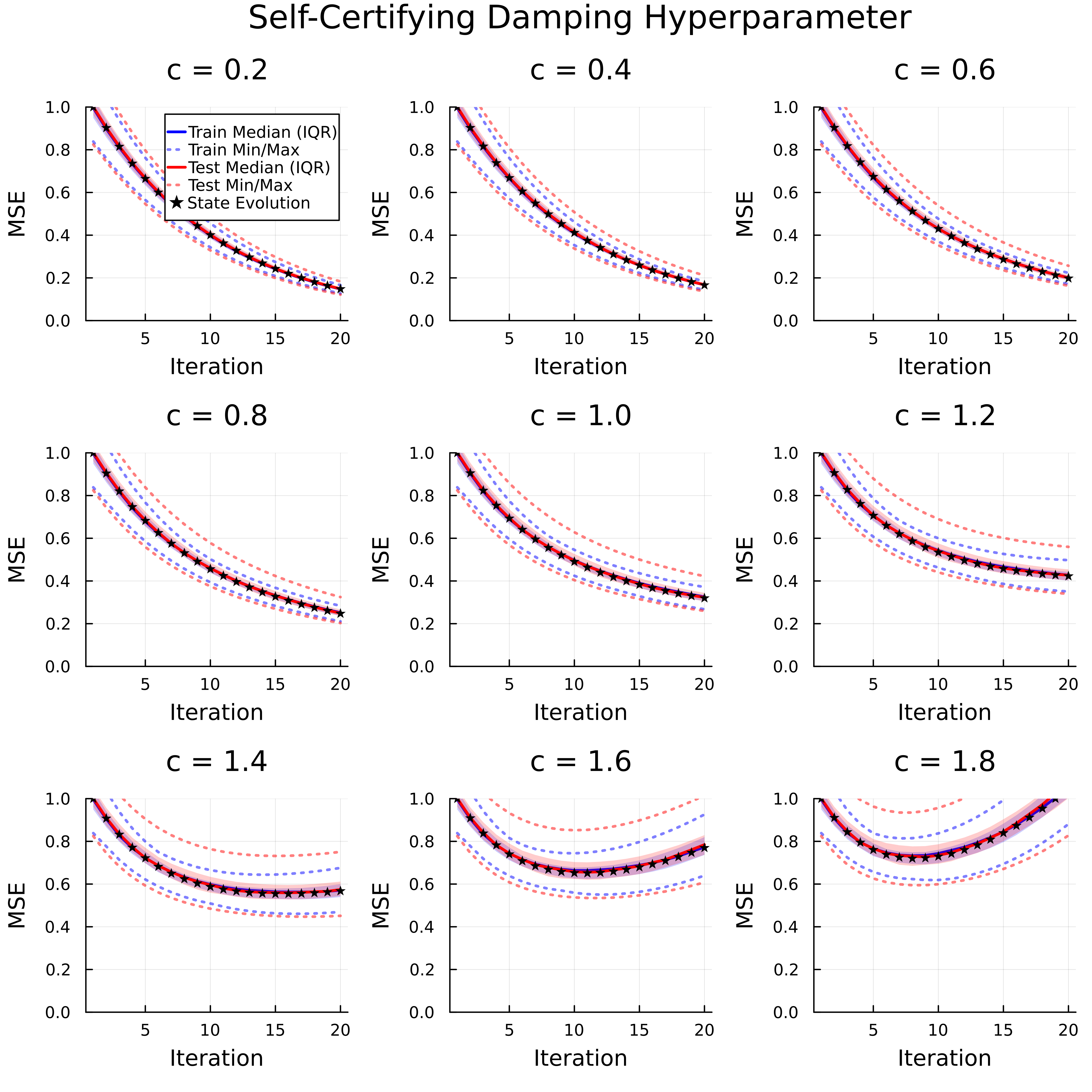}
     \caption{An example of hyperparameter tuning with DD over 100 individual replications, in each figure is the train and test error for damped pure DD algorithm in the signal-less regression problem with $n = 800, d = 800$, $\eta = 0.05$ and we sweep $c$ over the sub-space given in \eqref{eq:sweepC}. Blue lines refer to train error and red lines refer to test error; the solid line is the median error, the shaded region is the range from the $25$-th to $75$-th quartiles and the dotted lines are the minimum and maximum error runs. This represents an application where multiple instances of DD are run on the same data set and, because the train-test identity, we get a zero-cost validation method to selecting good hyperparameters that aid in training.}\label{fig:noSignalSweep}
\end{figure}

\subsection{The XOR Training Iterations}\label{sec:detailXor}

\begin{center}
\begin{tabular}{c c}
    \textbf{GD} & \textbf{Pure DD} \\[1ex]
    $\begin{aligned}
        (h^1_t, h^2_t) &= X (\theta^1_t, \theta^2_t)\\
        \hat{y}_t &= \sigma(a_t (h^1_t \odot h^2_t))\\
        \hat{h}^1_t &= (\hat{y}_t - y) \odot (a_t h^2_t)\\
        \hat{h}^2_t &= (\hat{y}_t - y) \odot (a_t h^1_t)\\
        \theta^1_{t+1} &= \theta^1_t - \eta X^\top \hat{h}^1_t\\
        \theta^2_{t+1} &= \theta^2_t - \eta X^\top \hat{h}^2_t\\
        a_{t+1} &= a_t - \eta (\hat{y}_t - y)^\top (h^1_t \odot h^2_t) / n
    \end{aligned}$
    &
    $\begin{aligned}
        (h_t^1, h_t^2) &= X (\theta^1_t, \theta^2_t) + \eta \sum_{s = 1}^{t-1} (\hat{h}_{s}^1, \hat{h}_{s}^2),\\
        \hat{y}_t &= \sigma(a_t (h^1_t \odot h^2_t)),\\
        \hat{h}^1_t &= (\hat{y}_t - y) \odot (a_t h^2_t)\\ 
        \hat{h}^2_t &= (\hat{y} - y) \odot (a_t h^1_t),\\
        H_t\!&=\!\frac{1}{n}\!\sum_{i=1}^n\!\Bigg(\!\hat{y}_{t,i}(1\!-\!\hat{y}_{t,i}) a_t^2\!
        \begin{bmatrix}(h^2_{t,i})^2 & h^1_{t,i}h^2_{t,i} \\ h^1_{t,i}h^2_{t,i} &
        (h^1_{t,i})^2 \end{bmatrix}\!\\
        &\quad+\!(\hat{y}_{t,i}\!-\!y_i)\! \begin{bmatrix} 0 & a_t \\ a_t & 0\end{bmatrix}\!\Bigg)\\
        \tilde{\theta}_{t} &= X^\top (\hat{h}_t^1, \hat{h}_t^2) - \alpha (\theta^1_{t}, \theta^2_{t}) H_t\\ 
        (\theta^1_{t+1}, \theta^2_{t+1}) &= (\theta^1_t, \theta^2_t) - \eta \tilde{\theta}_t\\
        a_{t+1} &= a_t - \eta (\hat{y}_t - y)^\top (h^1_t \odot h^2_t) / n
    \end{aligned}$
\end{tabular}
\end{center}

\begin{center}
\begin{minipage}{\textwidth}
\centering
\textbf{Damped Variant Of Pure DD (Fixed $a_t = 1$)} 

$\begin{aligned} 
        (h_t^1, h_t^2) &= X (\theta^1_t, \theta^2_t) + \eta_1 \sum_{s = 1}^{t-1} \eta_0^{(t-1)-s} (\hat{h}_{s}^1, \hat{h}_{s}^2),\\
        \hat{y}_t &= \sigma(h^1_t \odot h^2_t),\\
        \hat{h}^1_t &= (\hat{y}_t - y) \odot (h^2_t)\\ 
        \hat{h}^2_t &= (\hat{y}_t - y) \odot (h^1_t),\\
        H_t\!&=\!\frac{1}{n}\!\sum_{i=1}^n\!\Bigg(\!\hat{y}_{t,i}(1\!-\!\hat{y}_{t,i})\!
        \begin{bmatrix}(h^2_{t,i})^2 & h^1_{t,i}h^2_{t,i} \\ h^1_{t,i}h^2_{t,i} &
        (h^1_{t,i})^2 \end{bmatrix}\!+\!(\hat{y}_{t,i}\!-\!y_i)\! \begin{bmatrix} 0 & 1 \\ 1 & 0\end{bmatrix}\!\Bigg)\\
        \tilde{\theta}_{t} &= X^\top (\hat{h}_t^1, \hat{h}_t^2) - \alpha (\theta^1_{t}, \theta^2_{t}) H_t\\ 
        (\theta^1_{t+1}, \theta^2_{t+1}) &= \eta_0(\theta^1_t, \theta^2_t) - \eta_1 \tilde{\theta}_t\\
\end{aligned}$
\end{minipage}
\end{center}

\subsection{A Run-time Analysis For The MNIST Problem}\label{sec:mnistTime}

Figure \ref{fig:timeGDvDD} tracks the ``wall clock'' time per epoch (i.e. a single full batch) for our width-nine hidden layer model at both the per epoch level and for total training time. Although our implementation of DD suffers from a
$\approx L\times$ constant factor overhead, we believe the practical benefits of DD (including use of the full training set) may outweigh this computational penalty. A variant of DD mimicking SGD may further reduce this overhead (Section \ref{sec:conclusion}; future directions).

\begin{figure}
    \centering
    \includegraphics[width=0.75\linewidth]{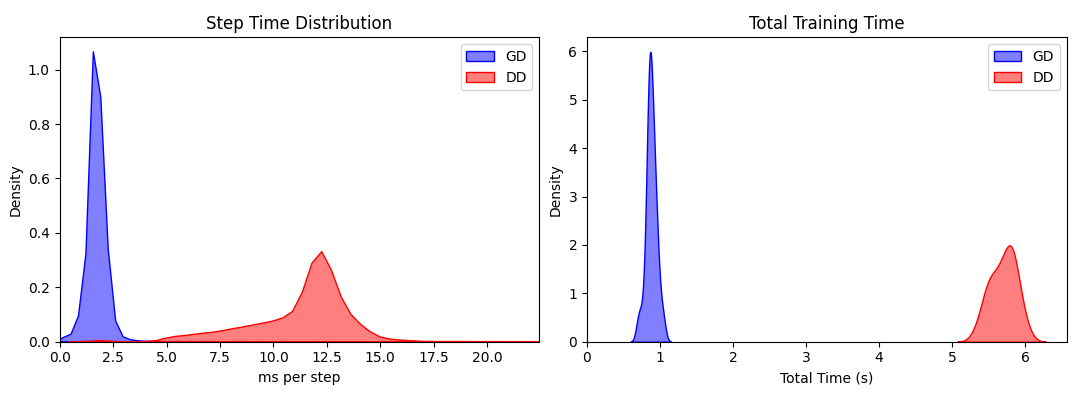}
    \caption{Empirical density estimates for runtime on a width-9 two-layer neural network (25 replications, 500 epochs). Left: Distribution of per-epoch clock time (ms) for GD vs. DD. Right: Total training time (s) per replication for the MNIST zeros/eights problem.}
    \label{fig:timeGDvDD}
\end{figure}

\subsection{Implementation of MNIST Problem With Nine Hidden
Layers}\label{sec:implementation}

Below is our implementation of the MNIST zeros/eights problem for a width-nine hidden layer model.

\begin{Verbatim}[
    commandchars=\\\{\},
    numbers=left,        % enables line numbers
    numbersep=10pt,      % space between numbers and code
    breaklines=true,     % enables automatic line wrapping
    breakindent=20pt,    % indentation for wrapped lines
    fontsize=\footnotesize
]
\PY{k+kn}{import}\PY{+w}{ }\PY{n+nn}{time}\PY{o}{,}\PY{+w}{ }\PY{n+nn}{math}\PY{o}{,}\PY{+w}{ }\PY{n+nn}{torch}\PY{o}{,}\PY{+w}{ }\PY{n+nn}{numpy}\PY{+w}{ }\PY{k}{as}\PY{+w}{ }\PY{n+nn}{np}
\PY{k+kn}{import}\PY{+w}{ }\PY{n+nn}{torch}\PY{n+nn}{.}\PY{n+nn}{nn}\PY{+w}{ }\PY{k}{as}\PY{+w}{ }\PY{n+nn}{nn}\PY{o}{,}\PY{+w}{ }\PY{n+nn}{torch}\PY{n+nn}{.}\PY{n+nn}{nn}\PY{n+nn}{.}\PY{n+nn}{functional}\PY{+w}{ }\PY{k}{as}\PY{+w}{ }\PY{n+nn}{F}
\PY{k+kn}{from}\PY{+w}{ }\PY{n+nn}{torchvision}\PY{+w}{ }\PY{k+kn}{import} \PY{n}{datasets}\PY{p}{,} \PY{n}{transforms}

\PY{k}{def}\PY{+w}{ }\PY{n+nf}{get\PYZus{}data}\PY{p}{(}\PY{n}{train}\PY{p}{)}\PY{p}{:}
    \PY{n}{d} \PY{o}{=} \PY{n}{datasets}\PY{o}{.}\PY{n}{MNIST}\PY{p}{(}\PY{l+s+s2}{\PYZdq{}}\PY{l+s+s2}{./data}\PY{l+s+s2}{\PYZdq{}}\PY{p}{,} \PY{n}{train}\PY{o}{=}\PY{n}{train}\PY{p}{,} \PY{n}{download}\PY{o}{=}\PY{k+kc}{True}\PY{p}{,} \PY{n}{transform}\PY{o}{=}\PY{n}{transforms}\PY{o}{.}\PY{n}{ToTensor}\PY{p}{(}\PY{p}{)}\PY{p}{)}
    \PY{n}{m} \PY{o}{=} \PY{p}{(}\PY{n}{d}\PY{o}{.}\PY{n}{targets} \PY{o}{==} \PY{l+m+mi}{0}\PY{p}{)} \PY{o}{|} \PY{p}{(}\PY{n}{d}\PY{o}{.}\PY{n}{targets} \PY{o}{==} \PY{l+m+mi}{8}\PY{p}{)}
    \PY{k}{return} \PY{n}{d}\PY{o}{.}\PY{n}{data}\PY{p}{[}\PY{n}{m}\PY{p}{]}\PY{o}{.}\PY{n}{float}\PY{p}{(}\PY{p}{)} \PY{o}{/} \PY{l+m+mf}{255.0}\PY{p}{,} \PY{p}{(}\PY{n}{d}\PY{o}{.}\PY{n}{targets}\PY{p}{[}\PY{n}{m}\PY{p}{]} \PY{o}{==} \PY{l+m+mi}{8}\PY{p}{)}\PY{o}{.}\PY{n}{float}\PY{p}{(}\PY{p}{)}

\PY{n}{x\PYZus{}tr}\PY{p}{,} \PY{n}{y\PYZus{}tr} \PY{o}{=} \PY{n}{get\PYZus{}data}\PY{p}{(}\PY{k+kc}{True}\PY{p}{)}
\PY{n}{x\PYZus{}test}\PY{p}{,} \PY{n}{y\PYZus{}test} \PY{o}{=} \PY{n}{get\PYZus{}data}\PY{p}{(}\PY{k+kc}{False}\PY{p}{)}

\PY{n}{N}\PY{p}{,} \PY{n}{T}\PY{p}{,} \PY{n}{H}\PY{p}{,} \PY{n}{lr}\PY{p}{,} \PY{n}{lam}\PY{p}{,} \PY{n}{reps} \PY{o}{=} \PY{l+m+mi}{800}\PY{p}{,} \PY{l+m+mi}{500}\PY{p}{,} \PY{l+m+mi}{9}\PY{p}{,} \PY{l+m+mi}{1}\PY{p}{,} \PY{l+m+mi}{30}\PY{p}{,} \PY{l+m+mi}{20}
\PY{n}{device} \PY{o}{=} \PY{l+s+s2}{\PYZdq{}}\PY{l+s+s2}{mps}\PY{l+s+s2}{\PYZdq{}} \PY{k}{if} \PY{n}{torch}\PY{o}{.}\PY{n}{backends}\PY{o}{.}\PY{n}{mps}\PY{o}{.}\PY{n}{is\PYZus{}available}\PY{p}{(}\PY{p}{)} \PY{k}{else} \PY{l+s+s2}{\PYZdq{}}\PY{l+s+s2}{cpu}\PY{l+s+s2}{\PYZdq{}}
\PY{n}{x\PYZus{}test}\PY{p}{,} \PY{n}{y\PYZus{}test} \PY{o}{=} \PY{n}{x\PYZus{}test}\PY{o}{.}\PY{n}{unsqueeze}\PY{p}{(}\PY{l+m+mi}{1}\PY{p}{)}\PY{o}{.}\PY{n}{to}\PY{p}{(}\PY{n}{device}\PY{p}{)}\PY{p}{,} \PY{n}{y\PYZus{}test}\PY{o}{.}\PY{n}{to}\PY{p}{(}\PY{n}{device}\PY{p}{)}
\PY{n}{alpha} \PY{o}{=} \PY{n}{N} \PY{o}{/} \PY{l+m+mi}{784}

\PY{k}{def}\PY{+w}{ }\PY{n+nf}{inject\PYZus{}noise}\PY{p}{(}\PY{n}{x}\PY{p}{,} \PY{n}{lam}\PY{p}{)}\PY{p}{:}
    \PY{n}{D} \PY{o}{=} \PY{n}{x}\PY{o}{.}\PY{n}{shape}\PY{p}{[}\PY{o}{\PYZhy{}}\PY{l+m+mi}{1}\PY{p}{]} \PY{o}{*} \PY{n}{x}\PY{o}{.}\PY{n}{shape}\PY{p}{[}\PY{o}{\PYZhy{}}\PY{l+m+mi}{2}\PY{p}{]}
    \PY{n}{p} \PY{o}{=} \PY{n}{torch}\PY{o}{.}\PY{n}{rand\PYZus{}like}\PY{p}{(}\PY{n}{x}\PY{p}{)}
    \PY{n}{z} \PY{o}{=} \PY{n}{torch}\PY{o}{.}\PY{n}{where}\PY{p}{(}\PY{n}{p} \PY{o}{\PYZlt{}} \PY{l+m+mf}{0.25}\PY{p}{,} \PY{o}{\PYZhy{}}\PY{n}{math}\PY{o}{.}\PY{n}{sqrt}\PY{p}{(}\PY{l+m+mi}{2}\PY{o}{/}\PY{n}{D}\PY{p}{)}\PY{p}{,} \PY{n}{torch}\PY{o}{.}\PY{n}{where}\PY{p}{(}\PY{n}{p} \PY{o}{\PYZlt{}} \PY{l+m+mf}{0.5}\PY{p}{,} \PY{n}{math}\PY{o}{.}\PY{n}{sqrt}\PY{p}{(}\PY{l+m+mi}{2}\PY{o}{/}\PY{n}{D}\PY{p}{)}\PY{p}{,} \PY{l+m+mf}{0.0}\PY{p}{)}\PY{p}{)}
    \PY{k}{return} \PY{p}{(}\PY{n}{lam} \PY{o}{/} \PY{n}{D}\PY{p}{)} \PY{o}{*} \PY{n}{x} \PY{o}{+} \PY{n}{z}\PY{o}{.}\PY{n}{to}\PY{p}{(}\PY{n}{x}\PY{o}{.}\PY{n}{device}\PY{p}{)}

\PY{k}{class}\PY{+w}{ }\PY{n+nc}{DecoupledMLP}\PY{p}{(}\PY{n}{nn}\PY{o}{.}\PY{n}{Module}\PY{p}{)}\PY{p}{:}
    \PY{k}{def}\PY{+w}{ }\PY{n+nf+fm}{\PYZus{}\PYZus{}init\PYZus{}\PYZus{}}\PY{p}{(}\PY{n+nb+bp}{self}\PY{p}{)}\PY{p}{:}
        \PY{n+nb}{super}\PY{p}{(}\PY{p}{)}\PY{o}{.}\PY{n+nf+fm}{\PYZus{}\PYZus{}init\PYZus{}\PYZus{}}\PY{p}{(}\PY{p}{)}
        \PY{n+nb+bp}{self}\PY{o}{.}\PY{n}{fc1}\PY{p}{,} \PY{n+nb+bp}{self}\PY{o}{.}\PY{n}{fc2} \PY{o}{=} \PY{n}{nn}\PY{o}{.}\PY{n}{Linear}\PY{p}{(}\PY{l+m+mi}{784}\PY{p}{,} \PY{n}{H}\PY{p}{,} \PY{n}{bias}\PY{o}{=}\PY{k+kc}{False}\PY{p}{)}\PY{p}{,} \PY{n}{nn}\PY{o}{.}\PY{n}{Linear}\PY{p}{(}\PY{n}{H}\PY{p}{,} \PY{l+m+mi}{1}\PY{p}{,} \PY{n}{bias}\PY{o}{=}\PY{k+kc}{False}\PY{p}{)}
    \PY{k}{def}\PY{+w}{ }\PY{n+nf}{forward}\PY{p}{(}\PY{n+nb+bp}{self}\PY{p}{,} \PY{n}{x}\PY{p}{,} \PY{n}{mem}\PY{o}{=}\PY{l+m+mi}{0}\PY{p}{)}\PY{p}{:}
        \PY{n}{z} \PY{o}{=} \PY{n+nb+bp}{self}\PY{o}{.}\PY{n}{fc1}\PY{p}{(}\PY{n}{x}\PY{o}{.}\PY{n}{view}\PY{p}{(}\PY{n}{x}\PY{o}{.}\PY{n}{size}\PY{p}{(}\PY{l+m+mi}{0}\PY{p}{)}\PY{p}{,} \PY{o}{\PYZhy{}}\PY{l+m+mi}{1}\PY{p}{)}\PY{p}{)} \PY{o}{+} \PY{n}{mem}
        \PY{k}{return} \PY{n+nb+bp}{self}\PY{o}{.}\PY{n}{fc2}\PY{p}{(}\PY{n}{torch}\PY{o}{.}\PY{n}{tanh}\PY{p}{(}\PY{n}{z}\PY{p}{)}\PY{p}{)}\PY{o}{.}\PY{n}{squeeze}\PY{p}{(}\PY{l+m+mi}{1}\PY{p}{)}\PY{p}{,} \PY{n}{z}

\PY{n}{dd\PYZus{}err} \PY{o}{=} \PY{n}{np}\PY{o}{.}\PY{n}{zeros}\PY{p}{(}\PY{p}{(}\PY{l+m+mi}{2}\PY{p}{,} \PY{n}{reps}\PY{p}{,} \PY{n}{T}\PY{p}{)}\PY{p}{)}
\PY{n}{dd\PYZus{}time} \PY{o}{=} \PY{p}{[}\PY{p}{]}

\PY{k}{for} \PY{n}{r} \PY{o+ow}{in} \PY{n+nb}{range}\PY{p}{(}\PY{n}{reps}\PY{p}{)}\PY{p}{:}
    \PY{n}{idx} \PY{o}{=} \PY{n}{torch}\PY{o}{.}\PY{n}{randperm}\PY{p}{(}\PY{n+nb}{len}\PY{p}{(}\PY{n}{x\PYZus{}tr}\PY{p}{)}\PY{p}{)}\PY{p}{[}\PY{p}{:}\PY{n}{N}\PY{p}{]}
    \PY{n}{x\PYZus{}n}\PY{p}{,} \PY{n}{y\PYZus{}n} \PY{o}{=} \PY{n}{inject\PYZus{}noise}\PY{p}{(}\PY{n}{x\PYZus{}tr}\PY{p}{[}\PY{n}{idx}\PY{p}{]}\PY{o}{.}\PY{n}{unsqueeze}\PY{p}{(}\PY{l+m+mi}{1}\PY{p}{)}\PY{o}{.}\PY{n}{to}\PY{p}{(}\PY{n}{device}\PY{p}{)}\PY{p}{,} \PY{n}{lam}\PY{p}{)}\PY{p}{,} \PY{n}{y\PYZus{}tr}\PY{p}{[}\PY{n}{idx}\PY{p}{]}\PY{o}{.}\PY{n}{to}\PY{p}{(}\PY{n}{device}\PY{p}{)}
    \PY{n}{x\PYZus{}te\PYZus{}n} \PY{o}{=} \PY{n}{inject\PYZus{}noise}\PY{p}{(}\PY{n}{x\PYZus{}test}\PY{p}{,} \PY{n}{lam}\PY{p}{)}

    \PY{n}{m\PYZus{}dd}\PY{p}{,} \PY{n}{mem} \PY{o}{=} \PY{n}{DecoupledMLP}\PY{p}{(}\PY{p}{)}\PY{o}{.}\PY{n}{to}\PY{p}{(}\PY{n}{device}\PY{p}{)}\PY{p}{,} \PY{n}{torch}\PY{o}{.}\PY{n}{zeros}\PY{p}{(}\PY{n}{N}\PY{p}{,} \PY{n}{H}\PY{p}{,} \PY{n}{device}\PY{o}{=}\PY{n}{device}\PY{p}{)}
    \PY{n}{t0} \PY{o}{=} \PY{n}{time}\PY{o}{.}\PY{n}{time}\PY{p}{(}\PY{p}{)}
    \PY{k}{for} \PY{n}{t} \PY{o+ow}{in} \PY{n+nb}{range}\PY{p}{(}\PY{n}{T}\PY{p}{)}\PY{p}{:}
        \PY{n}{logits}\PY{p}{,} \PY{n}{z} \PY{o}{=} \PY{n}{m\PYZus{}dd}\PY{p}{(}\PY{n}{x\PYZus{}n}\PY{p}{,} \PY{n}{mem}\PY{p}{)}
        \PY{n}{loss} \PY{o}{=} \PY{n}{F}\PY{o}{.}\PY{n}{binary\PYZus{}cross\PYZus{}entropy\PYZus{}with\PYZus{}logits}\PY{p}{(}\PY{n}{logits}\PY{p}{,} \PY{n}{y\PYZus{}n}\PY{p}{)}
        \PY{n}{h\PYZus{}hat} \PY{o}{=} \PY{n}{torch}\PY{o}{.}\PY{n}{autograd}\PY{o}{.}\PY{n}{grad}\PY{p}{(}\PY{n}{loss}\PY{p}{,} \PY{n}{z}\PY{p}{,} \PY{n}{create\PYZus{}graph}\PY{o}{=}\PY{k+kc}{True}\PY{p}{)}\PY{p}{[}\PY{l+m+mi}{0}\PY{p}{]}
        
        \PY{n}{J\PYZus{}bar} \PY{o}{=} \PY{n}{torch}\PY{o}{.}\PY{n}{stack}\PY{p}{(}\PY{p}{[}\PY{n}{torch}\PY{o}{.}\PY{n}{autograd}\PY{o}{.}\PY{n}{grad}\PY{p}{(}\PY{n}{h\PYZus{}hat}\PY{p}{[}\PY{p}{:}\PY{p}{,} \PY{n}{i}\PY{p}{]}\PY{o}{.}\PY{n}{sum}\PY{p}{(}\PY{p}{)}\PY{p}{,} \PY{n}{z}\PY{p}{,} \PY{n}{retain\PYZus{}graph}\PY{o}{=}\PY{k+kc}{True}\PY{p}{)}\PY{p}{[}\PY{l+m+mi}{0}\PY{p}{]}\PY{o}{.}\PY{n}{mean}\PY{p}{(}\PY{n}{dim}\PY{o}{=}\PY{l+m+mi}{0}\PY{p}{)} \PY{k}{for} \PY{n}{i} \PY{o+ow}{in} \PY{n+nb}{range}\PY{p}{(}\PY{n}{H}\PY{p}{)}\PY{p}{]}\PY{p}{)}

        \PY{k}{with} \PY{n}{torch}\PY{o}{.}\PY{n}{no\PYZus{}grad}\PY{p}{(}\PY{p}{)}\PY{p}{:}
            \PY{n}{m\PYZus{}dd}\PY{o}{.}\PY{n}{fc2}\PY{o}{.}\PY{n}{weight} \PY{o}{\PYZhy{}}\PY{o}{=} \PY{n}{lr} \PY{o}{*} \PY{n}{torch}\PY{o}{.}\PY{n}{autograd}\PY{o}{.}\PY{n}{grad}\PY{p}{(}\PY{n}{loss}\PY{p}{,} \PY{n}{m\PYZus{}dd}\PY{o}{.}\PY{n}{fc2}\PY{o}{.}\PY{n}{weight}\PY{p}{,} \PY{n}{retain\PYZus{}graph}\PY{o}{=}\PY{k+kc}{True}\PY{p}{)}\PY{p}{[}\PY{l+m+mi}{0}\PY{p}{]}
            \PY{n}{theta} \PY{o}{=} \PY{n}{m\PYZus{}dd}\PY{o}{.}\PY{n}{fc1}\PY{o}{.}\PY{n}{weight}
            \PY{n}{m\PYZus{}dd}\PY{o}{.}\PY{n}{fc1}\PY{o}{.}\PY{n}{weight}\PY{o}{.}\PY{n}{copy\PYZus{}}\PY{p}{(}\PY{n}{theta} \PY{o}{\PYZhy{}} \PY{n}{lr} \PY{o}{*} \PY{p}{(}\PY{n}{h\PYZus{}hat}\PY{o}{.}\PY{n}{t}\PY{p}{(}\PY{p}{)} \PY{o}{@} \PY{n}{x\PYZus{}n}\PY{o}{.}\PY{n}{view}\PY{p}{(}\PY{n}{N}\PY{p}{,} \PY{o}{\PYZhy{}}\PY{l+m+mi}{1}\PY{p}{)} \PY{o}{\PYZhy{}} \PY{n}{alpha} \PY{o}{*} \PY{p}{(}\PY{n}{J\PYZus{}bar} \PY{o}{@} \PY{n}{theta}\PY{p}{)}\PY{p}{)}\PY{p}{)}
            \PY{n}{mem} \PY{o}{+}\PY{o}{=} \PY{n}{lr} \PY{o}{*} \PY{n}{h\PYZus{}hat}\PY{o}{.}\PY{n}{detach}\PY{p}{(}\PY{p}{)}
            
        \PY{n}{dd\PYZus{}err}\PY{p}{[}\PY{l+m+mi}{0}\PY{p}{,} \PY{n}{r}\PY{p}{,} \PY{n}{t}\PY{p}{]} \PY{o}{=} \PY{n}{loss}\PY{o}{.}\PY{n}{item}\PY{p}{(}\PY{p}{)}
        \PY{k}{with} \PY{n}{torch}\PY{o}{.}\PY{n}{no\PYZus{}grad}\PY{p}{(}\PY{p}{)}\PY{p}{:} \PY{n}{dd\PYZus{}err}\PY{p}{[}\PY{l+m+mi}{1}\PY{p}{,} \PY{n}{r}\PY{p}{,} \PY{n}{t}\PY{p}{]} \PY{o}{=} \PY{n}{F}\PY{o}{.}\PY{n}{binary\PYZus{}cross\PYZus{}entropy\PYZus{}with\PYZus{}logits}\PY{p}{(}\PY{n}{m\PYZus{}dd}\PY{p}{(}\PY{n}{x\PYZus{}te\PYZus{}n}\PY{p}{)}\PY{p}{[}\PY{l+m+mi}{0}\PY{p}{]}\PY{p}{,} \PY{n}{y\PYZus{}test}\PY{p}{)}\PY{o}{.}\PY{n}{item}\PY{p}{(}\PY{p}{)}
    \PY{n}{dd\PYZus{}time}\PY{o}{.}\PY{n}{append}\PY{p}{(}\PY{n}{time}\PY{o}{.}\PY{n}{time}\PY{p}{(}\PY{p}{)} \PY{o}{\PYZhy{}} \PY{n}{t0}\PY{p}{)}
\end{Verbatim}

\subsection{Whitening Technique For CIFAR-10 Embedding Vectors} \label{sec:embedding}

We considered the following
pre-processing steps, which approach the conditions of Assumption \ref{as:main0}
to varying degrees: 
\begin{itemize}
    \item \textit{Vanilla (no whitening):} Features are rescaled by a
    $1/\sqrt{d}$ factor (assuming the embedding vectors have Gaussian entries).
    This leaves correlations between coordinates of the embedding vectors which
    our current DD algorithm does not account for.
    \item \textit{ZCA (whitening) on the train data:} We compute the empirical
    covariance matrix $\Sigma_{\text{train}} = \frac{1}{N-1} X_{\text{train}}^T
    X_{\text{train}}$ exclusively on the training set and compute the eigen-decomposition
    $\Sigma_{\text{train}} = V \Lambda V^T$. We construct the Zero-phase Component
    Analysis (ZCA) \cite{kessy2018} matrix $W_{\text{train}} =
    V(\Lambda)^{-1/2}V^T$, which is then applied to both the train and test data
    using $X' = X W_{\text{train}}$ and $\check X' = \check X W_{\text{train}}$. Finally, we apply the $1/\sqrt{d}$ normalization from the Vanilla bullet.
    \item \textit{Joint ZCA (whitening) on train and test data:} To prevent drift in the
    covariance between the train and test sets, we concatenate both data sets
    to form a joint covariance matrix $\Sigma_{\text{joint}}$. The resulting
    $W_{\text{joint}}$ is applied to all data with,
    \[\begin{bmatrix} X' \\ \check X' \end{bmatrix} = \begin{bmatrix} X
    W_{\text{joint}} \\ \check X W_{\text{joint}} \end{bmatrix}.\]
    Finally, we apply the $1/\sqrt{d}$ normalization from the Vanilla bullet.
\end{itemize}

\subsection{Data Collection and Attribution}

This work utilized both the \textbf{MNIST} and \textbf{CIFAR-10} datasets, standard benchmarks in the high-dimensional statistics and machine learning literature. 

\begin{itemize}
    \item \textbf{MNIST:} A collection of $70,000$ handwritten digit images \cite{lecun1998gradient}. The dataset is available under the \textbf{Creative Commons Attribution-Share Alike 3.0} license.
    \item \textbf{CIFAR-10:} Consists of $60,000$ $32 \times 32$ color images in 10 classes \cite{krizhevsky2009learning}. The dataset is used in accordance with the researchers' terms at the University of Toronto. 
\end{itemize}

\end{document}